\documentclass[reqno,11pt]{amsart}
\usepackage{setspace,tikz,xcolor,mathrsfs,listings,multicol,amssymb}
\usepackage{rotating}
\usepackage[vcentermath]{youngtab}
\usepackage[margin=1in,includefoot,footskip=30pt]{geometry}
\usepackage{enumerate}
\usepackage{booktabs}
\usepackage[centertableaux]{ytableau}
\usepackage[all,cmtip]{xy}
\usetikzlibrary{arrows,matrix}
\tikzset{tab/.style={matrix of math nodes,column sep=-.35, row sep=-.35,text height=7pt,text width=7pt,align=center,inner sep=2,font=\footnotesize}}

\usepackage[colorlinks=true, pdfstartview=FitV, linkcolor=blue, citecolor=blue, urlcolor=blue]{hyperref}

\newcommand{\fsl}{\mathfrak{sl}}
\newcommand{\gl}{\mathfrak{gl}}

\newcommand{\Sym}{\Sigma}
\newcommand{\field}{\mathbf{k}}
\newcommand{\GL}{\operatorname{GL}}
\newcommand{\Or}{\operatorname{O}}

\newcommand{\natrepr}{\mathbf{V}}

\newcommand{\iso}{\cong}

\newcommand{\abs}[1]{\left\lvert #1 \right\rvert}

\DeclareRobustCommand{\stirling}{\genfrac\{\}{0pt}{}}
\newcommand{\divides}{\mid}

\newcommand{\ZZ}{\mathbb{Z}}
\newcommand{\QQ}{\mathbb{Q}}
\newcommand{\RR}{\mathbb{R}}
\newcommand{\CC}{\mathbb{C}}

\newcommand{\mcA}{\mathcal{A}}
\newcommand{\mcB}{\mathcal{B}}
\newcommand{\mcC}{\mathcal{C}}
\newcommand{\mcE}{\mathcal{E}}
\newcommand{\mcG}{\mathcal{G}}
\newcommand{\mcI}{\mathcal{I}}
\newcommand{\mcL}{\mathcal{L}}
\newcommand{\mcM}{\mathcal{M}}
\newcommand{\mcP}{\mathcal{P}}
\newcommand{\mcQ}{\mathcal{Q}}
\newcommand{\mcR}{\mathcal{R}}
\newcommand{\mcS}{\mathcal{S}}
\newcommand{\mcT}{\mathcal{T}}
\newcommand{\mcU}{\mathcal{U}}

\DeclareMathOperator{\End}{End}
\DeclareMathOperator{\Span}{span}

\definecolor{darkred}{rgb}{0.7,0,0} 
\newcommand{\defn}[1]{{\color{darkred}\emph{#1}}} 

\definecolor{UQgold}{RGB}{196, 158, 54} 
\definecolor{UQpurple}{RGB}{73, 7, 94} 
\definecolor{UMNgold}{RGB}{255,200,46} 
\definecolor{UMNmaroon}{RGB}{106,0,50} 
\definecolor{OCUenji}{RGB}{153,0,51} 
\definecolor{OCUsapphire}{RGB}{0,51,102} 

\usepackage{listings}
\lstdefinelanguage{Sage}[]{Python}
{morekeywords={False,sage,True},sensitive=true}
\definecolor{dblackcolor}{rgb}{0.0,0.0,0.0}
\definecolor{dbluecolor}{rgb}{0.01,0.02,0.7}
\definecolor{dgreencolor}{rgb}{0.2,0.4,0.0}
\definecolor{dgraycolor}{rgb}{0.30,0.3,0.30}

\theoremstyle{plain}
\newtheorem{thm}{Theorem}[section]
\newtheorem{lemma}[thm]{Lemma}
\newtheorem{conj}[thm]{Conjecture}
\newtheorem{prop}[thm]{Proposition}
\newtheorem{cor}[thm]{Corollary}
\newtheorem{problem}[thm]{Problem}
\theoremstyle{definition}
\newtheorem{dfn}[thm]{Definition}
\theoremstyle{remark}
\newtheorem{ex}[thm]{Example}
\newtheorem{remark}[thm]{Remark}
\numberwithin{equation}{section}

\usepackage[colorinlistoftodos]{todonotes}

\title[Cellular partition subalgebras]{Cellular subalgebras of the partition algebra}

\author[T.~Scrimshaw]{Travis Scrimshaw}
\address[T.~Scrimshaw]{Department of Mathematics, Hokkaido University, 5 Ch\=ome Kita 8 J\=onishi, Kita Ward, Sapporo, Hokkaid\=o 060--0808, Japan}
\email{tcscrims@gmail.com}
\urladdr{https://tscrim.github.io}

\keywords{cellular algebra, centralizer algebra, complex reflection group}
\subjclass[2010]{05E10, 16G99, 05E16, 05A19}

\begin{document}

\begin{abstract}
We describe various diagram algebras and their representation theory using cellular algebras of Graham and Lehrer and the decomposition into half diagrams.
In particular, we show the diagram algebras surveyed here are all cellular algebras and parameterize their cell modules.
We give a new construction to build new cellular algebras from a general cellular algebra and subalgebras of the rook Brauer algebra that we call the cellular wreath product.
\end{abstract}

\maketitle

\tableofcontents

\section{Introduction}

Schur--Weyl duality is a classical result in representation theory due to Issai Schur~\cite{Schur27} and publicized by Hermann Weyl~\cite{Weyl39} that relates the irreducible representations of the symmetric group $\Sym_n$ with that of the general linear group $\GL_m(\CC)$.
If we take the natural representation $\natrepr := \CC^m$ of $\GL_m(\CC)$, then there is a natural $\Sym_n$-action on $\natrepr^{\otimes n}$ that permutes the factors.
Since $\GL_m(\CC)$ acts diagonally on $\natrepr^{\otimes n}$, this commutes with the $\Sym_n$-action.
Schur--Weyl duality is the statement that the image of the representation afforded by the $\Sym_n$-action is everything that commutes with the $\GL_m(\CC)$-action and vice versa.
Consequently, for $m \geq n$, we have that
\[
\natrepr^{\otimes n} \iso \bigoplus_{\lambda \vdash n} S^{\lambda} \boxtimes V(\lambda)
\]
as $(\Sym_n \times \GL_m(\CC))$-representations, where $S^{\lambda}$ (resp.~$V(\lambda)$) is the irreducible $\Sym_n$ (resp.~$\GL_m(\CC)$) representation indexed by the partition $\lambda$.
In particular, the decomposition is multiplicity free.
For more information, we refer the reader to~\cite{EGHLSVY11,FH91,Howe95}.
(Additionally, the actual endomorphism algebra $S(m,n) = \End_{\Sym_n} \natrepr^{\otimes n}$ is the well-studied Schur algebra; see, \textit{e.g.},~\cite{Green07}.)

If we have a subgroup $H \subseteq \GL_m(\CC)$, then $\natrepr$ is an $H$-representation by restriction.
Subsequently, the endomorphism algebra $\End_H \natrepr^{\otimes n}$ could potentially be larger (and usually is).
As a first case, we take $H = \Sym_m$.
To describe the endomorphism algebra, we will use the partition algebra $\mcP_n(\beta)$ introduced independently by Jones~\cite{Jones94} and Martin~\cite{Martin94}, which is connected to the study of the Potts model from statistical mechanics (see also, \textit{e.g.},~\cite{Martin91,Martin96,Martin00}).
The partition algebra has been well-studied from different perspectives and applications; see, \textit{e.g.},~\cite{BDVO15,HR05,HL05,HJ20,GL96,MR98,Xi99}.
In particular, it is the prototypical diagram algebra whose diagrams are a graphical representations of set partitions of $\{1,\dotsc,n\} \sqcup \{1',\dotsc,n'\}$.

This leads to a combinatorial description of the centralizer algebras for groups $\Sym_m \subseteq H \subseteq \GL_m(\CC)$.
One well-studied group is $H = \Or_n(\CC)$ the orthogonal group.
The subalgebra of the partition algebra is named the Brauer algebra after Brauer for his initial work on it, in particular for proving the corresponding Schur--Weyl duality~\cite{Brauer37}.
This also (essentially) covers the case of the symplectic group $H = Sp_{2n}(\CC)$ (see, \textit{e.g.},~\cite[Thm.~B6.3]{BR99}).
Another case is $H = G(r,p,m)$ by Tanabe~\cite{Tanabe97}, where $G(r,p,m)$ is the infinite family of complex reflection groups in the Shepard--Todd classification~\cite{ST54}.
In particular $G(2,1,m)$ is the Weyl group of $\Or_{2m+1}(\CC)$ consisting of signed permutations, with the centralizer algebra studied by Orellana~\cite{Orellana05,Orellana07}.
Another case that has been studied is the centralizer algebra of $G(r,1,m)$ for $r \geq n$~\cite{AMM21,Kosuda00,Kosuda06,OSSZ21}.

On the other hand, we can construct various subalgebras of $\mcP_n(\beta)$ by imposing combinatorial restrictions on the diagrams and study the corresponding algebras.
One of the most famous subalgebras is the Temperley--Lieb algebra first introduced by Temperley and Lieb in~\cite{TL71} in the context of the Potts model, which was rediscovered by Jones in his work on subfactors (\textit{e.g.},~\cite{Jones83,Jones85,Jones87}) and linked to knot theory.
The Temperley--Lieb algebra is a well-studied algebra with an extensive literature (see, \textit{e.g.},~\cite{BdGN01,BM05,BW89,GdlHJ89,GW93,Kauffman87,Martin91,PRdGN02,PRZ06,RSA14}) and generalizations (such as~\cite{BSA18,DG22,MS94,ILZ18,MW00,LRH20}).
Some facts include an analog of Schur--Weyl duality with the quantum group $U_q(\gl_2)$ and that it is a quotient of the Hecke algebra of $\Sym_m$~\cite{Jones87}.
Furthermore, it is (essentially) equivalent to the planar partition algebra (see, \textit{e.g.},~\cite{Jones94,HR05} and Problem~\ref{prob:TL_planar_zero}).
Other subalgebras that have been considered, such as the half partition algebra~\cite{HR05}, quasi-partition algebra~\cite{DO14}, rook algebra~\cite{Munn57,Solomon02}, and planar rook algebra~\cite{FHH09}.
For most of these cases, an analog of Schur--Weyl duality has been constructed.

Let us digress slightly by looking at the diagram algebras where blocks have size at most $2$.
For these cases, a new phenomenon can appear.
We can have a second independent parameter $\gamma$ that counts the number of simply-connected interior components when composing diagrams and $\beta$ counts the number of loops.
Specifically, these algebras are the rook Brauer algebra~\cite{DEG17,DG22II,HdM14,MM14}, Motzkin algebra~\cite{BH14,DEG17}, and partial Temperley--Lieb algebra~\cite{DG22}.
When $\gamma = \beta$, these are naturally a subalgebra of $\mcP_n(\beta)$, and they are isomorphic for all $\gamma \neq 0$ by rescaling the diagrams (see, \textit{e.g.},~\cite[Lemma~7.3]{DG22II}).
Hence, generically we can consider them to be diagrammatic subalgebras of $\mcP_n(\beta)$ up to rescaling.
Nothing seems to be known for the degenerate cases of $\beta = 0$ (resp.~$\beta = 0$) and $\gamma \neq 0$ (resp.~$\gamma \neq 0$).

In this note, we will survey a number of results on the representation theory of the partition algebra $\mcP_n(\beta)$ and some of its diagram subalgebras.
Our main tool will be using is the theory of cellular algebras introduced by Graham and Lehrer~\cite{GL96}.
We will then use the decomposition of the partition algebra of Xi~\cite{Xi99}, which is particularly amenable to combinatorics and diagrammatic subalgebra.
Thus, the representation theory of the algebras in large part reduces the study to that of the symmetric group.
In particular, we will be following the approach of half diagrams, given for the specific case of the Temperley--Lieb algebra under the name of (quotients of) link modules (see, \textit{e.g.},~\cite{GL96,Martin90,Martin91,RSA14}.
We remark that many of the algebras presented here (see Table~\ref{table:algebras} for many examples) are known to be cellular~\cite{BH14,DG22,GL96,Xi99}.
For nearly all of the other cases considered, they are possibly known by experts to be cellular even if they have not been explicitly written down.
More specifically, the centralizer algebra of $G(r,1,m)$ for $r \geq 2$ and $m \geq n$, the (planar) rook algebra, and half partition algebra can be shown to be cellular as a consequence of the Proposition~\ref{prop:subcellular} (which is immediate from the definition of a cellular algebra) and the partition algebra being cellular.
Yet not every case follows quite so easily.
In particular, to show that the (planar) quasi-partition algebra is cellular (Theorem~\ref{thm:quasi_cellular} and Theorem~\ref{thm:quasi_partition_cellular}) requires more technical analysis, which the author believes to be new.

Additionally, we remark that the techniques used in other papers (\textit{e.g.}~\cite{DO14,HJ20,MM14,OSSZ21}) are no less important than the approach taken here; on the contrary, they often provide more refined descriptions and can be better adapted to addressing other questions such as characters.
Yet, the diagrammatic approach with cellular algebras taken here allows us to describe the cell modules (and, in principle, the simple modules) and the algebra action naturally in terms of tableau more uniformly in terms of inducing representations of Young subgroups of $\Sym_n$.
On the other hand, the approach taken here allows us to study these diagram algebras over arbitrary fields $\field$ (not necessarily $\CC$) and discuss the semisimplicity of a number of these algebras over arbitrary fields and describe their irreducible modules.
However, we do not perform this analysis here as each such algebra not previously treated deserves its own detailed paper.
It is the hope of the author that this paper helps facilitate easier translations between the combinatorial and algebraic information.

\begin{table}
\begin{center}
\begin{tabular}{ccccl}
\toprule
name & notation & planar & propagating & block sizes \\
\midrule
partition & $\mcP_n(\beta)$ & no & no & any size \\
half partition & $\mcP_{n-1/2}(\beta)$ & no & $n \sim n'$ & any size \\
quasi-partition & $\mcQ\mcP_n(\beta)$ & no & no & $> 1$ \\
$G(r,d,m)$-centralizer & $\mcG_n^{(r,d,m)}(\beta)$ & no & no & based on $r$ \\
uniform block & $\mcU_k$ & no & yes & top equals bottom \\
rook Brauer & $\mcR\mcB_n(\beta)$ & no & no & $\leq 2$ \\
rook & $\mcR_n(\beta)$ & no & if size $2$ & $\leq 2$ \\
Brauer & $\mcB_n(\beta)$ & no & no & $= 2$ \\
symmetric group & $\field [\Sym_n]$ & no & yes & $= 2$ \\
planar partition & $\mcP\mcP_n(\beta)$ & yes & no & any size \\
Temperley--Lieb & $\mcT\mcL_n(\beta)$ & yes & no & $= 2$ \\
Motzkin & $\mcM_n(\beta)$ & yes & no & $\leq 2$ \\
partial TL & $\mcP\mcT\mcL_n(\beta)$ & yes & no & $\leq 2$ and balanced \\
planar quasi-partition & $\mcP\mcQ\mcP_n(\beta)$ & yes & no & $> 1$ \\
planar rook & $\mcP\mcR_n(\beta)$ & yes & if size $2$ & $\leq 2$ \\
planar $r$-color & $\mcP\mcC_{r,n}(\beta)$ & yes & no & based on $r$ \\
\bottomrule
\end{tabular}
\end{center}
\caption{A list of some subalgebras of the partition algebra and a summary of the conditions imposed on the indexing diagrams.}
\label{table:algebras}
\end{table}

To illustrate this, we consider the uniform block (permutation) algebra $\mcU_n$ studied in~\cite{OSSZ21}.
This is spanned by the set diagrams that represent permutations of blocks of equal size.
We show that a number of the constructions given in~\cite{OSSZ21} can be described as coming from the cellular structure of the partition algebra $\mcP_n(\beta)$ restricted to $\mcU_n$.
As a consequence, we are able to show that $\mcU_n$ is a cellular algebra and describe its simple representations (as far as we understand the representations of the symmetric group) over an arbitrary field (in~\cite{OSSZ21}, they only considered it as a $\CC$-algebra).
For another example, let us consider the description using multiset-valued tableaux, which are tableau filled with multisets with a given total ordering under the usual standard condition,\footnote{
The standard condition here means we restrict to set-valued tableaux.
Morevoer, the (multi)set-valued tableaux here are different than the those in the K-theory of the Grassmannian considered in~\cite{HS20,LP07,PP16}, which have different semistandard conditions, weights, and generating functions.}
for the $\mcP_n(\beta)$ irreducible representations.
This has appeared (sometimes implicitly) under a few different names in many different papers~\cite{BH19,BHH17,HJ20,MR98,OSSZ21,OZ21}.
We can see this as a rephrasing of the decomposition of~\cite{Xi99}, which breaks the diagrams into an upper part, lower part, and a middle permutation part using the propagating blocks, using a simple generalization of the blocks-with-defects approach (see Section~\ref{sec:partition_algebra}).
In both of these cases, we are just inducing a Specht module of a Young subgroup $\Sym_k$ and the exactly $k$ defects in the half diagrams corresponds to the multiset that we use to fill the tableau.
Moreover, the Robinson--Schensted--Knuth (RSK) algorithm~\cite{Robinson35,Schensted61,Knuth70} provides the link between tableaux and permutations (of the multisets) as for the symmetric group, with this link used explicitly in~\cite{COSSZ20}.
The restriction of this decomposition was given for the Temperley--Lieb algebra and Brauer algebra in~\cite{GL96}, which appeared as early as the work of Brown~\cite{Brown56}.
For more detailed treatments of (most of) the general approach in this paper, see~\cite{BH14} for the Motzkin algebra and~\cite{DG22} for the partial Temperley--Lieb algebra. 

Let us discuss the new results in this paper.
As previously mentioned, we show a number of diagram algebras are cellular and describe their cell modules.
Many of these diagram algebras are likely already known to experts but not written down, but we believe the (planar) quasi-partition algebra is new (Theorem~\ref{thm:quasi_cellular} and Theorem~\ref{thm:quasi_partition_cellular}).
A combinatorially interesting fact is that the cell modules of the planar quasi-partition algebra are given by the triangle Riordan numbers.
We provide a number of new formulas for the dimensions of the cell modules of the (planar) $G(r,1,m)$-centralizer algebra, including an appearance of the Fuss--Catalan numbers (Proposition~\ref{prop:planar_color_zero_defect_dim}); see Section~\ref{sec:planar_even} and Section~\ref{sec:planar_color}.
Another novel result in this paper (Theorem~\ref{thm:wreath_product}) is a general construction to construct new cellular algebras from a general cellular algebra and subalgebras of the rook Brauer algebra.
We call this the wreath product of cellular algebras.
This is the common generalization of the papers~\cite{RX04,ZC06} and also yields another proof of the cellularity of $\field [G(r,1,m)]$ (assuming the $r = 1$ case for the symmetric group).
We expect this can be extended to having the ``base'' be an algebra with Hecke type relations (\textit{e.g.}, BMW algebras), giving another proof of the cellularity~\cite{GL96} of the Ariki--Koike algebra~\cite{AK94}.

In the process of trying to describe the various algebras and their irreducible modules, we came across a number of questions that we have included for the interested reader to pursue.
These include corner cases where the behavior seems to differ at an algebraic level, but perhaps not within their representation theory (Problem~\ref{prob:two_param_rook_brauer_semisimple}, Problem~\ref{prob:TL_planar_zero}, Problem~\ref{prob:two_param_motzkin}).
Other examples include combinatorial questions that might lead to interesting relations (\textit{e.g.}, Problem~\ref{prob:irrep_dim_CE}).
There could also be new diagram-type algebras constructed by mixing different constraints and understood using the techniques in this paper.
For example, a blob algebra~\cite{MS94} (see also Section~\ref{sec:blob_algebra}) version of the Motzkin algebra, or putting other cellular algebras on the leftmost strands (which could be seen as a subalgebra of the wreath product).

Next, we mention a number of known Schur--Weyl duality statements in Table~\ref{table:schur_weyl}, although this is likely not exhaustive of those involving subalgebras of the partition algebra.
Let us mention some additional cases not previously discussed.
The first is a Schur--Weyl duality for the Burau representations, both reduced and unreduced, of the braid group $\mathbf{B}_m$ given in~\cite{DG21}.
In this case, the centralizer algebra is given by the rook algebra $\mcR_n([m]_q)$, where $[m]_q = \frac{q^m-1}{q-1}$ are the $q$-analogs of $m$.
Next is the so-called tangle algebra $\mcT_n(1)$ studied in~\cite{BE18}, which we conjecture is equivalent to the planar quasi-partition algebra (Conjecture~\ref{conj:tangle_pqp}).
The centralizer of the twin group $TW_m$ on the natural (type $A$) Hecke algebra representation, which can be considered as the Burau representation, was shown to be the rook Brauer algebra (see Section~\ref{sec:rook_brauer}) in~\cite{DG22II}.
The general Lie superalgebra $\gl(1|1)$ has a Schur--Weyl duality with the planar rook algebra~\cite{BM13}.
The last is the partition algebra Schur--Weyl duality has been extended to the non-semisimple setting over general commutative rings in~\cite{BDM22} with the kernel of the symmetric group action giving a cellular basis in~\cite{BDM22II}.
Similar and some more general results using different techniques were shown in~\cite{Donkin22}.

\begin{table}
\begin{center}
\begin{tabular}{ccc}
\toprule
tensor $T$-action & module $V$ & module $G$-action\\
\midrule
$\CC [\Sym_n]$ & natural $\natrepr \iso \CC^m$ & $\GL_m(\CC)$\\
$\mcG_n^{(r,p,m)}(m)$ & natural $\natrepr \iso \CC^m$ & $\CC [G(r,p,m)]$ \\
$\mcP_{n-1/2}(m)$ & natural $\CC [\Sym_m]$-module & $\CC [\Sym_{m-1}]$ \\
$\mcQ\mcP_n(m)$ & Specht $S^{(n-1.1)}$ & $\CC [\Sym_m]$ \\
$\mcR_n(m)$ & $V(1) \oplus V(0)$ & $\GL_m(\CC)$ \\
$\mcR_n([m]_q)$ & unreduced Burau & $\CC [\mathbf{B}_m]$ \\
$\mcR\mcB_n(m+1)$ & $V(1) \oplus V(0)$ & $\Or_m(\CC)$ \\
$\mcR\mcB_n(m, \beta')$ & Burau $\CC^m$ & $\CC [TW_m]$ \\
$\mcT\mcL_n\bigl(\pm(q+q^{-1})\bigr)$ & natural $V(1) \iso \CC^2$ & $U_q(\fsl_2)$ \\
$\mcM_n\bigl(1\pm(q+q^{-1})\bigr)$ & adjoint $V(1) \oplus V(0)$ & $U_q(\fsl_2)$ \\
$\mcP\mcR_n(\beta)$ & natural $\natrepr^{\otimes n}$ & $U\bigl(\gl(1|1)\bigr)$ \\
$\mcP\mcT\mcL_n\bigl(1\pm(q+q^{-1})\bigr)$ & $V(1) \oplus V(0)$ & $U_q(\gl_2)$ \\
$\mcT_n(1)$ & $V(2)$ & $U(\fsl_2)$ \\
\bottomrule
\end{tabular}
\end{center}
\caption{Some known Schur--Weyl duality statements on the module $V^{\otimes n}$, where $V$ is a (left) $G$-module.}
\label{table:schur_weyl}
\end{table}

We conclude by mentioning some additional references for the interested reader.
There is the survey on combinatorial representation theory by Barcelo and Ram~\cite{BR99} that discusses questions on diagram algebras (among others) from a different perspective.
Another perspective on diagram algebras was considered by Cox, Martin, Parker, and Xi in~\cite{CMPX06}, where they studied the diagram algebras as a tower of algebras through an abstract framework.
Furthermore, they introduce a generalization of the Temperley--Lieb algebra by assigning arrows to each line, which also generalizes the blob algebra.
By considering Schur--Weyl duality with $\GL_m(\CC)$ on $\natrepr^{\otimes n} \otimes (\natrepr^*)^{\otimes k}$, we obtain the walled Brauer algebra, which is a subalgebra of the Brauer algebra originating in the work of Tureav~\cite{Turaev89} and Koike~\cite{Koike89} with its own rich literature; see \textit{e.g.},~\cite{BCHLLS94,BJSH21,CDVDM08,Halverson96,JK20} and references therein.
In particular, Brundan and Stroppel in~\cite{BS12} make a link with their previous work on studying Khovanov's arc algebra and generalizations~\cite{BS10,BS11,BS11III,BS12IV}.
In turn, this is connected with another diagrammatic algebra known as the Khovanov--Lauda--Rouquier (KLR) algebra or quiver Hecke algebra~\cite{KL09,Rouquier08} that has a vast literature too large to even begin to list here, but we simply mention the generalized Schur--Weyl duality functors of Kang, Kashiwara, and Kim~\cite{KKK15,KKK18} (see~\cite{KKOP20,KKOP21} for some recent results related to these functors).
A broader framework of sandwiched cellular algebras was introduced by Tubbenhauer and coauthors~\cite{TM21,Tubbenhauer22,TV21} as a way to generalize Kazhdan--Lusztig cells to general algebras.

\subsection*{Acknowledgements}

The author thanks Andrew Mathas for suggesting to show the uniform block algebra is cellular by using the partition algebra and~\cite{Xi99}, numerous conversions over the years about cellular algebras, and comments on an earlier draft of this paper.
The author thanks Stephen Doty for catching a number of mistakes in an earlier drafts of this paper and providing numerous suggestions, historical comments, and references and deeply appreciates his proofreading.
The author thanks Hyohe Miyachi for the proof of Conjecture~\ref{conj:tangle_pqp} when both algebras are semisimple.
The author thanks Georgia Benkart, Rosa Orellana, Tom Halverson, Arun Ram, J{\o}rgen Rasmussen, David Ridout, Franco Saliola, Anne Schilling, and Mike Zabrocki for useful conversations and suggestions.
The author thanks the anonymous referee for their helpful comments and for the references regarding the braids and ties algebra.

The author was partially supported by Grant-in-Aid for JSPS Fellows 21F51028.
This work was partly supported by Osaka City University Advanced Mathematical Institute (MEXT Joint Usage/Research Center on Mathematics and Theoretical Physics JPMXP0619217849).

\section{Background}
\label{sec:background}

Consider a positive integer $m \in \ZZ_{>0}$.
Let $[m] := \{1 < 2 < \cdots < n\}$ and $[m'] := \{1' < 2' < \cdots < n'\}$.
Let $\field$ denote a field of characteristic $p$ (possibly $p = 0$).
All of the $\field$-algebras considered in this paper will be associative and unital.
Unless otherwise specified, tensor products will be over~$\field$.
Let $\ZZ_r := \ZZ / r \ZZ$, which we will often consider as a cyclic abelian group (under $+$).

Let $\Sym_m$ denote the symmetric group on $[m]$.
For $r, d \in \ZZ_{>0}$ such that $d \divides r$, let $G(r,d,m)$ denote the complex reflection group given by $r$ colored permutations (where multiplication is constructed as the wreath product $\ZZ_r \wr \Sym_m$) such that the sum of the colors is equivalent to $0 \pmod{d}$.
This has a natural representation $\natrepr$ on $\CC^m$ corresponding to products of permutation matrices and diagonal matrices $D(\zeta_r^{i_1}, \dotsc, \zeta_r^{i_m})$ such that $i_1 + \cdots + i_m \equiv 0 \pmod{d}$, where $\zeta_r$ is a primitive $r$-th root of unity.
In particular, the matrices all are generalized permutation matrices, with each row and column having exactly one nonzero entry of the form $\zeta_r^k$.
Note that $G(1,1,m) \iso \Sym_m$ and $G(r,1,m) \iso \ZZ_r \wr \Sym_m$.

A (integer) \defn{partition} $\mu$ of $n$ is a weakly decreasing sequence of positive numbers $(\mu_1 \geq \mu_2 \geq \cdots \geq \mu_{\ell} > 0)$ whose sum $\mu_1 + \cdots + \mu_{\ell} = n$.
We write this as $\mu \vdash n$ or $\abs{\mu} = n$, and let $\ell(\mu) = \ell$ denote its length.
For nonnegative integers $m,n \in \ZZ_{\geq 0}$ and partitions $\mu \in m$ and $\nu \in n$, we say $\mu \leq \nu$ in graded dominance order if $m \leq n$ or if $m = n$ then $\sum_{i=1}^k \mu_i \leq \sum_{i=1}^k \nu_i$ for all $k$ (\textit{i.e.}, usual dominance order), where we extend $\mu$ and $\nu$ with an infinite number of trailing $0$'s.
A \defn{standard tableau} of $\mu$ by a totally ordered alphabet $A$ is a filling of the Young diagram of $\mu$ such that each letter appears exactly once and rows and columns are (strictly) increasing.
We draw our partitions and tableaux using English convention.

Let $f_{\lambda}$ denote the number of standard tableaux for a partition $\lambda$.
It is a classical fact that
\begin{equation}
\label{eq:fla=n2}
\sum_{\lambda \vdash k} f_{\lambda}^2 = k!
\end{equation}
with many proofs; for example, this is a consequence of the Robinson--Schensted--Knuth (RSK) bijection (see, \textit{e.g.},~\cite[Ch.~7]{ECII}).

\medskip
\noindent
\textbf{Caution:} It can be the case that $\lambda$ is \emph{not} a partition in the sequel.


\subsection{Partition algebras}
\label{sec:partition_algebra}

Fix some $\beta \in \field$.
Let $\mcP_n(\beta)$ denote the \defn{partition algebra}, whose basis is indexed by set partitions of $[n] \sqcup [n]'$, which we represent as diagrams from $[n]$ on the top to $[n]'$ on the bottom and identify the parts of the set partition with connected components.
We often identify the basis elements of $\mcP_n(\beta)$ with their defining set partition.
Following composition conventions, multiplication $\rho \cdot \sigma$ of basis elements (diagrams) $\rho$ and $\sigma$ as the set partition formed by stacking the diagram of $\sigma$ on top of $\rho$ and removing the $D$ interior components times $\beta^D$.
We say a block (or part) $\rho_i$ of a basis element $\rho$ is \defn{propagating} if $\rho_i \cap [n] \neq \emptyset$ and $\rho_i \cap [n]' \neq \emptyset$, \textit{i.e.}, it connects to both sides of the diagram.

\begin{ex}
\label{ex:diagram_multiplication}
For the diagrams
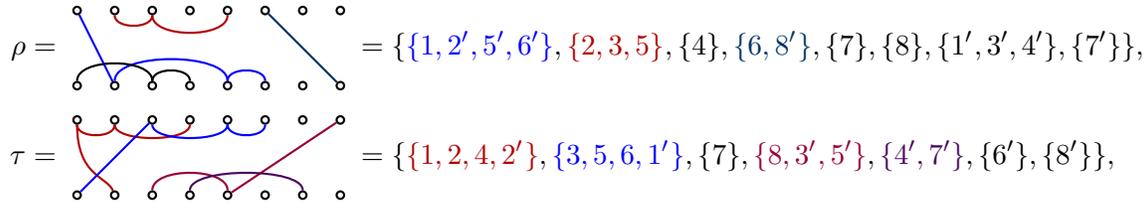
\begin{align*}
\rho & = \;
\begin{tikzpicture}[scale = 0.5,thick, baseline={(0,-1ex/2)}] 
\tikzstyle{vertex} = [shape=circle, minimum size=2pt, inner sep=1pt]
\foreach \i in {1,...,8} {
    \node[vertex] (G\i) at (\i, 1) [shape=circle, draw] {};
    \node[vertex] (G-\i) at (\i, -1) [shape=circle, draw] {};
}
\draw[-,blue] (G1) -- (G-2) .. controls +(0, .9) and +(0, .9) .. (G-5) .. controls +(0, .5) and +(0, .5) .. (G-6);
\draw[-,color=darkred] (G2) .. controls +(0, -.5) and +(0, -.5) .. (G3) .. controls +(0, -.75) and +(0, -.75) .. (G5);
\draw[-] (G-1) .. controls +(0, .75) and +(0, .75) .. (G-3) .. controls +(0, .5) and +(0, .5) .. (G-4);
\draw[-,color=OCUsapphire] (G6) -- (G-8);
\end{tikzpicture}
\; = \{{\color{blue}\{1,2',5',6'\}}, {\color{darkred}\{2,3,5\}}, \{4\}, {\color{OCUsapphire}\{6,8'\}}, \{7\}, \{8\}, \{1',3',4'\}, \{7'\}\},
\\[5pt]
\tau & = \;
\begin{tikzpicture}[scale = 0.5,thick, baseline={(0,-1ex/2)}] 
\tikzstyle{vertex} = [shape=circle, minimum size=2pt, inner sep=1pt]
\foreach \i in {1,...,8} {
    \node[vertex] (G\i) at (\i, 1) [shape=circle, draw] {};
    \node[vertex] (G-\i) at (\i, -1) [shape=circle, draw] {};
}
\draw[-,color=darkred] (G-2) .. controls +(-.2, 0.5) and +(0, -1.5) .. (G1) .. controls +(0, -.5) and +(0, -.5) .. (G2) .. controls +(0, -.6) and +(0, -.6) .. (G4);
\draw[-,blue] (G-1) -- (G3) .. controls +(0, -.6) and +(0, -.6) .. (G5) .. controls +(0, -.5) and +(0, -.5) .. (G6);
\draw[-,color=OCUenji] (G-3) .. controls +(0, .75) and +(0, .75) .. (G-5) -- (G8);
\draw[-,color=UQpurple] (G-4) .. controls +(0, .75) and +(0, .75) .. (G-7);
\end{tikzpicture}
 \; = \{{\color{darkred}\{1,2,4,2'\}}, {\color{blue}\{3,5,6,1'\}}, \{7\}, {\color{OCUenji}\{8,3',5'\}}, {\color{UQpurple}\{4',7'\}}, \{6'\}, \{8'\}\},
\end{align*}
we have the multiplication
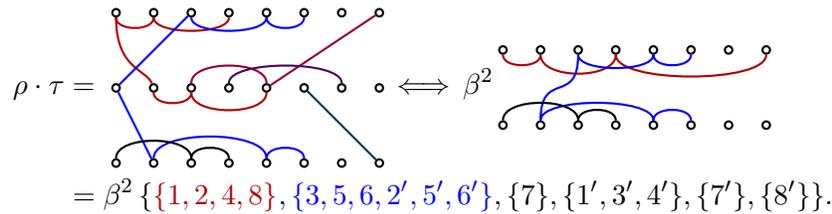
\begin{align*}
\rho \cdot \tau & = \;
\begin{tikzpicture}[scale = 0.5,thick, baseline={(0,-3.5ex)}] 
\tikzstyle{vertex} = [shape=circle, minimum size=2pt, inner sep=1pt]
\foreach \i in {1,...,8} {
    \node[vertex] (G\i) at (\i, 1) [shape=circle, draw] {};
    \node[vertex] (G-\i) at (\i, -1) [shape=circle, draw] {};
    \node[vertex] (Gp-\i) at (\i, -3) [shape=circle, draw] {};
}
\draw[-,blue] (G-1) -- (Gp-2) .. controls +(0, .9) and +(0, .9) .. (Gp-5) .. controls +(0, .5) and +(0, .5) .. (Gp-6);
\draw[-,color=darkred] (G-2) .. controls +(0, -.5) and +(0, -.5) .. (G-3) .. controls +(0, -.75) and +(0, -.75) .. (G-5);
\draw[-] (Gp-1) .. controls +(0, .75) and +(0, .75) .. (Gp-3) .. controls +(0, .5) and +(0, .5) .. (Gp-4);
\draw[-,color=OCUsapphire] (G-6) -- (Gp-8);
\draw[-,color=darkred] (G-2) .. controls +(-.2, 0.5) and +(0, -1.5) .. (G1) .. controls +(0, -.5) and +(0, -.5) .. (G2) .. controls +(0, -.6) and +(0, -.6) .. (G4);
\draw[-,blue] (G-1) -- (G3) .. controls +(0, -.6) and +(0, -.6) .. (G5) .. controls +(0, -.5) and +(0, -.5) .. (G6);
\draw[-,color=OCUenji] (G-3) .. controls +(0, .75) and +(0, .75) .. (G-5) -- (G8);
\draw[-,color=UQpurple] (G-4) .. controls +(0, .75) and +(0, .75) .. (G-7);
\end{tikzpicture}
\, \Longleftrightarrow \,
\beta^2 \,
\begin{tikzpicture}[scale = 0.5,thick, baseline={(0,-1ex/2)}] 
\tikzstyle{vertex} = [shape=circle, minimum size=2pt, inner sep=1pt]
\foreach \i in {1,...,8} {
    \node[vertex] (G\i) at (\i, 1) [shape=circle, draw] {};
    \node[vertex] (G-\i) at (\i, -1) [shape=circle, draw] {};
}
\draw[-,color=darkred] (G1) .. controls +(0, -.5) and +(0, -.5) .. (G2) .. controls +(0, -.8) and +(0, -.8) .. (G4) .. controls +(0, -.9) and +(0, -.9) .. (G8);
\draw[-,blue] (G-6) .. controls +(0, .5) and +(0, .5) .. (G-5) .. controls +(0, .75) and +(0, .75) .. (G-2) .. controls +(0, 1.5) and +(0, -1.5) .. (G3) .. controls +(0, -.6) and +(0, -.6) .. (G5) .. controls +(0, -.5) and +(0, -.5) .. (G6);
\draw[-] (G-1) .. controls +(0, .75) and +(0, .75) .. (G-3) .. controls +(0, .5) and +(0, .5) .. (G-4);
\end{tikzpicture}
\\ & = \beta^2 \, \{{\color{darkred}\{1,2,4,8\}}, {\color{blue}\{3,5,6,2',5',6'\}}, \{7\}, \{1',3',4'\}, \{7'\}, \{8'\}\}.
\end{align*}
We note that $\rho$ (resp.~$\tau$ and $\rho \cdot \tau$) has $2$ (resp.~$3$ and $1$) propagating blocks.
Likewise, we compute
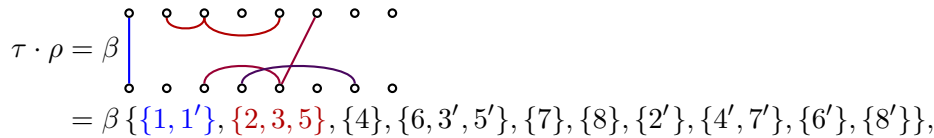
\begin{align*}
\tau \cdot \rho & = \beta \,
\begin{tikzpicture}[scale = 0.5,thick, baseline={(0,-1ex/2)}] 
\tikzstyle{vertex} = [shape=circle, minimum size=2pt, inner sep=1pt]
\foreach \i in {1,...,8} {
    \node[vertex] (G\i) at (\i, 1) [shape=circle, draw] {};
    \node[vertex] (G-\i) at (\i, -1) [shape=circle, draw] {};
}
\draw[-,blue] (G1) -- (G-1);
\draw[-,color=darkred] (G2) .. controls +(0, -.5) and +(0, -.5) .. (G3) .. controls +(0, -.75) and +(0, -.75) .. (G5);
\draw[-,color=OCUenji] (G-3) .. controls +(0, .75) and +(0, .75) .. (G-5) -- (G6);
\draw[-,color=UQpurple] (G-4) .. controls +(0, .75) and +(0, .75) .. (G-7);
\end{tikzpicture}
\\ & = \beta \, \{{\color{blue}\{1,1'\}}, {\color{darkred}\{2,3,5\}}, \{4\}, \{6, 3', 5'\}, \{7\}, \{8\}, \{2'\}, \{4',7'\}, \{6'\}, \{8'\}\},
\end{align*}
which also has only $1$ propagating block.
\end{ex}

\begin{remark}
Our multiplication convention might be the reverse of some authors, \textit{e.g.},~\cite{HJ20}.
\end{remark}

By~\cite[Lemma~4.2]{Xi99}, we have a decomposition
\begin{equation}
\label{eq:partition_decomposition}
\mcP_n \iso \bigoplus_{k=0}^n V'_k \otimes \field[\Sym_k] \otimes V_k
\end{equation}
as $\field$-modules, where $V_k$ (resp.~$V'_k$) is the free $\field$-module spanned by set partitions $\rho$ of $[n]$ (resp.~$[n]'$) with at least $k$ parts and chosen subset $S \subseteq \rho$ such that $\abs{S} = k$ called the \defn{defects}.
This means we can decompose our natural basis elements of $\mcP_n$ into two set partitions and a permutation encoding the crossings of the connecting blocks using the mapping from $\min(\rho_i \cap [n]) \mapsto \min(\rho_i \cap [n]')$.
Pictorially, this is dividing our diagram for $\rho$ into three parts, the lower set partition on $[n]$, the middle part with the crossings, and the upper set partition on $[n]'$.
Note that in the decomposition, the defects in the lower and upper parts correspond to the propagating blocks.
We call the basis elements of $V'_k$ (and $V_k$) \defn{half diagrams}.

\begin{ex}
\label{ex:decomposition}
If we consider the diagram $\tau$ from Example~\ref{ex:diagram_multiplication} (drawn below using a different realization where the propagation is done from smallest element to smallest element), then this decomposition is given by
\begin{align*}
\tau & \longleftrightarrow \;
\begin{tikzpicture}[scale = 0.5,thick, baseline={(0,-2ex/2)}] 
\tikzstyle{vertex} = [shape=circle, minimum size=2pt, inner sep=1pt]);
\foreach \i in {1,...,8} {
    \node[vertex] (G\i) at (\i, 2) [shape=circle, draw] {};
    \node[vertex] (G-\i) at (\i, -2) [shape=circle, draw] {};
}
\draw[-,color=darkred] (G-2) .. controls +(0, 1.5) and +(0, -1.5) .. (G1) .. controls +(0, -.5) and +(0, -.5) .. (G2) .. controls +(0, -.75) and +(0, -.75) .. (G4);
\draw[-,blue] (G-1) .. controls +(0, 1.5) and +(0, -1.5) .. (G3) .. controls +(0, -.75) and +(0, -.75) .. (G5) .. controls +(0, -.5) and +(0, -.5) .. (G6);
\draw[-,color=OCUenji] (G-5) .. controls +(0, .75) and +(0, .75) .. (G-3) .. controls +(0, 1.5) and +(0, -1.5) .. (G8);
\draw[-,color=UQpurple] (G-4) .. controls +(0, .75) and +(0, .75) .. (G-7);
\draw[dashed,black!50] (0,1) -- (9,1);
\draw[dashed,black!50] (0,-1) -- (9,-1);
\end{tikzpicture}
\hspace{30pt}
\begin{gathered}
(\{\{1,2,4\}, \{3,5,6\}, \{8\}\}, \{\{7\}\}) = v_{\tau}
\\
\begin{bmatrix} 1 & 3 & 8 \\ 2' & 1' & 3' \end{bmatrix} = \sigma_{\tau}
\\
(\{\{1'\},\{2'\},\{3',5'\}\},\{\{4',7'\},\{6'\},\{8'\}\}) = v'_{\tau}
\end{gathered}
\\ & \longleftrightarrow v'_{\tau} \otimes \sigma_{\tau} \otimes v_{\tau} \in V'_3 \otimes \field [\Sym_3] \otimes V_3,
\end{align*}
where the set partition is the union of the pair and the first part is the defects and the permutation is written in two-line notation.
\end{ex}

The following is a classical fact due to Jones.

\begin{thm}[{\cite{Jones94}}]
\label{thm:partition_double_centralizer}
There exists a surjection $\mcP_n(m) \to \End_{\Sym_m} \natrepr^{\otimes n}$.
Furthermore, this map is a bijection if and only if $m \geq 2n$.
\end{thm}

An explicit bijection between the Bratteli diagram approach for the general theory of $\End_{\Sym_m} \natrepr^{\otimes n}$ and the partition algebra when $m \geq 2n$ was constructed in~\cite{COSSZ20}.

\begin{remark}
The partition algebra has a nice set of generators (and relations) that involve fairly simple diagrams.
While these are useful to prove a number of facts (such as isomorphisms), we generally will not use them.
Consequently, we do not describe such presentations, but they can be found in the references.
\end{remark}

\subsection{Cellular algebras}

We give the necessary definitions following~\cite{GL96}.

\begin{dfn}[{Cellular algebra~\cite{GL96}}]
\label{defn:cellular_algebra}
Let $\mcA$ be a (unital associative) $\field$-algebra with an anti-involution~$\iota$.
Let $\Lambda$ be a finite poset.
Let $M = (M(\lambda) \mid \lambda \in \Lambda)$, where $M(\lambda)$ is a finite set.
Let
\[
C = \{C^{\lambda}_{ST} \mid \lambda \in \Lambda; S,T \in M(\lambda)\}
\]
be a $\field$-basis for $\mcA$.
We say $\mcA$ is a \defn{cellular algebra} with cell datum $(\Lambda, \iota, M, C)$ if
\begin{enumerate}
\item $\iota(C^{\lambda}_{ST}) = C^{\lambda}_{TS}$ and
\item \label{cell_basis_triangular} for every $\lambda \in \Lambda$, $S,T \in M(\lambda)$, and $a \in \mcA$, we can write
  \[
  a C^{\lambda}_{ST} = \sum_{U \in M(\lambda)} r_a(U,S) C^{\lambda}_{UT} + \mcA^{<\lambda},
  \]
  where $r_a(U,S) \in \field$ do not depend on $T$ and $\mcA^{<\lambda} := \Span_{\field}\{C^{\mu}_{UV} \mid \mu < \lambda; U, V \in M(\mu)\}$ is the module of lower order terms.
\end{enumerate}
The basis $C$ is called a \defn{cell basis} of $\mcA$.
\end{dfn}

From the definition, we have the following result, which is likely well-known to experts but the author could not find in the literature, about certain subalgebras of cellular algebras.

\begin{prop}
\label{prop:subcellular}
Let $\mcA$ be a cellular algebra with cell datum $(\Lambda, \iota, M, C)$.
Let $\overline{\mcA} \subseteq \mcA$ be a subalgebra with a basis $\overline{C} \subseteq C$ invariant under $\iota$.
Then $\overline{\mcA}$ is a cellular algebra with cell datum $(\overline{\Lambda}, \iota|_{\overline{\mcA}}, \overline{M}, \overline{C})$ with $\overline{\Lambda}$ and $\overline{M}$ being the indices that appear in $\overline{C}$.
\end{prop}

From the triangularity property~(\ref{cell_basis_triangular}) in the definition, we have the following way to construct new cellular bases that, similar to Proposition~\ref{prop:subcellular}, is likely well-known to experts.

\begin{prop}
\label{prop:cellular_triangle_basis}
Let $\mcA$ be a cellular algebra with cell datum $(\Lambda, \iota, M, C)$.
Then any basis of the form
\[
\widetilde{C} = \{ \widetilde{C}_{ST}^{\lambda} \in C_{ST}^{\lambda} + \mcA^{< \lambda} \mid \lambda \in \Lambda; S,T \in M(\lambda) \}
\]
is a cell basis of $\mcA$ and defines new cell datum $(\Lambda, \widetilde{\iota}, M, \widetilde{C})$ with $\widetilde{\iota}(\widetilde{C}_{ST}^{\lambda}) = \widetilde{C}_{TS}^{\lambda}$.
\end{prop}

\begin{proof}
We only need to show Definition~\ref{defn:cellular_algebra}(\ref{cell_basis_triangular}) holds.
We have
\[
a \widetilde{C}_{ST}^{\lambda} = a C_{ST}^{\lambda} + a \mcA^{<\lambda} = \sum_{U \in M(\lambda)} r_a(U,S) C_{UT} + \mcA^{<\lambda} = \sum_{U \in M(\lambda)} r_a(U,S) \widetilde{C}_{UT} + \mcA^{<\lambda},
\]
since $\mcA^{<\lambda}$ is a left ideal by Definition~\ref{defn:cellular_algebra} (see also~\cite{KX96}) and the change of basis $C \to C'$ is unitrangular.
\end{proof}

For a cellular algebra $\mcA$ with cell datum $(\Lambda, \iota, M, C)$, the \defn{cell module} (or \defn{standard module}) indexed by $\lambda \in \Lambda$ is the free $\field$-module
\[
W(\lambda) := \Span_{\field} \{ C_S \mid S \in M(\lambda) \}
\]
with the natural action $a C_S = \sum_U r_a(U,S) C_U$.
Roughly speaking, the action is given by fixing a $\lambda$ and forgetting one of the indices of the basis.
Hence, $\dim W(\lambda) = \abs{M(\lambda)}$.
An immediate consequence of these definitions is
\begin{equation}
\label{eq:dim_formula}
\dim \mcA = \sum_{\lambda \in \Lambda} \bigl( \dim W(\lambda) \bigr)^2.
\end{equation}

We define a bilinear form $\Phi_{\lambda} \colon W(\lambda) \times W(\lambda) \to \field$ by
\[
C^{\lambda}_{ST} C^{\lambda}_{UV} \equiv \Phi_{\lambda}(C_T, C_U) C^{\lambda}_{SV} \pmod{\mcA^{<\lambda}}
\]
and extended bilinearly.
The bilinear form $\Phi_{\lambda}$ is symmetric and $\mcA$-invariant in the sense $\Phi_{\lambda}(a v, w) = \Phi_{\lambda}(v, \iota(a) w)$ for all $a \in \mcA$ and $v,w \in W(\lambda)$.
This allows us to construct the simple modules.

\begin{thm}[{\cite{GL96}}]
The simple modules of a cellular algebra with cell datum $(\Lambda, \iota, M, C)$ are parameterized by $\{ \lambda \in \Lambda \mid \Phi_{\lambda} \neq 0 \}$.
Moreover, the (absolutely) simple module
\[
L(\lambda) \iso W(\lambda) / R(\lambda),
\]
where $R(\lambda) := \{w \in W(\lambda) \mid \Phi_{\lambda}(v,w) = 0 \text{ for all } v \in W(\lambda)\}$ is the radical of $\Phi_{\lambda}$.
\end{thm}

\begin{remark}
\label{rem:algebraic_closure}
A particular feature of this result is that for a cellular algebra, it is equivalent to work over the algebraic closure of $\field$, which is also necessarily cellular.
In other words, the semisimplicity and the set of simple modules is the same for a cellular algebra over $\field$ as its algebraic closure, which we can get by extension of scalars.
This is the definition of an absolutely simple module.

Note that this is not mean we can consider any subfield of $\field$ and the algebra remains cellular.
In particular, the group algebra $\field [\ZZ_n]$ when $\zeta_n \in \field$ (\textit{i.e.}, $\field$ has all $n$-th roots of unity) has $n$ irreducible representations (all of dimension $1$), but otherwise we can find fewer (nonsplit) irreducible representations of higher dimension.
For example, consider $n = 4$ and compare when $\field = \CC$ and when $\field = \RR$ (which has a simple $2$ dimensional module with the generator of $\ZZ_4$ acting as rotation by $\pi/2$).
Consequently, $\RR [\ZZ_4]$ is not cellular but $\CC [\ZZ_4]$ is (see Section~\ref{sec:wreath} below).
\end{remark}

\begin{remark}
\label{rem:assoc_graded}
From the definition of a cellular algebra, there is a natural $\Lambda$-filtration of left ideals on any cellular algebra $\mcA$; that is $\mcA = \bigcup_{\lambda \in \Lambda} \mcA^{\leq \lambda}$ such that $\mcA^{\leq \lambda} \subseteq \mcA^{\leq \mu}$ whenever $\lambda \leq \mu$ in the poset $\Lambda$ and $\mcA \cdot \mcA^{\leq \lambda} \subseteq \mcA^{\leq \lambda}$ for any fixed $\lambda \in \Lambda$.
It is clear $\iota(\mcA^{\leq \lambda}) = \mcA^{\leq \lambda}$, and so the left ideal $\mcA^{\leq \lambda}$ is a two-sided ideal.
The cell modules are the ``square root by $\iota$'' of the direct summands of the associated graded module of the (left) regular representation with respect to this filtration; more precisely, $\mcA = \bigoplus_{\lambda \in \Lambda} \mcA^{\leq \lambda} / \mcA^{< \lambda}$ and $\mcA^{\leq \lambda} / \mcA^{< \lambda} \iso W(\lambda) \otimes \iota\bigl( W(\lambda) \bigr)$.
This is essentially the basis-free definition of cellular algebras given in~\cite{KX96} as the cell modules are left ideals of $\mcA^{\leq \lambda}$.
The construction of the link modules for the Temperley--Lieb algebra (see, \textit{e.g.},~\cite{GL96,RSA14}) is similar by looking at a maximal $\Lambda$-indexed chain of left ideals.
\end{remark}

We note that $\field [\Sym_k]$ has an explicit construction of a cellular basis with $M(\lambda)$ given as the set of standard Young tableaux of shape $\lambda \vdash k$; see~\cite{GL96}.
By using a cellular basis for $\field [\Sym_k]$ and the decomposition~\eqref{eq:partition_decomposition}, we have the following.

\begin{thm}[{\cite[Thm.~4.1]{Xi99}}]
\label{thm:partition_cellular}
The partition algebra $\mcP_n(\beta)$ is a cellular algebra with cell datum $(\Lambda, \iota, M, C)$ given by
\begin{itemize}
\item $\Lambda = \{\lambda \vdash k \mid k \in \{0,1,2,\dotsc,n\}\}$ under graded dominance order;
\item $\iota$ reflecting the diagram vertically;
\item $M(\lambda)$ is all pairs $(\rho, T)$, where $\rho$ is a set partition of $[n]$ such that $\abs{\rho} \geq \abs{\lambda}$, where $\abs{\rho}$ denotes the number of parts of $\rho$ as usual, and $T$ is a standard tableau of shape $\lambda$ in the alphabet $A \subseteq \rho$ such that $\abs{A} = \abs{\lambda}$ under some total order on the subsets of $[n]$;
\item some cell basis $C$.
\end{itemize}
\end{thm}

\subsection{Cellular partition theory}
\label{sec:crt}

Recall that the characteristic of $\field$ is $p$.
An important consequence~\cite[Cor.~4.11]{Xi99} is we can parameterize the simple modules of $\mcP_n(\beta)$ as those $\lambda \in \Lambda$ that are a $p$-partition of size $\delta_{\beta0} \leq k \leq n$.
Next, we reinterpret the description of the cell modules from~\cite[Cor.~4.10]{Xi99} using interpretation given in~\cite{HL05} with the half diagram description.
Indeed, the basis for $W(\lambda)$ is given by $M(\lambda)$, where we use the total order given on subsets of $[n]$ by comparing the smallest value in each part.\footnote{Some authors use the largest value, such as in~\cite{BH19,HJ20,OSSZ21}. This is inconsequential and can be considered as using a different cellular basis.}
Note that the elements in the standard tableau are the defects, and the cell module comes with a natural $\mcP_n(\beta)$-action.
We note that the $\mcP_n(\beta)$-action cannot increase the number of defects, and since we want the number of defects to remain fixed, any element of $\mcP_n(\beta)$ that decreases the number of defects acts by $0$.
This translates to an action of the natural basis on the set partition and tableau pair, where we can apply the Garnir straightening relations.
Hence, two alternative descriptions~\cite{BDVK15,Xi99} of the cell module for $\lambda \vdash k$ are
\begin{equation}
\label{eq:specht_decomposition}
W(\lambda) \iso V'_k \otimes S^{\lambda} \otimes v_k \iso V'_{n,k} \otimes_{\field [\Sym_k]} S^{\lambda},
\end{equation}
where $S^{\lambda}$ is the Specht module of $\Sym_k$, $v_k$ is a fixed vector in $V_k$, and $V'_{n,k}$ is the $\field$-span of the diagrams with $k$ propagating blocks and $i$, for all $1 \leq i \leq n-k$, a singleton block.

From the half diagram description, we can write some formulas for the dimension of $V_k'$, which then yields $\dim W(\lambda) = f_{\lambda} \dim V'_k$.
Let $\stirling{a}{b}$ denote the Stirling number of the second kind counting the number of set partitions of $[a]$ into $b$ parts.
Define $B_a = \sum_{b=0}^a \stirling{a}{b}$ as the $a$-th Bell number that counts the number of set partitions of $[a]$,
We claim
\[
\dim V'_k = \sum_{j=k}^n \binom{n}{j} \stirling{j}{k} B_{n-j} = \sum_{j=k}^n \binom{j}{k} \stirling{n}{j}.
\]
The first formula comes from choosing a subset of size $j$ for all of the defect blocks and then taking a set partition on the remaining elements, which was proven bijectively in~\cite[Thm.~2.4]{CDDSY07} (see also~\cite{CD02,DS95,Roby91,Roby95}).
The second formula is from choosing $k$ parts of a set partition of $n$ (with exactly $j$ parts) to be the defects.

When $\field$ is a field of characteristic $0$ and $\beta \notin \{0,1,2,\dotsc,2n-1\}$, the cell modules are the set of simple modules from the double-centralizer theory with $\Sym_m$ (see, \textit{e.g.},~\cite{HR05,MS94II}).
Alternatively, vacillating tableaux or, equivalently, chains in the Bratelli diagram give dimension formulas for the cell modules.

We can give a necessary pictorial condition for the bilinear form $\Phi_{\lambda}(C_T, C_U) = 0$ as the diagram formed by reflecting $U$ vertically and connecting it with $T$ does not induce a bijection between the defects of $T$ and $U$.
Indeed, this will cause the number of propagating blocks (which equals the number of defects on the top and bottom) in the result to decrease and corresponds to a multiplication by $0$; see also Remark~\ref{rem:assoc_graded}.
On the other hand, if we have a bijection, then it induces a permutation in $\Sym_k$, which is the one coming from the decomposition of the diagrams for $\mcP_n(\beta)$.
Furthermore, $\Phi_{\lambda}(C_T, C_U)$ is equal to $\beta^D$, where $D$ is the number of interior components not containing a defect, times the bilinear form induced from the module $S^{\lambda}$ of $\field [\Sym_k]$ evaluated at the corresponding induced permutation.

\begin{ex}
Consider the elements $v_{\tau}$ and $v'_{\tau}$ from Example~\ref{ex:decomposition}, which identify with elements $C_U$ and $C_T$, respectively, in a cell module $W(\lambda)$, where $\abs{\lambda} = 3$.
Then the corresponding pairing $\Phi_{\lambda}(C_T, C_U) = \Phi_{\lambda}(C_U, C_T) = 0$ as we do not have a bijection between the defects:
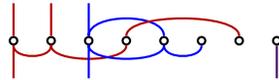
\[
\begin{tikzpicture}[scale = 0.5,thick, baseline=10pt] 
\tikzstyle{vertex} = [shape=circle, minimum size=2pt, inner sep=1pt]
\foreach \i in {1,...,8} {
    \node[vertex] (G\i) at (\i, 2) [shape=circle, draw] {};
}
\draw[-,color=darkred] (G1) + (0,-1) -- (G1) .. controls +(0, -.5) and +(0, -.5) .. (G2) .. controls +(0, -.6) and +(0, -.6) .. (G4);
\draw[-,blue] (G3) + (0,-1) -- (G3) .. controls +(0, -.6) and +(0, -.6) .. (G5) .. controls +(0, -.5) and +(0, -.5) .. (G6);
\draw[-,color=blue] (G5) .. controls +(0, .75) and +(0, .75) .. (G3) -- ++(0,1);
\draw[-,color=darkred] (G4) .. controls +(0, .75) and +(0, .75) .. (G7);
\draw[-,color=darkred] (G1) -- ++(0,1);
\draw[-,color=darkred] (G2) -- ++(0,1);
\draw[-,color=UQpurple] (G8) -- ++(0,-1);
\end{tikzpicture}
\]
where we have reflected $T$ instead of $U$.
However, if we instead took
\[
U' = (\{\{2'\},\{3',5'\},\{8'\}\},\{\{1'\},\{4',7'\},\{6'\}\}),
\]
then for $\Phi_{\lambda}(C_T, C_{U'})$ we do obtain a bijection:
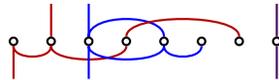
\[
\begin{tikzpicture}[scale = 0.5,thick, baseline=10pt] 
\tikzstyle{vertex} = [shape=circle, minimum size=2pt, inner sep=1pt]
\foreach \i in {1,...,8} {
    \node[vertex] (G\i) at (\i, 2) [shape=circle, draw] {};
}
\draw[-,color=darkred] (G1) + (0,-1) -- (G1) .. controls +(0, -.5) and +(0, -.5) .. (G2) .. controls +(0, -.6) and +(0, -.6) .. (G4);
\draw[-,blue] (G3) + (0,-1) -- (G3) .. controls +(0, -.6) and +(0, -.6) .. (G5) .. controls +(0, -.5) and +(0, -.5) .. (G6);
\draw[-,color=blue] (G5) .. controls +(0, .75) and +(0, .75) .. (G3) -- ++(0,1);
\draw[-,color=darkred] (G4) .. controls +(0, .75) and +(0, .75) .. (G7);
\draw[-,color=UQpurple] (G8) -- ++(0,1);
\draw[-,color=darkred] (G2) -- ++(0,1);
\draw[-,color=UQpurple] (G8) -- ++(0,-1);
\end{tikzpicture}
\]
which induces the identity permutation and is scaled by $\beta^0$.
Finally, if we take
\[
U'' = (\{\{3',5',6'\},\{4'\},\{8'\}\},\{\{1'\},\{2'\},\{7'\}),
\]
then for $\Phi_{\lambda}(C_T, C_{U'})$ obtain
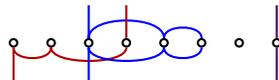
\[
\begin{tikzpicture}[scale = 0.5,thick, baseline=10pt] 
\tikzstyle{vertex} = [shape=circle, minimum size=2pt, inner sep=1pt]
\foreach \i in {1,...,8} {
    \node[vertex] (G\i) at (\i, 2) [shape=circle, draw] {};
}
\draw[-,color=darkred] (G1) + (0,-1) -- (G1) .. controls +(0, -.5) and +(0, -.5) .. (G2) .. controls +(0, -.6) and +(0, -.6) .. (G4);
\draw[-,blue] (G3) + (0,-1) -- (G3) .. controls +(0, -.6) and +(0, -.6) .. (G5) .. controls +(0, -.5) and +(0, -.5) .. (G6);
\draw[-,color=blue] (G6) .. controls +(0, .5) and +(0, .5) .. (G5) .. controls +(0, .75) and +(0, .75) .. (G3) -- ++(0,1);
\draw[-,color=darkred] (G4) + (0,1) -- (G4);
\draw[-,color=UQpurple] (G8) -- ++(0,1);
\draw[-,color=UQpurple] (G8) -- ++(0,-1);
\end{tikzpicture}
\]
which induces the simple transposition $(1 \; 2)$ and scaled by $\beta^1$.
\end{ex}

\section{Cellular subalgebras}

For the remainder of this paper $\iota$ will be the automorphism of the partition algebra that reflects the diagram vertically (\textit{i.e.}, sends $i \leftrightarrow i'$ within the set partition) restricted to the subalgebra we are considering.

\subsection{Half integer partition algebra}

We begin with the \defn{half integer partition algebra} $\mcP_{n-1/2}(\beta)$ studied in~\cite{HR05}, which is defined as the subalgebra spanned by all diagrams such that $n$ and $n'$ are in the same part.
As such, the analysis of the representation theory is similar to that of the usual partition algebra with a few small differences.
We have $\dim \mcP_{n-1/2}(\beta) = B_{2n-1}$ since we chose the set partition for $[n-1] \sqcup [n]'$ and the part $n$ belongs to is the same as $n`$.
This is analogous to $\dim \mcP_n(\beta) = B_{2n}$.
By Proposition~\ref{prop:subcellular}, we have the following.

\begin{prop}
The half integer partition algebra $\mcP_{n-1/2}(\beta)$ is cellular.
\end{prop}

Since $n$ and $n'$ need to be in the same block, the half diagrams on $[n]'$ are now required to have $n'$ as an element of a defect, and so $\Lambda$.
Thus, we see that $\Lambda$ is the set of all partitions $\lambda$ such that $1 \leq \abs{\lambda} \leq n$ and
\[
\dim W(\lambda) = (k \widetilde{V}'_k + \widetilde{V}'_{k-1}) f_{\lambda},
\]
where $k:= \abs{\lambda}$ and $\widetilde{V}'_k$ is the half diagram module for $\mcP_{n-1}(\beta-1)$ (from~\eqref{eq:specht_decomposition}), by choosing the defect block to add $n'$ to or if $n'$ is a defect on its own.

\subsection{Quasi-partition algebra}

The quasi-partition algebra $\mcQ\mcP_n(\beta)$ was defined in~\cite{DO14} as a centralizer algebra with $\beta = m \in \ZZ_{>0}$.
However, we instead define it as a subalgebra of $\mcP_n(\beta - 1)$ using~\cite[Lemma~2.3]{DO14} with a basis indexed by all diagrams without any isolated vertices (\textit{cf.}~\cite[Sec.~2.4]{DO14}).
(Due to the differences in multiplication conventions, we need to flip all diagrams by $\iota$.)
However, the basis is not given by these diagrams alone, but instead as a sum over certain subsets.
In order for this algebra to be well-defined, we require $\beta \neq 0$.
For brevity, we do not include the explicit description of the basis here as we only need the properties given by~\cite[Lemma~2.2]{DO14}.

\begin{thm}
\label{thm:quasi_cellular}
The quasi-partition algebra is a cellular algebra.
\end{thm}

\begin{proof}
By~\cite[Lemma~2.2]{DO14}, the basis of $\mcQ\mcP_n(\beta)$ is triangular when expressed in the basis of $\mcP_n(\beta-1)$ with respect to refinement, denoted $\leq$, of set partitions.
Furthermore, when we refine any block on the top, all of the vertices in that block will become isolated.
Hence, we strictly decrease the number of propagating blocks when we perform any refinements of the top (the element of $V_k$ in the decomposition~\eqref{eq:partition_decomposition}).
Therefore, in terms of~\eqref{eq:partition_decomposition}, we can express any basis element $\widetilde{\rho} \in \mcQ\mcP_n(\beta)$ as an element $\mcP_n(\beta-1)$ by
\begin{equation}
\label{eq:top_expansion}
\widetilde{\rho} = \widetilde{v}'_{\rho} \otimes \widetilde{\sigma}_{\rho} \otimes \widetilde{v}_{\rho} = \sum_{\tau \leq \rho} b_{\tau\rho} (v'_{\tau} \otimes \sigma_{\tau} \otimes \widetilde{v}_{\rho}) + \mcA^{<},
\end{equation}
where $\widetilde{v}_{\rho} = v_{\rho}$ and $\mcA^{<} = \sum_{\lambda \vdash (k-1)} \mcA^{<\lambda}$ with $k$ equaling the number of propagating blocks of $\rho$.
Thus, by the argument of Proposition~\ref{prop:cellular_triangle_basis}, we can ignore these terms in $\mcA^{<}$.
Next, the definition of multiplication in $\mcP_n(\beta-1)$ implies there exists maps $\phi_k \colon \mcQ\mcP_n(\beta) \times V_k \to V_k$ and $\theta_k \colon \mcQ\mcP_n(\beta) \times V_k \to \Sym_k$ such that
\[
a \cdot \widetilde{\rho} = \phi_k(a, \widetilde{v}'_{\tau}) \otimes \theta_k(a, \widetilde{v}'_{\tau}) \widetilde{\sigma}_{\rho} \otimes v_{\rho}.
\]
We use the map $\widetilde{\iota}$ that simply reflects the diagram indexing the basis (like for the usual partition algebra) as the anti-involution (analogous to Proposition~\ref{prop:cellular_triangle_basis}).
That this is an anti-involution follows from~\cite[Eq.~(5)]{DO14} and the fact that flipping the diagram is an anti-involution on $\mcP_n(\beta-1)$.
Hence, this satisfies the conditions for an interated inflation~\cite{GP18,Xi99}.
\end{proof}

\begin{ex}
The map $\widetilde{\iota}$ is not simply the restriction of $\iota$ for $\mcP_n(\beta-1)$ since
\[
\begin{tikzpicture}[scale = 0.5,thick, baseline={(0,-1ex/2)}] 
\tikzstyle{vertex} = [shape=circle, minimum size=2pt, inner sep=1pt]
\foreach \i in {1,2} {
    \node[vertex] (G\i) at (\i, 1) [shape=circle, draw] {};
    \node[vertex] (G-\i) at (\i, -1) [shape=circle, draw] {};
}
\draw[-] (G1) .. controls +(0, -.5) and +(0, -.5) .. (G2) -- (G-2) .. controls +(0, .5) and +(0, .5) .. (G-1) -- (G1);
\end{tikzpicture}
\quad \longleftrightarrow \quad
\begin{tikzpicture}[scale = 0.5,thick, baseline={(0,-1ex/2)}] 
\tikzstyle{vertex} = [shape=circle, minimum size=2pt, inner sep=1pt]
\foreach \i in {1,2} {
    \node[vertex] (G\i) at (\i, 1) [shape=circle, draw] {};
    \node[vertex] (G-\i) at (\i, -1) [shape=circle, draw] {};
}
\draw[-] (G1) .. controls +(0, -.5) and +(0, -.5) .. (G2) -- (G-2) .. controls +(0, .5) and +(0, .5) .. (G-1) -- (G1);
\end{tikzpicture}
- \beta^{-1}\,
\begin{tikzpicture}[scale = 0.5,thick, baseline={(0,-1ex/2)}] 
\tikzstyle{vertex} = [shape=circle, minimum size=2pt, inner sep=1pt]
\foreach \i in {1,2} {
    \node[vertex] (G\i) at (\i, 1) [shape=circle, draw] {};
    \node[vertex] (G-\i) at (\i, -1) [shape=circle, draw] {};
}
\draw[-] (G-1) -- (G1) .. controls +(0, -.5) and +(0, -.5) .. (G2);
\end{tikzpicture}
- \beta^{-1}\,
\begin{tikzpicture}[scale = 0.5,thick, baseline={(0,-1ex/2)}] 
\tikzstyle{vertex} = [shape=circle, minimum size=2pt, inner sep=1pt]
\foreach \i in {1,2} {
    \node[vertex] (G\i) at (\i, 1) [shape=circle, draw] {};
    \node[vertex] (G-\i) at (\i, -1) [shape=circle, draw] {};
}
\draw[-] (G1) .. controls +(0, -.5) and +(0, -.5) .. (G2) -- (G-2);
\end{tikzpicture}
+ \beta^{-2}\,
\begin{tikzpicture}[scale = 0.5,thick, baseline={(0,-1ex/2)}] 
\tikzstyle{vertex} = [shape=circle, minimum size=2pt, inner sep=1pt]
\foreach \i in {1,2} {
    \node[vertex] (G\i) at (\i, 1) [shape=circle, draw] {};
    \node[vertex] (G-\i) at (\i, -1) [shape=circle, draw] {};
}
\draw[-] (G1) .. controls +(0, -.5) and +(0, -.5) .. (G2);
\end{tikzpicture}
+ \beta^{-2}\,
\begin{tikzpicture}[scale = 0.5,thick, baseline={(0,-1ex/2)}] 
\tikzstyle{vertex} = [shape=circle, minimum size=2pt, inner sep=1pt]
\foreach \i in {1,2} {
    \node[vertex] (G\i) at (\i, 1) [shape=circle, draw] {};
    \node[vertex] (G-\i) at (\i, -1) [shape=circle, draw] {};
}
\end{tikzpicture}\, \in \mcP_n(\beta-1).
\]
Indeed, this element is invariant under $\widetilde{\iota}$ but not $\iota$.
\end{ex}

A more from-first-principles proof of Theorem~\ref{thm:quasi_cellular} is also possible.
We can build a cellular basis for $\mcQ\mcP_n(\beta)$ and $\mcP_n(\beta-1)$ by starting with a cellular basis of $\Sym_k$ for all $0 \leq k \leq n$ (RSK allows us to go between the cell basis indices and diagrams).
Then Equation~\eqref{eq:top_expansion} implies that this cellular basis has the analogous triangular expansion in terms of a cellular basis of $\mcP_n(\beta-1)$ that fixes the second index; that is, $\widetilde{C}_{ST}^{\lambda} = \sum_{R \leq S} b_{RS} C_{RT}^{\lambda} + \mcA^{<\lambda}$.
We then would conclude the proof by using that we have a cellular basis of $\mcP_n(\beta-1)$.

Note that $\Lambda$ is given by all partitions of size at most $n$ just like for $\mcP_n(\beta)$, recovering~\cite[Cor.~4.3]{DO14}.
We recover~\cite[Thm.~4.6]{DO14} from a straightforward counting of the half diagrams of the cell modules.

\begin{cor}[{\cite[Thm.~4.6]{DO14}}]
Let $\lambda \in \Lambda$ and $k = \abs{\lambda}$.
Then we have
\[
\dim W(\lambda) = f_{\lambda} \sum_{s=0}^k \binom{n}{s} \sum_{j=k-s}^{\lfloor \frac{n-s}{2} \rfloor} \binom{j}{k-s} \stirling{n-s}{j}_{\geq2},
\]
where $\stirling{a}{b}_{\geq2}$ is the number of set partitions of $a$ into $b$ parts with each part having size at least $2$.
\end{cor}

\begin{proof}
From Theorem~\ref{thm:quasi_cellular}, it is sufficient to count the half diagrams by selecting the defect blocks as the $f_{\lambda}$ comes from the decomposition~\eqref{eq:specht_decomposition}.
We chose exactly $s$ singleton defects, then we chose the remaining $k - s$ defects from the parts of the set partition, which necessarily has no singletons.
\end{proof}

The dimension of the quasi-partition algebra was given in~\cite[Cor.~2.9]{DO14} as
\[
\dim \mcQ\mcP_n(\beta) = \sum_{j=1}^{2n} (-1)^{j-1} B_{2n-j} + 1;
\]
see also~\cite[A000296]{OEIS}.
We can also perform the same analysis like the half partition algebra to get the odd sized base sets, but it would require showing that the restricted basis is closed under multiplication.
If that is true, the Schur--Weyl duality for the half partition algebra suggests a Schur--Weyl duality for the half quasi-partition algebra.

\subsection{Complex reflection group centralizer}

We begin this subsection with the submodule $\mcG^{(r,d,m)}_n(\beta) \subseteq \mcP_n(\beta)$ Tanabe studied in~\cite{Tanabe97} and its relationship with $\End_{G(r,d,m)} \natrepr^{\otimes n}$, the centralizer algebra of $G(r,d,m)$ under the natural diagonal action.
Following~\cite[Lemma~2.1]{Tanabe97}, we \emph{define} $\mcG^{(r,d,m)}_n$ as the submodule of $\mcP_n(\beta)$ with basis given by the set partitions $\rho = \{\rho_1, \rho_2, \dotsc, \rho_{\ell}\}$ such that $\ell \leq m$ and satisfy either of the following conditions:
Denote $N(\rho_i) := \abs{\rho_i \cap [n]}$ and $N'(\rho_i) := \abs{\rho_i \cap [n]'}$
\begin{enumerate}
\item \label{cond:blocks} $N(\rho_i) \equiv N'(\rho_i) \pmod{r}$ for all $i$; or
\item \label{cond:d} $\ell = m$ and
  \begin{enumerate}
    \item \label{cond:diff} $N(\rho_i) \equiv N'(\rho_i) \pmod{r/d}$ for all $i$, and
    \item \label{cond:constant_level} there exists $s \in \{1, 2, \dotsc, r-1\}$ such that $N(\rho_i) - N'(\rho_i) \equiv s \pmod{r}$ for all $i$.
  \end{enumerate}
\end{enumerate}
Note that Condition~(\ref{cond:constant_level}) is stronger than $N(\rho_i) - N'(\rho_i) \not\equiv 0 \pmod{r}$ for all $i$.
We remark that Condition~(\ref{cond:d}) can never be satisfied if $d = 1$ or if $m > 2n$.
Similarly, Condition~(\ref{cond:diff}) becomes vacuous if $d = r$.

In~\cite[Lemma~2.1]{Tanabe97}, it is claimed that $\mcG^{(r,d,m)}_n(m) \iso \End_{G(r,d,m)} \natrepr^{\otimes n}$ as $\field$-algebras (taking $\mcG^{(r,d,m)}$ as a subalgebra), but his claim of linear independence relies on the faithfulness of the $\mcP_n(m)$-representation on $\End_{\Sym_m} \natrepr^{\otimes n}$ from Theorem~\ref{thm:partition_double_centralizer}.
Thus his proof that this does form a basis of the centralizer of the $G(r,d,m)$-action only holds when $m \geq 2n$.
In fact, for $m < n$, $\mcG_n^{(r,1,m)}$ does not necessarily have a unit as the next example shows, and hence it cannot be isomorphic (as $\field$-algebras) to the endomorphism algebra (which is unital).

\begin{ex}
Consider $\mcG_3^{(2,1,2)}(\beta)$ for $\beta \neq 0$, and it can be verified that $\dim \mcG_3^{(2,1,2)}(\beta) = 25$ and $\mcG_3^{(2,1,2)}(\beta)$ is closed under the usual diagram multiplication.
All of its diagrams have either one or two propagating blocks, and we denote the set of diagrams with $i$ propagating blocks by $P_i$.
We have $\abs{P_1} = 16$ and $\abs{P_2} = 9$, and one such diagram in $P_1$ and $P_2$ are
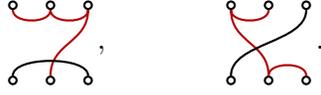
\[
\begin{tikzpicture}[scale = 0.5,thick, baseline={(0,-1ex/2)}] 
\tikzstyle{vertex} = [shape=circle, minimum size=2pt, inner sep=1pt]
\foreach \i in {1,2,3} {
    \node[vertex] (G\i) at (\i, 1) [shape=circle, draw] {};
    \node[vertex] (G-\i) at (\i, -1) [shape=circle, draw] {};
}
\draw[-,color=darkred] (G-2) .. controls +(0,1) and +(0,-1) .. (G3) .. controls +(0, -.5) and +(0, -.5) .. (G2) .. controls +(0, -.5) and +(0, -.5) .. (G1);
\draw[-] (G-3) .. controls +(0., .65) and +(0, .65) .. (G-1);
\end{tikzpicture}\,,
\qquad\qquad
\begin{tikzpicture}[scale = 0.5,thick, baseline={(0,-1ex/2)}] 
\tikzstyle{vertex} = [shape=circle, minimum size=2pt, inner sep=1pt]
\foreach \i in {1,2,3} {
    \node[vertex] (G\i) at (\i, 1) [shape=circle, draw] {};
    \node[vertex] (G-\i) at (\i, -1) [shape=circle, draw] {};
}
\draw[-,color=darkred] (G2) .. controls +(0, -.5) and +(0, -.5) .. (G1) .. controls +(0,-1) and +(0,1) ..
         (G-2) .. controls +(0., .5) and +(0, .5) .. (G-3);
\draw[-] (G3) .. controls +(0,-1) and +(0,1) .. (G-1);
\end{tikzpicture}\,.
\]
Note that $\mcG_3^{(2,1,2)}(\beta)$ does not contain $1 \in \mcP_n(\beta)$ as this diagram has $3$ parts.
The product of diagrams $P_1 \cdot P_2 \subseteq P_1$ and $P_2 \cdot P_1 \subseteq P_1$ by a parity argument.
Hence, $\mcG_3^{(2,1,2)}(\beta)$ is not a unital when considered as a subalgebra.
\end{ex}

However, his main result describing the generators yields the following.

\begin{thm}[{\cite{Tanabe97}}]
\label{thm:complex_centralizer}
Suppose $\mcG_n^{(r,d,m)}$ is an algebra over $\field$.
There exists a surjection
\[
\phi \colon \mcG_n^{(r,d,m)}(m) \to \End_{G(r,d,m)} \natrepr^{\otimes n}.
\]
Furthermore, the surjection $\phi$ is an isomorphism if $m \geq 2n$.
\end{thm}

\begin{conj}
For $r > 1$ and $d = 1$, the map $\psi$ is an isomorphism if and only if $m \geq n$.
\end{conj}

Let us further discuss the assumption that $\mcG_n^{(r,d,m)}(\beta)$ is an algebra.
Indeed, the number of blocks could potentially grow larger than $m$.
In~\cite{Tanabe97}, it was said diagrams with strictly more than $m$ parts are $0$ in $\mcG_n^{(r,d,m)}(\beta)$, which means such diagrams should \emph{span} an ideal $\mcI$ of the partition algebra such that $\mcG_n^{(r,d,m)}(\beta) \iso \mcP_n(\beta) / \mcI$.
Indeed, taking such a quotient is equivalent to setting (the basis elements of) the diagrams to $0$, and there would be no other relations since the remaining diagrams (were claimed in~\cite[Lemma~2.1]{Tanabe97} to) form a basis for $\mcG_n^{(r,d,m)}$.
However, this is not the case, as the next example shows.
This is a byproduct of the fact that the diagrams should be in a quotient of the algebra $\mcG_n^{(r,d,2n)}(\beta)$ for small $m$.

\begin{ex}
\label{ex:G2222_diagrams}
Consider $\mcG_2^{(2,2,2)}(\beta)$.
We have the diagrams coming from Condition~(\ref{cond:d}):
\[
\begin{tikzpicture}[scale = 0.5,thick, baseline={(0,-1ex/2)}] 
\tikzstyle{vertex} = [shape=circle, minimum size=2pt, inner sep=1pt]
\foreach \i in {1,2} {
    \node[vertex] (G\i) at (\i, 1) [shape=circle, draw] {};
    \node[vertex] (G-\i) at (\i, -1) [shape=circle, draw] {};
}
\draw[-] (G-1) -- (G1) .. controls +(0, -.5) and +(0, -.5) .. (G2);
\end{tikzpicture}\,,
\qquad\qquad
\begin{tikzpicture}[scale = 0.5,thick, baseline={(0,-1ex/2)}] 
\tikzstyle{vertex} = [shape=circle, minimum size=2pt, inner sep=1pt]
\foreach \i in {1,2} {
    \node[vertex] (G\i) at (\i, 1) [shape=circle, draw] {};
    \node[vertex] (G-\i) at (\i, -1) [shape=circle, draw] {};
}
\draw[-] (G1) .. controls +(0, -.5) and +(0, -.5) .. (G2) -- (G-2);
\end{tikzpicture}\,,
\qquad\qquad
\begin{tikzpicture}[scale = 0.5,thick, baseline={(0,-1ex/2)}] 
\tikzstyle{vertex} = [shape=circle, minimum size=2pt, inner sep=1pt]
\foreach \i in {1,2} {
    \node[vertex] (G\i) at (\i, 1) [shape=circle, draw] {};
    \node[vertex] (G-\i) at (\i, -1) [shape=circle, draw] {};
}
\draw[-] (G1) -- (G-1) .. controls +(0, .5) and +(0, .5) .. (G-2);
\end{tikzpicture}\,,
\qquad\qquad
\begin{tikzpicture}[scale = 0.5,thick, baseline={(0,-1ex/2)}] 
\tikzstyle{vertex} = [shape=circle, minimum size=2pt, inner sep=1pt]
\foreach \i in {1,2} {
    \node[vertex] (G\i) at (\i, 1) [shape=circle, draw] {};
    \node[vertex] (G-\i) at (\i, -1) [shape=circle, draw] {};
}
\draw[-] (G-1) .. controls +(0, .5) and +(0, .5) .. (G-2) -- (G2);
\end{tikzpicture}\,,
\]
and using these elements, we can form the elements
\begin{equation}
\label{eq:max2_3parts}
\begin{tikzpicture}[scale = 0.5,thick, baseline={(0,-1ex/2)}] 
\tikzstyle{vertex} = [shape=circle, minimum size=2pt, inner sep=1pt]
\foreach \i in {1,2} {
    \node[vertex] (G\i) at (\i, 1) [shape=circle, draw] {};
    \node[vertex] (G-\i) at (\i, -1) [shape=circle, draw] {};
}
\draw[-] (G-1) -- (G1);
\end{tikzpicture}\,,
\qquad\qquad
\begin{tikzpicture}[scale = 0.5,thick, baseline={(0,-1ex/2)}] 
\tikzstyle{vertex} = [shape=circle, minimum size=2pt, inner sep=1pt]
\foreach \i in {1,2} {
    \node[vertex] (G\i) at (\i, 1) [shape=circle, draw] {};
    \node[vertex] (G-\i) at (\i, -1) [shape=circle, draw] {};
}
\draw[-] (G-2) -- (G2);
\end{tikzpicture}\,.
\end{equation}
Hence, the ``basis'' for $\mcG_2^{(2,2,2)}(\beta)$ with $\beta \neq 0$ (such as $\beta = m = n = 2$) is not closed under multiplication unless we consider the product to be $0$.
As above, we must have the diagrams with $3$ or $4$ parts spanning an ideal of $\mcP_2(\beta)$.
However, we can take one of the products
\[
\begin{tikzpicture}[scale = 0.5,thick, baseline={(0,-1ex/2)}] 
\tikzstyle{vertex} = [shape=circle, minimum size=2pt, inner sep=1pt]
\foreach \i in {1,2} {
    \node[vertex] (G\i) at (\i, 1) [shape=circle, draw] {};
    \node[vertex] (G-\i) at (\i, -1) [shape=circle, draw] {};
}
\draw[-] (G-1) .. controls +(0, .5) and +(0, .5) .. (G-2) -- (G2) .. controls +(0, -.5) and +(0, -.5) .. (G1);
\end{tikzpicture}
\; \cdot \;
\begin{tikzpicture}[scale = 0.5,thick, baseline={(0,-1ex/2)}] 
\tikzstyle{vertex} = [shape=circle, minimum size=2pt, inner sep=1pt]
\foreach \i in {1,2} {
    \node[vertex] (G\i) at (\i, 1) [shape=circle, draw] {};
    \node[vertex] (G-\i) at (\i, -1) [shape=circle, draw] {};
}
\draw[-] (G-2) -- (G2);
\end{tikzpicture}
\; = \;
\begin{tikzpicture}[scale = 0.5,thick, baseline={(0,-1ex/2)}] 
\tikzstyle{vertex} = [shape=circle, minimum size=2pt, inner sep=1pt]
\foreach \i in {1,2} {
    \node[vertex] (G\i) at (\i, 1) [shape=circle, draw] {};
    \node[vertex] (G-\i) at (\i, -1) [shape=circle, draw] {};
}
\draw[-] (G-1) .. controls +(0, .5) and +(0, .5) .. (G-2) -- (G2);
\end{tikzpicture}\,,
\qquad\qquad
\begin{tikzpicture}[scale = 0.5,thick, baseline={(0,-1ex/2)}] 
\tikzstyle{vertex} = [shape=circle, minimum size=2pt, inner sep=1pt]
\foreach \i in {1,2} {
    \node[vertex] (G\i) at (\i, 1) [shape=circle, draw] {};
    \node[vertex] (G-\i) at (\i, -1) [shape=circle, draw] {};
}
\draw[-] (G-2) -- (G2);
\end{tikzpicture}
\; \cdot \;
\begin{tikzpicture}[scale = 0.5,thick, baseline={(0,-1ex/2)}] 
\tikzstyle{vertex} = [shape=circle, minimum size=2pt, inner sep=1pt]
\foreach \i in {1,2} {
    \node[vertex] (G\i) at (\i, 1) [shape=circle, draw] {};
    \node[vertex] (G-\i) at (\i, -1) [shape=circle, draw] {};
}
\draw[-] (G-1) .. controls +(0, .5) and +(0, .5) .. (G-2) -- (G2) .. controls +(0, -.5) and +(0, -.5) .. (G1);
\end{tikzpicture}
\; = \;
\begin{tikzpicture}[scale = 0.5,thick, baseline={(0,-1ex/2)}] 
\tikzstyle{vertex} = [shape=circle, minimum size=2pt, inner sep=1pt]
\foreach \i in {1,2} {
    \node[vertex] (G\i) at (\i, 1) [shape=circle, draw] {};
    \node[vertex] (G-\i) at (\i, -1) [shape=circle, draw] {};
}
\draw[-] (G1) .. controls +(0, -.5) and +(0, -.5) .. (G2) -- (G-2);
\end{tikzpicture}\,,
\]
which is not in the supposed ideal (either left or right).

In fact, we note that the elements in~\eqref{eq:max2_3parts} plus $s_1$ and $\{\{1,2,1',2'\}\}$ (both are defined as ``basis'' elements of $\mcG_2^{(2,2,2)}(\beta)$) are the generators of the full partition algebra $\mcP_2(\beta)$.
Hence, under the usual multiplication, we have $\mcG_2^{(2,2,2)}(\beta) = \mcP_2(\beta)$.
\end{ex}

Another way to recover the important portion of Tanabe's results in~\cite{Tanabe97} would be to say there exists a surjection $\mcG_n^{(r,d,2n)}(m) \to \End_{G(r,d,m)} \natrepr^{\otimes n}$.
Note that $\mcG_n^{(r,d,m)}(\beta) \iso \mcG_n^{(r,d,2n)}(\beta)$ for all $m \geq 2n$.
Furthermore, it is likely that the set of diagrams for $m < 2n$ is the indexing set for a basis, but the multiplication is more complicated than the usual concatenation of diagrams.

Consequently, we will be considering either when $r > n$ or when $m \geq 2n$.
Recall that in these cases, diagrams satisfying Condition~(\ref{cond:d}) can never be satisfied.
Furthermore, we will always take $\mcG_n^{(r,d,m)}(\beta)$ to be a subalgebra.
Now we present our first main result.

\begin{thm}
\label{thm:ccc} 
Suppose that either $r > n$ or $m \geq 2n$.
The $G(r,d,m)$-partition algebra $\mcG^{(r,d,m)}_n(\beta)$ is a cellular algebra.
\end{thm}

\begin{proof}
The same proof as in~\cite[Thm.~4.1]{Xi99} (see also~\cite{GP18}) holds here since the additional conditions makes the allowed set of permutations isomorphic to a direct product of symmetric groups.
Note that the set of such diagrams is invariant under $\iota$.
This gives us the indexing set $\Lambda$ as a $k$-tuple of partitions, where $k$ is the number of factors in the direct product for those diagrams satisfying Condition~(\ref{cond:blocks}).
We take the ordering on $\Lambda$ as a product poset of the graded dominance order.
For any $\lambda \in \Lambda$, the set $M(\lambda)$ is the corresponding set of pairs consisting of a tuple of semistandard Young tableaux of shape $\lambda$ and a set partition.
\end{proof}

Recall that $G(1,1,m) \iso \Sym_m$ and the centralizer is the full partition algebra, which was shown to be cellular for $m \geq 2n$ by Xi~\cite{Xi99} (with a small correction by~\cite{GP18}).
This likely extends to all cases such that $\mcG_n^{(r,d,m)}(\beta)$ is a $\field$-algebra with a potentially more complicated description of $\Lambda$.

Next, we will explore some particular cases in more detail and describe the cell datum and cell modules.

\subsubsection{Uniform block algebra}
\label{sec:uniform_block}

We consider the case of $G(r,d,m)$ when $d = 1$ and $r > n$.

Let $\mcU_n$ denote the \defn{uniform block algebra}, which is a subalgebra of the partition algebra such that the number of top elements equals the number of bottom elements.
We omitted $m$ from our notation as we are primarily interested in the case $m \geq n$, where the description is independent of $m$, and the case $m < n$ is not fundamentally different.
The construction of all simple modules of $\mcU_n$ over $\CC$ was recently done by Orellana, Saliola, Schilling, and Zabrocki~\cite{OSSZ21} when $m \geq n$.
Their representation theoretic results essentially becomes a corollary of Theorem~\ref{thm:ccc} using the framework of Section~\ref{sec:crt}.

Let us examine their results in detail.
We begin with the (sub)algebra of idempotents $\mcI\mcU_n$ whose structure was given by~\cite[Lemma~2.3]{OSSZ21}.
The following proposition is a straightforward consequence and forms the foundation of the constructions of~\cite{OSSZ21}.

\begin{prop}
\label{prop:uniform_idempotent_cellular}
The algebra $\mcI\mcU_k$ is a cellular algebra with cell datum $(\Lambda, \iota, M, C)$ given by
\begin{itemize}
\item $\Lambda$ is all set partitions of $[n]$ with at most $m$ parts under the refinement order with $\{[n]\}$ as the bottom element;
\item $\iota$ becomes the identity;
\item $M(\lambda) = \{\lambda\}$ (so we can ignore it);
\item $C$ is the natural diagram basis.
\end{itemize}
Moreover, the cell modules are all (nonzero) simple modules and one dimensional.
\end{prop}

The main observation is that we cannot have an arbitrary permutation $\sigma \in \Sym_k$ for $k \leq \min(n,m)$, but instead have to preserve the block sizes.
Hence, we remark that~\cite[Prop.~2.5]{OSSZ21} is effectively the decomposition of~\cite[Lemma~4.2]{Xi99}.
Furthermore, we obtain a Young subgroup corresponding to the block sizes, which was the maximal group for the corresponding idempotent in~\cite{OSSZ21}.
This means we could use Proposition~\ref{prop:subcellular} to give an alternative proof via a product of symmetric group algebra cellular bases.

Within each Young subgroup, we have a permutation of multisets of the same size $k$ (contrast this with the partition algebra) representing what can happen with the blocks of size $k$.
This permutation can be represented by a pair of semistandard tableaux whose entries are (disjoint) sets of size $k$, where we use the total ordering given by comparing the smallest element in each set.
Consequently, we see that we can interpret $\Lambda$ as the set of partition tuples $\lambda = (\lambda^{(1)}, \dotsc, \lambda^{(n)})$ such that $\sum_{k=1}^n k \abs{\lambda^{(k)}} = n$.
Furthermore, $M(\lambda)$ becomes the set of semistandard tableau tuples, where the $k$-th element has shape $\lambda^{(k)}$ and is filled with subsets of $[n]$, and the union of all of the entries is $[n]$.
Compare this with Theorem~\ref{thm:partition_cellular}, where we are requiring for each pair $(\rho, T)$ that $\abs{rho} = \abs{\lambda}$ and are further separating into separate blocks based on the size of the parts.

From the above description, we see that the cell modules are precisely the modules given in~\cite{OSSZ21}.
Additionally, the bilinear form $\Phi_{\lambda}$ is given by the corresponding bilinear form of a product of symmetric groups, which is the product of the bilinear forms of the appropriate symmetric group.
Consequently, all of the cell modules are simple when $\field$ has characteristic $0$, which recovers the description in~\cite{OSSZ21}.
Furthermore, we obtain a set of simple modules corresponding over a field of characteristic $p > 0$ when each partition in the tuple $\lambda$ is a $p$-regular partition (\textit{cf.}~\cite[Cor.~4.11]{Xi99}).

\subsubsection{Parity matching algebra}

Recall that the group of signed permutations of rank $m$, which is the Weyl group of $\Or_{2m+1}(\CC)$ or the Coxeter group of type $B_m$, and corresponds to $G(2,1,m)$.
The centralizer algebra was studied in more detail in~\cite{Orellana05,Orellana07}, where it was related to the colored partition algebra introduced by Bloss~\cite{Bloss03}.
From~\cite[Lemma~2.1]{Tanabe97}, the diagrams for this algebra are given by the set partitions of $[n]$ and $[n]'$ such that
\begin{enumerate}[(I)]
\item the blocks of even (resp.\ odd) size connect to blocks of even (resp.\ odd) size;
\item all odd blocks must be connected; and
\item there are at most $m$ such connected components in the resulting diagram. \label{cond:block_sizes}
\end{enumerate}
In particular, the number of odd sized blocks must be the same for $[n]$ and $[n]'$.
If we have $n \leq m$, the last condition~(\ref{cond:block_sizes}) is always satisfied and $\mcG_n^{(2,1,m)}(\beta)$ is a (unital) $\field$-algebra.
Hence, assuming $m \geq n$, we call the subalgebra $\mcP_n(\beta)$ spanned by these diagrams the \defn{parity matching algebra} and denote it by $\mcP\mcM_n(\beta)$ since it becomes independent of $m$ like the uniform block algebra.
Furthermore, $\mcP\mcM_n(\beta)$ is cellular by Theorem~\ref{thm:ccc}.

An alternative description of the basis of $\mcP\mcM_n(\beta)$ consists of all diagrams $\rho = \{\rho_1, \dotsc, \rho_{\ell}\}$ that are even set partitions, that is, we have $\abs{\rho_i} \equiv 0 \pmod{2}$ for all $i$.
Hence, another name for this algebra could be the ``even partition algebra.''
We can see this is indeed a $\field$-subalgebra (\textit{i.e.}, closed under multiplication) by a straightforward parity argument.
As a consequence, we obtain the recursion formula (see, \textit{e.g.},~\cite[Sec.~5]{Orellana07}) for its dimension
\[
\dim \mcP\mcM_n(\beta) = \sum_{i=1}^n \binom{2k-1}{2i-1} \dim \mcP\mcM_{n-i}(\beta)
\]
by removing the block containing $n$ from each diagram.
The sequence of dimensions is~\cite[A005046]{OEIS}, which contains further combinatorial interpretations and the closed formula
\[
\dim \mcP\mcM_n(\beta) = \sum_{k=1}^{2n} \sum_{i=0}^{k-1} (-1)^i \frac{(i-k)^{2n} }{2^{k-1} k!} \binom{2k}{i}.
\]

Another similarity to the uniform block algebra is the even and odd blocks cannot interact.
Hence we can use the Young subgroup of $\Sym_{k_1} \times \Sym_{k_2} \subseteq \Sym_k$, where there are $k_1$ (resp.~$k_2$) odd (resp.\ even) with $k = k_1 + k_2$, to deconstruct our diagrams.

The cell modules of $\mcP\mcM_n(\beta)$ are parameterized by pairs of partitions $(\mu, \nu)$ such that
\begin{subequations}
\label{eq:PM_conditions}
\begin{align}
\abs{\mu} + 2\abs{\nu} & \leq n, \label{eq:PM_min_sizes}
\\ n - \abs{\mu} & \equiv 0 \pmod{2}. \label{eq:PM_even_disconnect}
\end{align}
\end{subequations}
Recall this is the set $\Lambda$.
The inequality~\eqref{eq:PM_min_sizes} comes from partitioning the blocks into the even and odd propagating blocks, which must have size at least $2$ and $1$ respectively.
The condition~\eqref{eq:PM_even_disconnect} is that the non-defect blocks must have even size.

The indexing set $M(\lambda)$ for the basis of the irreducible representation $W(\lambda)$ for $\lambda = (\mu, \nu) \in \Lambda$ is a pair of standard tableaux of shapes $(\mu, \nu)$ such that $\mu$ (resp.~$\nu$) are filled with subsets of $[n]$ of odd (resp.~even) size and a set partition of the remaining letters with all blocks having even size such that the disjoint union of all of the entries is $[n]$.
For the element $v_k$ in the construction~\eqref{eq:specht_decomposition}, we can have very propagating block (which is fixed by the choice of $\lambda$) have size $1$ or $2$ in some fixed order.
By the condition~\eqref{eq:PM_min_sizes}, such an element exists.

To describe the dimensions of the cell modules, we need some combinatorial data.
Let $O_{n,k}$ (resp.\ $E_{n,k}$) denote the number of set partitions $\rho = \{\rho_1, \dotsc, \rho_k\}$ (so exactly $k$ parts) of $[n]$ such that $\abs{\rho_i} \equiv 1 \pmod{2}$ (resp.\ $\abs{\rho_i} \equiv 0 \pmod{2}$) for all $i$.
These are the set of odd (resp.\ even) set partitions with exactly $k$ parts, and $O_{n,k}$ is the sequence~\cite[A136630]{OEIS} (resp.~\cite[A156289]{OEIS} for $E_{2n,k}$).
We note the formulas
\begin{align*}
O_{n,k} & = O_{n-2,k-2} + k^2 O_{n-2,k},
&
E_{2n,k} & = (2k-1) E_{2n-2,k-1} + k^2 E_{2n-2,k},
\\
O_{n,k} & = \frac{1}{2^k k!} \sum_{j=0}^k (-1)^{k-j} \binom{k}{j} (2j - k)^n,
&
E_{2n,k} & = \frac{2}{2^k k!} \sum_{j=1}^k (-1)^{k-j} \binom{2k}{k-j} j^n.
\end{align*}
Note that $E_{2n+1,k} = 0$ and $O_{n,k} = 0$ whenever $n + k \equiv 1 \pmod{2}$ by parity arguments.
The $O_{n,k}$ recurrence relation can be proven combinatorially from two cases:
We have a singleton $\{n\}$, which necessarily means there is another singleton, and removing both of these singletons yields $O_{n-2,k-2}$ (it does not affect the result which other singleton we remove by renaming).
Otherwise we need to remove another element from the part containing $n$.
To reconstruct such a part, for some fixed $\nu$ contributing to $O_{n-2,k}$, we choose a part $\nu_i$ in some $\nu \in O_{n-2,k}$ to add $n$ to and choose another part to correspond to $n-1 \leftrightarrow \min \nu_j - 1/2$ by renaming and reordering.
A combinatorial proof for the $E_{n,k}$ recurrence is similar.
While an explicit combinatorial proof does not seem to be in the literature, these proofs sketched above are likely known to experts.

For $\lambda = (\mu, \nu) \in \Lambda$, from the combinatorial description of $M(\lambda)$, we see that
\begin{subequations}
\label{eq:dim_cell_PM}
\begin{align}
\dim W(\lambda) & = f_{\mu} f_{\nu} \sum_{i=k_1}^n \binom{n}{i} O_{i,k_1} \sum_{j=k_2}^{\lfloor (n - i) / 2 \rfloor} \binom{j}{k_2} E_{n-i,j}
\\ & = f_{\mu} f_{\nu} \sum_{i=k_2}^{\lfloor (n-k_1)/2 \rfloor} \binom{n}{2i} \sum_{j=k_2}^i \binom{j}{k_2} E_{2i,j} \cdot O_{n-2i,k_1},
\end{align}
\end{subequations}
where $k_1 = \abs{\mu}$ and $k_2 = \abs{\nu}$.
Indeed, for the first equality, the first sum comes from choosing the elements for the odd sized blocks (all of which must be defects), the second sum comes from choosing exactly $k_2$ of the even sized set partitions to be defects.
The second equality is similar but first choosing the elements for the even sized blocks.
Furthermore, we have
\begin{align*}
\dim \mcP\mcM_n(\beta) & = \sum_{\lambda \in \Lambda} \bigl( \dim W(\lambda) \bigr)^2
\\ & = \sum_{k_1=0}^n \sum_{k_2=0}^{\rfloor (n-k_1) / 2 \lfloor} k_1! k_2! \left( \sum_{i=k_1}^n \binom{n}{i} O_{i,k_1} \sum_{j=k_2}^{\lfloor (n - i) / 2 \rfloor} \binom{j}{k_2} E_{n-i,j} \right)^2
\\ & = \sum_{k_1=0}^n \sum_{k_2=0}^{\rfloor (n-k_1) / 2 \lfloor} k_1! k_2! \left( \sum_{i=k_2}^{\lfloor (n-k_1)/2 \rfloor} \binom{n}{2i} \sum_{j=k_2}^i \binom{j}{k_2} E_{2i,j} \cdot O_{n-2i,k_1} \right)^2
\end{align*}
from Equation~\eqref{eq:dim_cell_PM}, where we have used Equation~\eqref{eq:fla=n2} to obtain the final two formulas.

By the double centralizer theorem (see, \textit{e.g.},~\cite{EGHLSVY11}), the irreducible representations of $\mcP\mcM_n(m)$ are in bijection with those irreducible $G(2,1,m)$-representations appearing in the decomposition of $\natrepr^{\otimes n}$.
All irreducible representations of $G(2,1,m)$ which are indexed by pairs of partitions $(\mu, \widetilde{\nu})$ such that $\abs{\mu} + \abs{\widetilde{\nu}} = m$ (this was attributed to Specht~\cite{Specht32} in~\cite{Ram97}).
We can see by directly counting that we do not obtain all irreducible representations of $G(2,1,n)$ when $m \geq n$.

\begin{ex}
The basis for $\mcP\mcM_2(\beta)$ is given by four diagrams
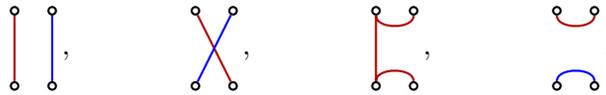
\ytableausetup{boxsize=1.1em}
\[
\begin{tikzpicture}[scale = 0.5,thick, baseline={(0,-1ex/2)}] 
\tikzstyle{vertex} = [shape=circle, minimum size=2pt, inner sep=1pt]
\foreach \i in {1,2} {
    \node[vertex] (G\i) at (\i, 1) [shape=circle, draw] {};
    \node[vertex] (G-\i) at (\i, -1) [shape=circle, draw] {};
}
\draw[-,color=darkred] (G1) -- (G-1);
\draw[-,color=blue] (G2) -- (G-2);
\end{tikzpicture}\,,
\qquad\qquad
\begin{tikzpicture}[scale = 0.5,thick, baseline={(0,-1ex/2)}] 
\tikzstyle{vertex} = [shape=circle, minimum size=2pt, inner sep=1pt]
\foreach \i in {1,2} {
    \node[vertex] (G\i) at (\i, 1) [shape=circle, draw] {};
    \node[vertex] (G-\i) at (\i, -1) [shape=circle, draw] {};
}
\draw[-,color=darkred] (G1) -- (G-2);
\draw[-,color=blue] (G2) -- (G-1);
\end{tikzpicture}\,,
\qquad\qquad
\begin{tikzpicture}[scale = 0.5,thick, baseline={(0,-1ex/2)}] 
\tikzstyle{vertex} = [shape=circle, minimum size=2pt, inner sep=1pt]
\foreach \i in {1,2} {
    \node[vertex] (G\i) at (\i, 1) [shape=circle, draw] {};
    \node[vertex] (G-\i) at (\i, -1) [shape=circle, draw] {};
}
\draw[-,color=darkred] (G2) .. controls +(0, -.5) and +(0, -.5) .. (G1) --
         (G-1) .. controls +(0., .5) and +(0, .5) .. (G-2);
\end{tikzpicture}\,,
\qquad\qquad
\begin{tikzpicture}[scale = 0.5,thick, baseline={(0,-1ex/2)}] 
\tikzstyle{vertex} = [shape=circle, minimum size=2pt, inner sep=1pt]
\foreach \i in {1,2} {
    \node[vertex] (G\i) at (\i, 1) [shape=circle, draw] {};
    \node[vertex] (G-\i) at (\i, -1) [shape=circle, draw] {};
}
\draw[-,color=darkred] (G1) .. controls +(0, -.5) and +(0, -.5) .. (G2);
\draw[-,color=blue] (G-1) .. controls +(0, .5) and +(0, .5) .. (G-2);
\end{tikzpicture}\,,
\]
which indeed forms a $\field$-algebra.
The set $\Lambda$ is given by the following pairs of partitions
\[
\left( \ydiagram{2}, \emptyset \right),
\qquad
\left( \ydiagram{1,1}, \emptyset \right),
\qquad
\left( \emptyset, \ydiagram{1} \right),
\qquad
(\emptyset, \emptyset),
\]
which are each one-dimensional modules with the corresponding semistandard tableaux
\[
\left( \ytableaushort{12}, \emptyset \right),
\qquad
\left( \ytableaushort{1,2}, \emptyset \right),
\qquad
\left( \emptyset, \ytableaushort{{12}} \right),
\qquad
(\emptyset, \emptyset),
\]
However, there are 5 irreducible representations of $G(2,1,2)$ corresponding to the pair of partitions
\[
\left( \ydiagram{2}, \emptyset \right),
\qquad
\left( \ydiagram{1,1}, \emptyset \right),
\qquad
\left( \emptyset, \ydiagram{1,1}\right),
\qquad
\left( \emptyset, \ydiagram{2} \right),
\qquad
\left( \ydiagram{1}, \ydiagram{1} \right),
\]
where the last one corresponds to a $2$ dimensional representation.
\end{ex}

\subsubsection{Colored permutation symmetrizer}

The remaining case corresponding to colored permutations $G(r,1,m)$ with $2 < r \leq n$ is essentially the same as $r = 2$.
The index set will be a subset of all $r$-tuples of partitions that satisfy the general form of the conditions~\eqref{eq:PM_conditions} for $\lambda =(\lambda^{(1)}, \dotsc, \lambda^{(r)})$:
There exists a $k_0 \in \ZZ_{\geq 0}$ such that
\begin{align*}
r k_0 + \sum_{j=1}^r j \abs{\lambda^{(j)}} & = 0,
\end{align*}
with $\lambda$ again coming from encoding the propagating blocks and $k_0$ approximately counting the non-propagating blocks (which must necessarily be a multiple of $r$; hence $k_0$ is a sum of these factors).
Furthermore, such a diagram exists independent of $m$ as we can always combine the blocks of size $r$ with one of those for $\lambda^{(j)}$ for any $j$ and since all of these are connected with the defect blocks, they do not contribute any additional connected components.
The set $M(\lambda)$ is the analogous tuple of semistandard tableaux of shape $\lambda$ with $\lambda^{(j)}$ being filled with multisets of size $j$ and a set partition with all the blocks having size equivalent to $0$ modulo $r$.
As before, all entries must be disjoint.

\subsection{Brauer algebra}

The \defn{Brauer algebra} $\mcB_n(\beta)$ is the subalgebra of the partition algebra $\mcP_n(\beta)$ spanned by all diagrams with parts that have exactly size $2$.
It was introduced by Brauer~\cite{Brauer37} coming from Schur--Weyl duality using the orthognal group, and it is known to be a cellular algebra~\cite{GL96} (see also~\cite{Xi99}) of dimension
\[
\dim \mcB_n(\beta) = \frac{(2n)!}{2^n n!} = (2n - 1)!!.
\]
Furthermore, from our description, we can easily see that the cell modules (which are Brauer's Specht modules) have dimension
\[
\dim W(\lambda) = \binom{n}{\abs{\lambda}} (n - \abs{\lambda} - 1)!! \dim V_{\Sym_{\abs{\lambda}}}(\lambda),
\]
where $\lambda$ is a partition such that $\abs{\lambda} \leq n$ and $\abs{\lambda} \equiv n \pmod{2}$.
Indeed, we select $\abs{\lambda}$ propagating (necessarily) singletons from $[n]$, then we are left choosing a Brauer diagram on the renaming $n - \abs{\lambda}$ such nodes; the usual Specht module $V_{\Sym_{\abs{\lambda}}}(\lambda)$ comes from Equation~\eqref{eq:specht_decomposition}.
When $\field$ has characteristic $0$, Wenzl showed~\cite{Wenzl1988} it is semisimple whenever $\beta \in \field \setminus \{0, \pm 1, \dotsc, \pm n\}$.
The representations of $\mcB_n(\beta)$ when it is not semisimple were studied by Doran, Hanlon, and Wales~\cite{DWH99} and Cox, De Visscher, and Martin~\cite{CDVM09,CDVM09II}.
The classification of when $\mcB_n(\beta)$ is semisimple for the general characteristic case was shown by Rui and Si~\cite{Rui05,RS06}.

\subsection{Rook Brauer algebra}
\label{sec:rook_brauer}

The \defn{rook Brauer algebra} $\mcR\mcB_n(\beta)$ (also known as the partial Brauer algebra) is defined as the subalgebra of $\mcP_n(\beta)$ spanned by all diagrams such that the parts have size at most $2$.
This has been studied in~\cite{HdM14,MM14} with two different descriptions of their irreducible representations.
The monoid algebra version (so $\beta = 1$) was also studied in~\cite{DEG17}.
Furthermore, we have the following result.

\begin{thm}[{\cite{MM14}}]
\label{thm:rook_Brauer}
The rook Brauer algebra is a cellular algebra. Moreover, it is Morita equivalent to
\[
\mcR\mcB_n(\beta) \simeq \mcB_n(\beta - 1) \oplus \mcB_{n-1}(\beta - 1)
\]
for $\beta - 1, \beta \neq 0$.
\end{thm}

Theorem~\ref{thm:rook_Brauer} gives us one method to describe all of the irreducible representations.
However, we could also undertake the same half diagram analysis like for the Brauer algebra to determine the dimensions of its cell modules:
\[
\dim W(\lambda) = f_{\lambda} \binom{n}{k} \sum_{m=0}^{\lfloor (n-k)/2 \rfloor} \binom{n-k}{2m} (2m-1)!!,
\]
where $k = \abs{\lambda}$, by choosing the positions of the singletons.

We can allow $\beta = 0$ in Theorem~\ref{thm:rook_Brauer} if we slightly change the multiplication to only count the number of loops removed in the product, not paths contractible to a point.
In fact, this leads to a more general definition of the rook Brauer algebra (and its subalgebras) using two parameters $\mcR\mcB_n(\beta,\gamma)$, where $\beta$ counts the number of loops removed and $\gamma$ counts the number of contractible paths.
Thus $\mcR\mcB_n(\beta) = \mcR\mcB_n(\beta, \beta)$, and furthermore, we have isomorphic algebras $\mcR\mcB_n(\beta, \beta) \iso \mcR\mcB_n(\beta, \gamma)$ whenever $\beta, \gamma \neq 0$~\cite{HdM14,MM14}.

\begin{ex}
\label{ex:nonassoc_product}
There is a misprint in the definition of the diagram multiplication in~\cite{HdM14}, as it is not sufficient to only consider singletons for the second parameter $\gamma$.
Indeed, if we take this as the definition, then we do not have an associative algebra for general $\gamma$:
\[
\gamma^2 \,
\begin{tikzpicture}[scale = 0.5,thick, baseline={(0,-1ex/2)}] 
\tikzstyle{vertex} = [shape=circle, minimum size=2pt, inner sep=1pt]
\foreach \i in {1,2} {
    \node[vertex] (Gp\i) at (\i, 3) [shape=circle, draw] {};
    \node[vertex] (Gm\i) at (\i, -3) [shape=circle, draw] {};
}
\draw[-,color=darkred] (Gp1) .. controls +(0, -.5) and +(0, -.5) .. (Gp2);
\draw[-,color=UQpurple] (Gm1) .. controls +(0, .5) and +(0, .5) .. (Gm2);
\end{tikzpicture}
\quad = \quad
\begin{tikzpicture}[scale = 0.5,thick, baseline={(0,-1ex/2)}] 
\tikzstyle{vertex} = [shape=circle, minimum size=2pt, inner sep=1pt]
\foreach \i in {1,2} {
    \node[vertex] (Gp\i) at (\i, 3) [shape=circle, draw] {};
    \node[vertex] (G-\i) at (\i, -1) [shape=circle, draw] {};
    \node[vertex] (Gm\i) at (\i, -3) [shape=circle, draw] {};
}
\draw[-,color=darkred] (Gp1) .. controls +(0, -.5) and +(0, -.5) .. (Gp2);
\draw[-,color=UQpurple] (Gm1) .. controls +(0, .5) and +(0, .5) .. (Gm2);
\end{tikzpicture}
\quad = \quad
\begin{tikzpicture}[scale = 0.5,thick, baseline={(0,-1ex/2)}] 
\tikzstyle{vertex} = [shape=circle, minimum size=2pt, inner sep=1pt]
\foreach \i in {1,2} {
    \node[vertex] (Gp\i) at (\i, 3) [shape=circle, draw] {};
    \node[vertex] (G\i) at (\i, 1) [shape=circle, draw] {};
    \node[vertex] (G-\i) at (\i, -1) [shape=circle, draw] {};
    \node[vertex] (Gm\i) at (\i, -3) [shape=circle, draw] {};
}
\draw[-,color=darkred] (Gp1) .. controls +(0, -.5) and +(0, -.5) .. (Gp2);
\draw[-,color=blue] (G1) .. controls +(0, .5) and +(0, .5) .. (G2) -- (G-1);
\draw[-,color=UQpurple] (Gm1) .. controls +(0, .5) and +(0, .5) .. (Gm2);
\end{tikzpicture}
\quad = \quad \gamma \,
\begin{tikzpicture}[scale = 0.5,thick, baseline={(0,-1ex/2)}] 
\tikzstyle{vertex} = [shape=circle, minimum size=2pt, inner sep=1pt]
\foreach \i in {1,2} {
    \node[vertex] (Gp\i) at (\i, 3) [shape=circle, draw] {};
    \node[vertex] (G\i) at (\i, 1) [shape=circle, draw] {};
    \node[vertex] (Gm\i) at (\i, -3) [shape=circle, draw] {};
}
\draw[-,color=darkred] (Gp1) .. controls +(0, -.5) and +(0, -.5) .. (Gp2);
\draw[-,color=blue] (G1) .. controls +(0, .5) and +(0, .5) .. (G2);
\draw[-,color=UQpurple] (Gm1) .. controls +(0, .5) and +(0, .5) .. (Gm2);
\end{tikzpicture}
\quad = \quad \beta \gamma \,
\begin{tikzpicture}[scale = 0.5,thick, baseline={(0,-1ex/2)}] 
\tikzstyle{vertex} = [shape=circle, minimum size=2pt, inner sep=1pt]
\foreach \i in {1,2} {
    \node[vertex] (Gp\i) at (\i, 3) [shape=circle, draw] {};
    \node[vertex] (Gm\i) at (\i, -3) [shape=circle, draw] {};
}
\draw[-,color=darkred] (Gp1) .. controls +(0, -.5) and +(0, -.5) .. (Gp2);
\draw[-,color=UQpurple] (Gm1) .. controls +(0, .5) and +(0, .5) .. (Gm2);
\end{tikzpicture}
\]
\end{ex}

Given the two parameters, we have the following statement and problem (\textit{cf.}~\cite{MM14,DG22II}).

\begin{prop}
\label{prop:two_param_rook_brauer_iso}
For all $\gamma \neq 0$, we have $\mcR\mcB_n(0, \gamma) \iso \mcR\mcB_n(0, 1)$.
\end{prop}

\begin{proof}
Same as the case $\beta \neq 0$.
\end{proof}

\begin{problem}
\label{prob:two_param_rook_brauer_semisimple}
Determine when $\mcR\mcB_n(0, \gamma)$ and $\mcR\mcB_n(\beta, 0)$ are semisimple.
\end{problem}

We remark that the second parameter can never occur in the Brauer algebra as we can only have loops in the product.

\subsection{Rook algebra}

The \defn{rook algebra} $\mcR_n(\beta) \subseteq \mcP_n(\beta)$ is the set of diagrams $\rho$ such that every block has size at most $2$ and $N(\rho), N'(\rho) \leq 1$.
In other words, every block of size $2$ must be propagating.
An alternative description is in terms of partial permutations of $[n]$, which are injective maps $\pi \colon D \to [n]'$ for some $D \subseteq [n]$.
Furthermore, by rescaling by $\beta^{-I}$, where $I$ is the number of isolated in one half of the diagram (note that we necessarily have the number of isolated vertices on each side being equal), we can see $\mcR_n(\beta) \iso \mcR_n(1)$ for all $\beta \neq 0$.

This was first introduced and the irreducible representations studied by Munn~\cite{Munn57}\footnote{Munn in~\cite{Munn57} considered it as a monoid algebra, and so $\beta = 1$.} and later refined by Solomon~\cite{Solomon02}.
Like for the (rook) Brauer algebra, we have $\Lambda$ as for the partition algebra, but $M(\lambda)$ consists only of the standard Young tableau with entries consisting of a single entry.
This is just like the classical case when we restrict to a specific subset of $[n]$ of size $\abs{\lambda}$, which correspond to the propagating blocks.

As a consequence, we obtain that
\begin{equation}
\label{eq:rook_dim}
\dim \mcR_n(\beta) = \sum_{k=0}^n \sum_{\mu \vdash k} \binom{n}{\abs{\mu}}^2 f_{\mu}^2 = \sum_{k=0}^n \binom{n}{k}^2 k!,
\end{equation}
where $f_{\mu}$ is equal to the number of standard Young tableau on $[k]$ of shape $\mu \vdash k$.
We have factored out the choice of subset in the binomial coefficient and used Equation~\eqref{eq:fla=n2}.
The dimension formula in Equation~\eqref{eq:rook_dim} is also known from the original definition of the rook monoid through a straightforward combinatorial argument; see, \textit{e.g.},~\cite[Eq.~(1.2)]{Solomon02}.

We remark that we cannot have any loops to remove in the product of the rook algebra.
As such, we have $\mcR_n(\gamma) \subseteq \mcR\mcB_n(\beta,\gamma)$ and $\beta$ is not involved.

In~\cite{DG21}, it was shown that the centralizer algebra of the $n$-fold tensor power of the (unreduced) Burau representation $\mathbf{F}$ of the braid group $\mathbf{B}_m$ with $m$ strands is isomorphic to $\mcR_n([m]_q)$ for $q$ not a root of unity and $m > 2n$.
It is generated by $\sigma_i$ being the simple crossing of the $i$-th strand over the $(i+1)$-th strand, and these satisfy the braid relations $\sigma_i \sigma_{i+1} \sigma_i = \sigma_{i+1} \sigma_i \sigma_{i+1}$.
However, the braid group has a number of interesting subgroups:
\begin{itemize}
\item The affine braid group $\widetilde{\mathbf{B}}_{m-1}$ of $m- 1 $ braids in a cylinder by considering the center hole to be an additional (fixed) strand.
\item If $m = 2M - 1$ is odd, the type $B_M$ Artin group is $\langle \sigma_1 \sigma_m, \sigma_2 \sigma_{m-1}, \dotsc, \sigma_{M-1} \sigma_{M+1}, \sigma_M \rangle$ as it can be easily checked that $\tau_{M-1} \sigma_M \tau_{M-1} \sigma_M  = \sigma_M  \tau_{M-1} \sigma_M \tau_{M-1}$ for $\tau_{M-1} := (\sigma_{M-1} \sigma_{M+1})$.
\item The pure braid group $\mathbf{PB}_m$, which can be defined by the short exact sequence $0 \to \mathbf{PB}_m \to \mathbf{B}_m \to \Sym_m \to 0$ using the natural projection.
\item The subgroup $\mathbf{R}_m = \langle \sigma_i^2 \rangle$, which is a right-angled Artin group (also known as a partially commutative group) of a line\footnote{Some conventions in the literature use the complement graph.} by~\cite{Collins94,CP01,Humphries94}.
\end{itemize}

There are other subgroups of these subgroups.
For example, the affine braid group $\widetilde{\mathbf{B}}_{m-1}$ has the pure affine braid group $\widetilde{\mathbf{PB}}_{m-1}$ with the short exact sequence $0 \to \widetilde{\mathbf{PB}}_{m-1} \to \widetilde{\mathbf{B}}_{m-1} \to \widetilde{\Sym}_{m-1} \to 0$, where $\widetilde{\Sym}_{m=1}$ is the affine symmetric group; the right-angled Artin group of a circle generated by the square of generators of $\widetilde{\mathbf{B}}_{m-1}$; and the affine type $\widetilde{C}_M$ braid group (constructed similarly to the type $B_M$ inside $\mathbf{B}_{2M-1}$).
Right-angled Artin groups are known to have many interesting subgroups~\cite{BB97}, and $\mathbf{R}_m$ naturally projects onto the twin group studied in~\cite{DG22II}.
(Some surveys on general (right-angled) Artin groups are~\cite{Charney07,McCammond17}.)

\begin{problem}
Determine the centralizer algebra for the subgroups of the braid group mentioned above acting on $\mathbf{F}^{\otimes n}$.
Furthermore, classify the subgroups where the centralizer algebra is isomorphic to a subalgebra of $\mcP_n([m]_q)$.
\end{problem}

It is possible that the techniques used in~\cite{DG22II} can be applied to any projection from an Artin group to its corresponding Coxeter group, where we consider the Burau representation as Hecke algebra representation.

\section{Planar algebras}

A diagram is \defn{planar} if it can be drawn without crossings.
A \defn{planar algebra} is a subalgebra of the partition algebra if it has a basis given by planar diagrams.
We will sometimes use the term ``noncrossing'' instead of ``planar'' in this manuscript.
For any planar algebra, we necessarily have for each diagram the corresponding permutation $\sigma = 1$, but this does not completely classify planar diagrams in general.
Furthermore, since the diagrams are  noncrossing, the bilinear form of~\cite[Lemma~4.3]{Xi99} maps to the $\field$-span of the trivial permutation.
Thus, each diagram is an element of the cellular basis for $\mcP_n(\beta)$ given in Theorem~\ref{thm:partition_cellular}.
Proposition~\ref{prop:subcellular} yields the following.

\begin{thm}
\label{thm:planar_cellular}
Any planar algebra on $[n] \sqcup [n]'$ fixed under $\iota$ is cellular with a cell basis given by the natural diagram basis and $\Lambda$ being a graded poset, where the grading is the number of defects.
Moreover, there is a unique top element in $\Lambda$ of rank $n$ corresponding to $1 \in \mcP$, and if there exists an element of rank $0$ in $\Lambda$, it is unique.
\end{thm}

Note that the element of rank $n$ corresponds to half of the identity element.
For the element of rank $0$, fix some set partition
$\rho_0$ (resp. $\rho_0'$) of $[n]$ (resp.~$[n]'$) with $0$ defects.
Then we can compose any $0$-defect set partition $\nu$ of $[n]'$ with $\{\rho_0, \rho_0'\}$, and the result is $\rho_0'$.

Consequently, we will always take the trivial $\field [\Sym_k]$ module to build the cell modules.
Thus, in each case where we restrict to the planar subalgebra, we can define the basis of the cell modules by the corresponding noncrossing set partitions (with some fixed number of designated propagating blocks).
Furthermore, for generic $\beta$, the pictorial description of bilinear form being $0$ is also a sufficient condition as in the resulting diagram $\Delta$, the result is simply $\beta^D$, where $D$ is the number of interior components (\textit{i.e.}, not containing a matching pair of defects).

\begin{ex}
If we have the pairing of $\Phi_{\lambda}(C_U, C_T)$ given pictorally as
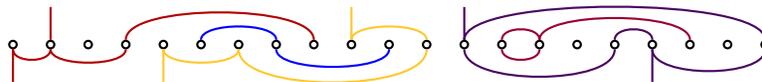
\[
\begin{tikzpicture}[scale = 0.5,thick, baseline=10pt] 
\tikzstyle{vertex} = [shape=circle, minimum size=2pt, inner sep=1pt]
\foreach \i in {1,...,21} {
    \node[vertex] (G\i) at (\i, 2) [shape=circle, draw] {};
}
\draw[-,color=darkred] (G1) + (0,-1) -- (G1) .. controls +(0, -.5) and +(0, -.5) .. (G2) .. controls +(0, -.6) and +(0, -.6) .. (G4) .. controls +(0, 1.1) and +(0, 1.1) .. (G9);
\draw[-,color=darkred] (G2) -- ++(0,1);
\draw[-,color=UMNgold] (G5) + (0,-1) -- (G5) .. controls +(0, -.6) and +(0, -.6) .. (G7) .. controls +(0, -1.3) and +(0, -1.3) .. (G12);
\draw[-,color=UMNgold] (G12) .. controls +(0, .6) and +(0, .6) .. (G10) -- ++(0,1);
\draw[-,color=blue] (G6) .. controls +(0, .6) and +(0, .6) .. (G8) .. controls +(0, -.75) and +(0, -.75) .. (G11);
\draw[-,color=UQpurple] (G13) + (0,1) -- (G13) .. controls +(0, 1.3) and +(0, 1.3) .. (G21)  .. controls +(0, -.9) and +(0, -.9) .. (G18) .. controls +(0, .5) and +(0, .5) ..  (G17) .. controls +(0, -1.1) and +(0, -1.1) .. (G13);
\draw[-,color=UQpurple] (G18) -- ++(0,-1);
\draw[-,color=OCUenji] (G15) .. controls +(0, .5) and +(0, .5) .. (G14) .. controls +(0, -.5) and +(0, -.5) .. (G15)  .. controls +(0, .9) and +(0, .9) .. (G19);
\end{tikzpicture}
\]
then we have $\Phi_{\lambda}(C_U, C_T) = \beta^5$.
\end{ex}

We note that the description of the cell modules is essentially the same as the construction given in~\cite[Thm.~3.21]{HJ20} for the planar algebras considered there.
In turn, this yields the tableaux description by following the construction in Section~\ref{sec:crt}.

\subsection{Temperley--Lieb and planar partition algebras}

We begin with the Temperley--Lieb algebra $\mcT\mcL_n(\beta)$~\cite{TL71}, which was surveyed in~\cite{RSA14}.
This is spanned by the set of diagrams such that all parts in the set partition have size $2$, which is also called a perfect matching and it counted by the even Catalan numbers
\begin{equation}
\label{eq:TL_dim}
\dim \mcT\mcL_n(\beta) = C_N = \frac{1}{N+1}\binom{2N}{N},
\end{equation}
where $N = 2n$.
Thus, we see the cell modules built from the partition algebra are a reformulation of the link modules~\cite{GL96} (see also Remark~\ref{rem:assoc_graded}) with the simple modules indexed by $\Lambda = \{n - 2k \mid k \in \langle n \rangle \}$ under the natural order, where $\langle n \rangle := \{0,1,\dotsc,\lfloor n/2 \rfloor\}$.\footnote{In~\cite{GL96,RSA14}, instead of the $\Lambda$ given here, the authors used the set $\langle n \rangle$ to index the cell modules, which correspond to the number of defects instead of the number of non-defect entries as described here.}
Precise conditions on $\beta$ when $\mcT\mcL_n(\beta)$ is semisimple can be found in, \textit{e.g.},~\cite[Thm~4.7]{RSA14} and~\cite[Thm.~A.1]{DG22} (which interprets results in~\cite{GdlHJ89}).

For any half diagram for $M(\lambda)$, the number of defects is equal to $\lambda$.
Furthermore, we see that the dimensions of the cell modules are the triangle Catalan numbers
\[
\dim W(\lambda) = \abs{M(\lambda)} = C_{N,k} = \binom{N+k}{k} - \binom{N+k}{k-1} = \frac{N+k-1}{N+1} \binom{N+k}{N},
\]
where $\lambda = 2n - k$ and $N = n - k$.
In particular, when $\lambda = 0,1$, this is precisely the $\lfloor n/2 \rfloor$-th Catalan number.
We also have that $\dim W(\lambda)$ counts the number of Dyck paths of length $n$ that end at height $\lambda$.
Furthermore, Equation~\eqref{eq:dim_formula} yields the identity
\begin{equation}
\label{eq:catalan_square}
C_{2n} = \sum_{k=0}^{\lfloor n/2 \rfloor} C_{n-k,k}^2.
\end{equation}

Next, we consider the planar partition algebra $\mcP\mcP_n(\beta)$.
For $\beta \neq 0$, this reduces to the case of the Temperley--Lieb algebra by the following well-known result (see, \textit{e.g.},~\cite{Jones94,HR05}).

\begin{thm}
\label{thm:TL_planer_iso}
For $\beta \neq 0 $, we have
\[
\mcP\mcP_n(\beta^2) \iso \mcT\mcL_{2n}(\beta).
\]
\end{thm}

Roughly speaking, the isomorphism is to consider the outline of a thickened version of the planar partition algebra.
Making this precise, for $\beta \neq 0$, the isomorphism is given by
\begin{align*}
p_i =
\begin{tikzpicture}[scale = 0.5,thick, baseline={(0,-1ex/2)}] 
\tikzstyle{vertex} = [shape=circle, minimum size=2pt, inner sep=1pt]
\foreach \i in {1,3,5,6,8} {
    \node[vertex] (G\i) at (\i, 1) [shape=circle, draw] {};
    \node[vertex] (G-\i) at (\i, -1) [shape=circle, draw] {};
    \draw[-] (G\i) -- (G-\i);
}
\node[vertex] (G4) at (4, 1) [shape=circle, draw] {};
\node[vertex] (G-4) at (4, -1) [shape=circle, draw] {};
\draw (2,1) node {$\cdots$};
\draw (2,-1) node {$\cdots$};
\draw (7,1) node {$\cdots$};
\draw (7,-1) node {$\cdots$};
\end{tikzpicture}
& \mapsto
\beta \cdot
\begin{tikzpicture}[scale = 0.5,thick, baseline={(0,-1ex/2)}] 
\tikzstyle{vertex} = [shape=circle, minimum size=2pt, fill=red, inner sep=1pt]
\foreach \i in {1,3,5,6,8} {
    \node[vertex] (G\i) at (\i-.2, 1) [shape=circle, draw] {};
    \node[vertex] (Gp\i) at (\i+.2, 1) [shape=circle, draw] {};
    \node[vertex] (G-\i) at (\i-.2, -1) [shape=circle, draw] {};
    \node[vertex] (Gp-\i) at (\i+.2, -1) [shape=circle, draw] {};
    \draw[-] (G\i) -- (G-\i);
    \draw[-] (Gp\i) -- (Gp-\i);
}
\node[vertex] (G4) at (4-.2, 1) [shape=circle, draw] {};
\node[vertex] (Gp4) at (4+.2, 1) [shape=circle, draw] {};
\node[vertex] (G-4) at (4-.2, -1) [shape=circle, draw] {};
\node[vertex] (Gp-4) at (4+.2, -1) [shape=circle, draw] {};
\draw[-] (G4) .. controls +(0.1, -.25) and +(-0.1, -.25) .. (Gp4);
\draw[-] (G-4) .. controls +(0.1, .25) and +(-0.1, .25) .. (Gp-4);
\draw (2,1) node {$\cdots$};
\draw (2,-1) node {$\cdots$};
\draw (7,1) node {$\cdots$};
\draw (7,-1) node {$\cdots$};
\end{tikzpicture}
= \beta e_{2i-1},
\\
b_i =
\begin{tikzpicture}[scale = 0.5,thick, baseline={(0,-1ex/2)}] 
\tikzstyle{vertex} = [shape=circle, minimum size=2pt, inner sep=1pt]
\foreach \i in {1,3,4,5,6,8} {
    \node[vertex] (G\i) at (\i, 1) [shape=circle, draw] {};
    \node[vertex] (G-\i) at (\i, -1) [shape=circle, draw] {};
}
\foreach \i in {1,3,6,8} {
    \draw[-] (G\i) -- (G-\i);
}
\draw (2,1) node {$\cdots$};
\draw (2,-1) node {$\cdots$};
\draw (7,1) node {$\cdots$};
\draw (7,-1) node {$\cdots$};
\draw[-] (G4) .. controls +(0, -.5) and +(0, -.5) .. (G5) -- (G-5) .. controls +(0, .5) and +(0, .5) .. (G-4) -- (G4);
\end{tikzpicture}
& \mapsto
\beta^{-1} \cdot
\begin{tikzpicture}[scale = 0.5,thick, baseline={(0,-1ex/2)}] 
\tikzstyle{vertex} = [shape=circle, minimum size=2pt, fill=red, inner sep=1pt]
\foreach \i in {1,3,4,5,6,8} {
    \node[vertex] (G\i) at (\i-.2, 1) [shape=circle, draw] {};
    \node[vertex] (Gp\i) at (\i+.2, 1) [shape=circle, draw] {};
    \node[vertex] (G-\i) at (\i-.2, -1) [shape=circle, draw] {};
    \node[vertex] (Gp-\i) at (\i+.2, -1) [shape=circle, draw] {};
}
\foreach \i in {1,3,6,8} {
    \draw[-] (G\i) -- (G-\i);
    \draw[-] (Gp\i) -- (Gp-\i);
}
\draw[-] (Gp4) .. controls +(0.1, -.25) and +(-0.1, -.25) .. (G5);
\draw[-] (Gp-4) .. controls +(0.1, .25) and +(-0.1, .25) .. (G-5);
\draw[-] (G4) -- (G-4);
\draw[-] (Gp5) -- (Gp-5);
\draw (2,1) node {$\cdots$};
\draw (2,-1) node {$\cdots$};
\draw (7,1) node {$\cdots$};
\draw (7,-1) node {$\cdots$};
\end{tikzpicture}
= \beta^{-1} e_{2i}.
\end{align*}
It is a simple direct check to see this is a morphism as it satisfies the Temperley-Lieb algebra relations.
Since the image is the generating set of $\mcT\mcL_{2n}(\beta)$ and the dimensions of the two algebras are equal, this is an isomorphism.
(Alternatively, the map is clearly invertible that maps to the generating set of $\mcP\mcP_n(\beta)$.)

For $\beta = 0$ (or if we wanted to work over more general rings, when $\beta$ is not a unit), the map above does not make sense.
However, based upon computations done using \textsc{SageMath}~\cite{sage} with $\field = \QQ$ (this should have no dependence on the field since all structure coefficients are either $1$ or $0$) and vague claims in the literature, there should still be an isomorphism at $\beta = 0$.
The author does not know of any proof and could not construct an explicit isomorphism for $n = 2$ over $\CC$.

\begin{problem}
\label{prob:TL_planar_zero}
Determine if Theorem~\ref{thm:TL_planer_iso} holds at $\beta = 0$.
\end{problem}

Therefore, by Theorem~\ref{thm:TL_planer_iso} and assuming it also holds for $\beta = 0$, the semisimplicity of $\mcP\mcP_n(\beta)$ is determined by the Temperley--Lieb algebra $\mcT\mcL_{2n}(\sqrt{\beta})$ description from~\cite{RSA14} as we can always extend $\field$ to include $\sqrt{\beta}$ by Remark~\ref{rem:algebraic_closure}.
It is possible these algebras are not isomorphic at $\beta = 0$ but satisfy the weaker statement of being Morita equivalent, which would be sufficient for the purposes of this article since we are only interested in properties of the representations.
A proof of the Morita equivalence should be a consequence of the description of the bilinear forms $\Phi_{\lambda}$ from~\cite[Sec.~4]{RSA14} with suitable modifications for the natural basis of $\mcP\mcP_n(\beta)$.

It is known that the $\beta = 0$ case is often one of the more interesting cases (see, \textit{e.g.},~\cite{RSA14} and references therein).
Other interesting behaviors can occur for $\beta = \pm (q + q^{-1})$ with $q$ being a root of unity, which covers all of the cases when the Temperley--Lieb algebra is possibly not semisimple (\textit{e.g.}, $q = \sqrt{-1}$ gives $\beta = 0$).
For example, in~\cite{ILZ19}, the Temperley--Lieb algebra $\mcT\mcL_n\bigl(\pm(q+q^{-1})\bigr)$ at $q$ being a root of unity was related to the fusion category generated by tilting modules of $U_q(\fsl_2)$; see also~\cite{AST18}.

\subsection{Planar uniform block algebra}

The planar subalgebra $\mcP\mcU_n$ of $\mcI\mcU_n \subseteq \mcU_n$ is simply the subalgebra of planar idempotents of $\mcU_n$ (see Section~\ref{sec:uniform_block}).
Therefore, the basis is indexed by compositions of $n$.
For such a composition $\mu = (\mu_1, \dotsc, \mu_{\ell})$, the basis element is indexed by the idempontents of the form
\[
\bigl\{ \{1, \dotsc, \mu_1\}, \{\Psi_1+1, \dotsc, \Psi_1+\mu_2\}, \dotsc, \{\Psi_{\ell-1}+1, \dotsc, n\} \bigr\},
\]
where $\Psi_k = \mu_1 + \cdots + \mu_k$, to the same primed set partition.
Consequently, we see that
\begin{equation}
\label{eq:planar_uniform_dim}
\dim \mcP\mcU_n = 2^{n-1},
\end{equation}
and by Proposition~\ref{prop:uniform_idempotent_cellular}, all of the cell modules are one dimensional irreducible representations.

\subsection{Planar rook algebra}

The planar rook algebra $\mcP\mcR_n(\beta)$ was studied in~\cite{FHH09} and consists of planar diagrams with either singletons or propagating blocks.
As a consequence, the planar condition is now equivalent to having a corresponding permutation $\sigma = 1$.

\begin{cor}
The planar rook algebra $\mcP\mcR_n(\beta)$ is a cellular algebra with cell data given by $\Lambda = \{0,1,\dotsc,n\}$ and $M(\lambda)$ being the subsets of $[n]$ of size $\lambda$.
The Gram matrix for the bilinear form $\Phi_{\lambda}$ with respect to the natural diagram basis is a diagonal matrix with entries being $\beta^{\lambda}$.
It is always semisimple unless $\beta = 0$, in which case there is a unique simple one dimensional module $M(0)$.
\end{cor}

\begin{proof}
The first statement is Theorem~\ref{thm:planar_cellular} with the subsets of $[n]$ corresponding to the singleton blocks.
The remaining claims are straightforward from the combinatorial description.
\end{proof}

For $\beta = 0$, we can describe the action on $M(0)$ explicitly by $a v = c_a v$ for all $a \in \mcP\mcR_n(\beta)$ and $v \in M(0)$, where $c_a$ is the coefficient of $1$ in $a$.

We can also easily see from Equation~\eqref{eq:dim_formula} that
\[
\dim \mcP\mcR_n(\beta) = \binom{2n}{n} = \sum_{\lambda=0}^n \binom{n}{\lambda}^2.
\]

\subsection{Motzkin algebra}

The Motzkin algebra $\mcM(\beta)$ introduced in~\cite{BH14} can be described as the Temperley--Lieb algebra but with singletons allowed.
The dimension of $\mcM(\beta)$ is the $2n$-th Motzkin number $M_{2n}$.
For the cell datum, we have $\Lambda = [n]$ with $M(\lambda)$ being the set of noncrossing matchings with allowing singletons with exactly $\lambda$ defects, a set of non-nested singletons.
The dimensions of the cell modules are counted by the Motzkin triangle numbers, which are known to be computed by~\cite[Eq.~(3.22)]{BH14} and~\cite{Lando03} (see also~\cite[A026300]{OEIS}):
\begin{align*}
M_{n,k} & = \sum_{i=0}^{\lfloor (n-k)/2 \rfloor} \binom{n}{k-2i} \stirling{k+2i}{i}
\\ & = \sum_{i=0}^{\lfloor (n-k)/2 \rfloor} \binom{n}{2i+k} \left[ \binom{2i+k}{i} - \binom{2i+k}{i-1} \right]
\\ & = M_{n-1,k-1} + M_{n-1,k} + M_{n-1,k+1},
\end{align*}
with $M_{n,-1} = M_{n,n+1} = 0$ and $M_{0,0} = 1$.
The bijection from a half diagram in $M(\lambda)$ to a Motzkin path of length $n$ that ends at height $\lambda$ is by reading the diagram from left-to-right and treating each pair $\{a,b\}$ as an arc from $a$ to $b$, where each outgoing (resp.\ incoming) edge is $+1$ (resp.~$-1$), non-defect singletons are $0$, and defect singletons are $+1$.
This is the cell datum shown in Benkart and Halverson~\cite[Thm.~4.16]{BH14}.
We also have the analog of Equation~\eqref{eq:catalan_square} by Equation~\eqref{eq:dim_formula}:
\begin{equation}
\label{eq:motzkin_square}
M_{2n} = \sum_{\lambda=0}^n M_{n,\lambda}^2.
\end{equation}

\begin{thm}[{\cite[Thm.~5.14]{BH14}}]
The Motzkin algebra $\mcM_n(\beta)$ is semisimple if and only if $\beta - 1$ is not the root of the rescaled Chebyshev polynomials of the second kind\footnote{The rescaled Chebyshev polynomials have coefficients in $\ZZ$, so they can still be evaluated when $x$ is an element of a field of characteristic $2$.}
defined by $u_k(x) = U_k(x/2)$ for all $1 \leq k < n$.
\end{thm}

The first definition in~\cite{BH14} has the Motzkin algebra with two parameters $\mcM(\beta, \gamma)$ as it is a subalgebra of the rook Brauer algebra.
There is a similar misprint in~\cite{BH14} as in~\cite{HdM14} as described in Example~\ref{ex:nonassoc_product}.
We have the analogs of Proposition~\ref{prop:two_param_rook_brauer_iso} and Problem~\ref{prob:two_param_rook_brauer_semisimple}.

\begin{prop}
For all $\gamma \neq 0$, we have $\mcR\mcB_n(0, \gamma) \iso \mcR\mcB_n(0, 1)$.
\end{prop}

\begin{problem}
\label{prob:two_param_motzkin}
Determine when the two-parameter Motzkin algebra $\mcM(\beta, 0)$ is semisimple.
Are the points where it is not semisimple described as roots to some specialization of a higher level generalization of Chebyshev polynomials in the Askey scheme, more specifically as a specialization of Jacobi polynomials?
\end{problem}

We remark on another curious appearance of Chevyshev polynomials of the second kind with the branching rule of the Brauer algebra $\mcB_n(\beta)$ to $\mcT\mcL_n(\beta)$ that was given in~\cite{BM05}.

\begin{problem}
Determine if there is a relationship between the branching rule from the Motzkin algebra to the Brauer algebra and the semisimplicity of the Motzkin algebra.
\end{problem}

One way to define a half integer Motzkin algebra would be to mandate that, \textit{e.g.}, $1$ is always a singleton.
Indeed, this is a subalgebra of dimension $M_{2n-1}$, but it is not immediately clear it is cellular since the basis is not invariant under $\iota$.

\subsection{Partial Temperley--Lieb algebra}

The partial Temperley--Lieb algebra $\mcP\mcT\mcL_n(\beta)$ was recently introduced in~\cite{DG22} as a subalgebra of the Motzkin algebra $\mcM_n(\beta)$.
Specifically, its basis is indexed by \defn{balanced} diagrams: For a diagram $\rho$, the number of pairs of $\rho \cap [n]$ equals those of $\rho \cap [n]'$.
Each balanced diagram corresponds to a basis element given as an alternating sum of diagrams by removing edges (but leaving the vertices).
The proof of cellularity by Doty and Giaquinto is using the following stronger result.

\begin{thm}[{\cite[Thm.~4.2]{DG22}}]
\label{thm:PTL_Morita}
We have the Morita equivalence
\[
\mcP\mcT\mcL_n(\beta) \simeq \bigoplus_{k=1}^n \mcT\mcL_k(\beta - 1).
\]
\end{thm}

As a consequence of Theorem~\ref{thm:PTL_Morita}, we have a complete classification of the semisimplicity of the partial Temperley--Lieb algebra from the Temperley--Lieb one.
Furthermore, their description of the modules of $\mcP\mcT\mcL_n(\beta)$ is essentially the same as given here by using half diagrams with the cellular algebra structure.
We remark that Proposition~\ref{prop:subcellular} and Proposition~\ref{prop:cellular_triangle_basis} give an alternative simple proof that $\mcP\mcL_n(\beta)$ is a cellular algebra (using that $\mcM_n(\beta)$ is a cellular algebra).

\subsection{Planar even algebra}
\label{sec:planar_even}

If we take the planar version of the parity matching algebra, we obtain a cellular algebra with interesting combinatorial properties.
We call the planar subalgebra of $\mcP\mcM_n(\beta)$ the \defn{planar even algebra} and denote it by $\mcP\mcE_n(\beta)$.

The dimension of this subalgebra is known~\cite[A001764]{OEIS}:
\begin{equation}
\label{eq:even_planar_set_partitions}
\dim \mcP\mcE_n(\beta) = \frac{1}{2n+1}\binom{3n}{n},
\end{equation}
as well as many other combinatorial interpretations.
It would be interesting to see what the product structure in $\mcP\mcE_n(\beta)$ is on other such interpretations, and if these have other natural algebraic structures, what is there intrepretation in terms of $\mcP\mcE_n(\beta)$. 
Next, we explicitly describe a parameterization of the cell modules for $\beta \neq 0$.

\begin{thm}
\label{thm:planar_even_cell_modules}
The set $\Lambda$ for $\mcP\mcE_n(\beta)$ is given by all words $\lambda = \lambda_1 \dotsm \lambda_{\ell}$ in the alphabet $\{1, 2\}$ such that the sum $S = \sum_{i=1}^{\ell} \lambda_i \leq n$ and $n \equiv S \pmod{2}$.
Furthermore, the bilinear form $\Phi_{\lambda} \neq 0$ for all $\lambda \in \Lambda$.
\end{thm}

\begin{proof}
We first assume $\beta \neq 0$.
The basis elements of all cell modules are given by all planar set partitions of $[n]'$, where the choice of defect blocks must not be nested.
Our first claim that for any such basis diagram $\nu \in W$, we can multiply it multiplied by a diagram $\rho \in \mcP\mcE_n(\beta)$ such that we get some sequence $\lambda$ of defect blocks of sizes $1$ and $2$ in some order followed by a sequence of caps $\{i', (i+1)'\}$.
We denote such an element by $\rho_{\lambda}$.

If there are no defects (necessarily, we must have $n$ being even), we can simply multiply by the cap and cup element:
\begin{equation}
\label{eq:cap_cup_element}
\begin{gathered}
\{\{1,2\},\dotsc,\{2n-1,2n\}, \{1',2'\},\dotsc,\{(2n-1)',(2n)'\} = b_1 b_3 \dotsm b_{2n-1}
\\ =
\begin{tikzpicture}[scale = 0.5,thick, baseline={(0,-1ex/2)}] 
\tikzstyle{vertex} = [shape=circle, minimum size=2pt, inner sep=1pt]
\foreach \i in {1,3,4,5,7} {
    \node[vertex] (G\i) at (2*\i, 1) [shape=circle, draw] {};
    \node[vertex] (Gp\i) at (2*\i+1, 1) [shape=circle, draw] {};
    \node[vertex] (G-\i) at (2*\i, -1) [shape=circle, draw] {};
    \node[vertex] (Gp-\i) at (2*\i+1, -1) [shape=circle, draw] {};
    \draw[-] (G\i) .. controls +(0, -.5) and +(0, -.5) .. (Gp\i);
    \draw[-] (G-\i) .. controls +(0, .5) and +(0, .5) .. (Gp-\i);
}
\draw (4.5,1) node {$\cdots$};
\draw (4.5,-1) node {$\cdots$};
\draw (12.5,1) node {$\cdots$};
\draw (12.5,-1) node {$\cdots$};
\end{tikzpicture}
\end{gathered}
\end{equation}
We can do a similar multiplication but reordering all elements in the defects of $\nu$ so that they become $[m']$ for some $m$.
In particular, if the defects are on the elements $i_1 < \cdots i_m$, then we can multiply by the diagram in $\mcP\mcE_n(\beta)$ consisting of $\{i_1,1'\}, \dotsc, \{i_m,m'\}$ and all remaining elements are cups or caps.
Thus, we can assume every element in $[n]'$ belongs to a defect of $\nu$.

It is sufficient to prove it for when $\nu = \{[n]'\}$ with a single defect block.
If $n$ is odd, the we use the same diagram in Equation~\eqref{eq:cap_cup_element} except we make the last block a propagating block $\{n, n'\}$.
If $n$ is odd, we make the rightmost block a propagating block instead of a cap-cup pair.
Thus, we have shown the first claim.

Next, we claim that an even propagating block cannot cross an odd propagating block.
This follows from a straightforward parity argument.
Hence, the elements $\rho_{\lambda}$ uniquely determine the cell modules up to isomorphism, which is clearly in bijection with the claimed $\Lambda$.
This also naturally extends to a description of the cellular basis by connecting a pair of such cell module basis elements (essentially undoing the decomposition~\eqref{eq:partition_decomposition}).

For $\beta = 0$, we have some additional terms equal to $0$, which does not change the proof above that the resulting basis is cellular.

To see that $\Phi_{\lambda} \neq 0$ for some fixed $\lambda \in \Lambda$, consider the element $\widetilde{\rho}_{\lambda}$, which connects all of the caps in $\rho_{\lambda}$ to the last defect block.
The pairing $\Phi_{\lambda}(\widetilde{\rho}_{\lambda}, \widetilde{\rho}_{\lambda}) = 1$.
\end{proof}

Define $E_P(m) := \dim \mcP\mcE_{m/2}$ given in Equation~\eqref{eq:even_planar_set_partitions} if $m \in 2\ZZ$ and $0$ otherwise.
Thus, $E_P(m)$ counts the number of planar even set partitions of $[m]$.
Note $E_P(0) = 1$ by Equation~\eqref{eq:even_planar_set_partitions}.

\begin{prop}
\label{prop:even_planar_cell_dim}
Fix $\lambda \in \Lambda$ such that $\lambda$ is a permutation of $1^{k_1} 2^{k_2}$.
Then we have
\begin{equation}
\label{eq:planar_even_cell_dim}
\dim W(\lambda) = \sum_{\{i_1 < \cdots < i_{\ell-1}\} \in \binom{[n-1]}{\ell-1}} \prod_{j=1}^{k_1} (2 - \delta_{i_j,i_{j-1}+1}) E_P(i_j - i_{j-1} - 1) \prod_{j=k_1+1}^{\ell} E_P(i_j - i_{j-1}),
\end{equation}
where $i_0 = 0$.
\end{prop}

\begin{proof}
Since we cannot cross propagating blocks, we divide $[n]'$ up into $\ell = k_1 + k_2$ sets
\[
\{1', \dotsc, i_1'\} \sqcup \{(i_1+1)', \dotsc, i_2'\} \sqcup \cdots \sqcup \{(i_{\ell-1}+1)', \dotsc, n'\},
\]
where $i_1 < \cdots < i_{\ell-1}$, such that the smallest entry in each block part of the propagating block.
Without loss of generality, we can look at the first set $\{1', \dotsc, i_1'\}$.
If $\lambda_1 = 2$, then if we consider all even set partitions on $[i_1']$, we simply consider the lexicographic smallest part as the propagating block.
If $\lambda_1 = 1$, then we consider all even set partitions on $\{2', \dotsc, i_1'\}$, but we have two cases.
The first is $\{1'\}$ is an isolated propagating block, and the second is we join $1'$ to the lexicographic smallest part.
Note that the second case is impossible when $i_{j-1} + 1 = i_j$.
The claim follows as we can chose these even set partitions independently in each set.
\end{proof}

Consequently, we see that the dimension only depends on the number of $1$'s and $2$'s in the word $\lambda$, and there are $\binom{k_1 + k_2}{k_1}$ such sequences.
This gives the following identity of binomial coefficients
\[
E_P(n) = \sum_{k_1 + k_2 \leq n} \binom{k_1+k_2}{k_1} \bigl( \dim W(1^{k_1}2^{k_2}) \bigr)^2
\]
after substituting in Equation~\eqref{eq:planar_even_cell_dim}.
We give some examples of the dimensions in Table~\ref{table:even_dim}.
We note that $W(\emptyset) = E_P(n)$ since it consists of all noncrossing even set partitions on $[n]'$.

\begin{table}
\begin{center}
\begin{tabular}[t]{cccc}
\toprule
$n=1$ & $n=2$ & $n=3$ & $n=4$
\\ \midrule
\begin{tabular}[t]{r|cc}
 & 0 & 1
\\ \hline 0 & 0 & 0
\\  1 & 0 & 1
\end{tabular}
&
\begin{tabular}[t]{r|ccc}
 & 0 & 1 & 2
\\ \hline 0 & 1 & 0 & 0
\\ 1 & 1 & 0 & 0
\\ 2 & 0 & 0 & 1
\end{tabular}
&
\begin{tabular}[t]{r|cccc}
 & 0 & 1 & 2 & 3
\\ \hline 0 & 0 & 0 & 0 & 0
\\ 1 & 0 & 3 & 0 & 0
\\ 2 & 0 & 1 & 0 & 0
\\ 3 & 0 & 0 & 0 & 1
\end{tabular}
&
\begin{tabular}[t]{r|ccccc}
 & 0 & 1 & 2 & 3 & 4
\\ \hline 0 & 3 & 0 & 0 & 0 & 0
\\ 1 & 4 & 0 & 0 & 0 & 0
\\ 2 & 1 & 0 & 5 & 0 & 0
\\ 3 & 0 & 0 & 1 & 0 & 0
\\ 4 & 0 & 0 & 0 & 0 & 1
\end{tabular}
\\ \bottomrule
\end{tabular}
\end{center}
\caption{The dimension of $W(1^{k_1}2^{k_2})$ in $\mcP\mcE_n(\beta)$ for $n \leq 4$. The value $k_1$ (resp.\ $k_1 + k_2$ or the number of propagating blocks) is given by the column (resp.\ row) in each table.}
\label{table:even_dim}
\end{table}

\begin{problem}
\label{prob:irrep_dim_CE}
Find a recurrence relation, generating function, and more compact (closed) formula for $\dim(1^{k_1} 2^{k_2})$.
\end{problem}

\begin{ex}
The basis diagrams for the cell modules of $\mcP\mcE_4(\beta)$ are
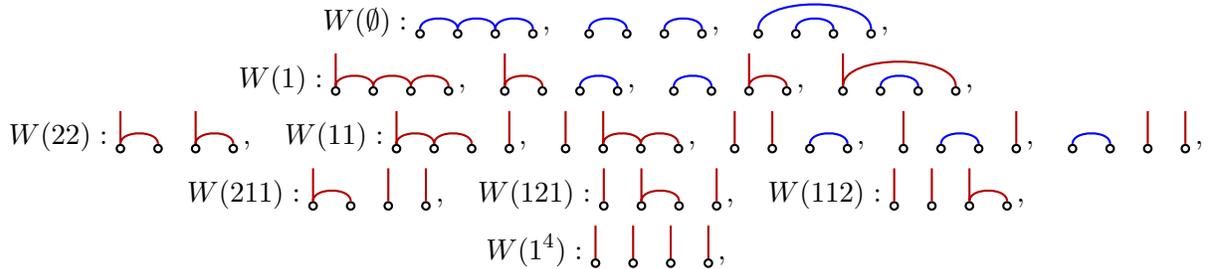
\begin{gather*}
W(\emptyset):
  \begin{tikzpicture}[scale = 0.5,thick, baseline=30pt] 
  \tikzstyle{vertex} = [shape=circle, minimum size=2pt, inner sep=1pt]
  \foreach \i in {1,...,4} {
      \node[vertex] (G\i) at (\i, 2) [shape=circle, draw] {};
  }
  \draw[-,blue] (G1) .. controls +(0, .5) and +(0, .5) .. (G2) .. controls +(0, .5) and +(0, .5) .. (G3) .. controls +(0, .5) and +(0, .5) .. (G4);
  \end{tikzpicture}\,,
  \quad
  \begin{tikzpicture}[scale = 0.5,thick, baseline=30pt] 
  \tikzstyle{vertex} = [shape=circle, minimum size=2pt, inner sep=1pt]
  \foreach \i in {1,...,4} {
      \node[vertex] (G\i) at (\i, 2) [shape=circle, draw] {};
  }
  \draw[-,blue] (G1) .. controls +(0, .5) and +(0, .5) .. (G2); \draw[-,blue] (G3) .. controls +(0, .5) and +(0, .5) .. (G4);
  \end{tikzpicture}\,,
  \quad
  \begin{tikzpicture}[scale = 0.5,thick, baseline=30pt] 
  \tikzstyle{vertex} = [shape=circle, minimum size=2pt, inner sep=1pt]
  \foreach \i in {1,...,4} {
      \node[vertex] (G\i) at (\i, 2) [shape=circle, draw] {};
  }
  \draw[-,blue] (G1) .. controls +(0, 1) and +(0, 1) .. (G4); \draw[-,blue] (G3) .. controls +(0, .5) and +(0, .5) .. (G2);
  \end{tikzpicture}\,,
\\
W(1):
  \begin{tikzpicture}[scale = 0.5,thick, baseline=30pt] 
  \tikzstyle{vertex} = [shape=circle, minimum size=2pt, inner sep=1pt]
  \foreach \i in {1,...,4} {
      \node[vertex] (G\i) at (\i, 2) [shape=circle, draw] {};
  }
  \draw[-,darkred] (G1) + (0,1) -- (G1) .. controls +(0, .5) and +(0, .5) .. (G2) .. controls +(0, .5) and +(0, .5) .. (G3) .. controls +(0, .5) and +(0, .5) .. (G4);
  \end{tikzpicture}\,,
  \quad
  \begin{tikzpicture}[scale = 0.5,thick, baseline=30pt] 
  \tikzstyle{vertex} = [shape=circle, minimum size=2pt, inner sep=1pt]
  \foreach \i in {1,...,4} {
      \node[vertex] (G\i) at (\i, 2) [shape=circle, draw] {};
  }
  \draw[-,darkred] (G1) + (0,1) -- (G1) .. controls +(0, .5) and +(0, .5) .. (G2); \draw[-,blue] (G3) .. controls +(0, .5) and +(0, .5) .. (G4);
  \end{tikzpicture}\,,
  \quad
  \begin{tikzpicture}[scale = 0.5,thick, baseline=30pt] 
  \tikzstyle{vertex} = [shape=circle, minimum size=2pt, inner sep=1pt]
  \foreach \i in {1,...,4} {
      \node[vertex] (G\i) at (\i, 2) [shape=circle, draw] {};
  }
  \draw[-,blue] (G1) .. controls +(0, .5) and +(0, .5) .. (G2); \draw[-,darkred] (G3) + (0,1) -- (G3) .. controls +(0, .5) and +(0, .5) .. (G4);
  \end{tikzpicture}\,,
  \quad
  \begin{tikzpicture}[scale = 0.5,thick, baseline=30pt] 
  \tikzstyle{vertex} = [shape=circle, minimum size=2pt, inner sep=1pt]
  \foreach \i in {1,...,4} {
      \node[vertex] (G\i) at (\i, 2) [shape=circle, draw] {};
  }
  \draw[-,darkred] (G1) + (0,1) -- (G1) .. controls +(0, 1) and +(0, 1) .. (G4); \draw[-,blue] (G3) .. controls +(0, .5) and +(0, .5) .. (G2);
  \end{tikzpicture}\,,
\\
W(22):
  \begin{tikzpicture}[scale = 0.5,thick, baseline=30pt] 
  \tikzstyle{vertex} = [shape=circle, minimum size=2pt, inner sep=1pt]
  \foreach \i in {1,...,4} {
      \node[vertex] (G\i) at (\i, 2) [shape=circle, draw] {};
  }
  \draw[-,darkred] (G1) + (0,1) -- (G1) .. controls +(0, .5) and +(0, .5) .. (G2);
  \draw[-,darkred] (G3) + (0,1) -- (G3) .. controls +(0, .5) and +(0, .5) .. (G4);
  \end{tikzpicture}\,,
\quad
W(11):
  \begin{tikzpicture}[scale = 0.5,thick, baseline=30pt] 
  \tikzstyle{vertex} = [shape=circle, minimum size=2pt, inner sep=1pt]
  \foreach \i in {1,...,4} {
      \node[vertex] (G\i) at (\i, 2) [shape=circle, draw] {};
  }
  \draw[-,darkred] (G1) + (0,1) -- (G1) .. controls +(0, .5) and +(0, .5) .. (G2) .. controls +(0, .5) and +(0, .5) .. (G3);
  \draw[-,darkred] (G4) + (0,1) -- (G4);
  \end{tikzpicture}\,,
  \quad
  \begin{tikzpicture}[scale = 0.5,thick, baseline=30pt] 
  \tikzstyle{vertex} = [shape=circle, minimum size=2pt, inner sep=1pt]
  \foreach \i in {1,...,4} {
      \node[vertex] (G\i) at (\i, 2) [shape=circle, draw] {};
  }
  \draw[-,darkred] (G2) + (0,1) -- (G2) .. controls +(0, .5) and +(0, .5) .. (G3) .. controls +(0, .5) and +(0, .5) .. (G4);
  \draw[-,darkred] (G1) + (0,1) -- (G1);
  \end{tikzpicture}\,,
  \quad
  \begin{tikzpicture}[scale = 0.5,thick, baseline=30pt] 
  \tikzstyle{vertex} = [shape=circle, minimum size=2pt, inner sep=1pt]
  \foreach \i in {1,...,4} {
      \node[vertex] (G\i) at (\i, 2) [shape=circle, draw] {};
  }
  \draw[-,blue] (G3) .. controls +(0, .5) and +(0, .5) .. (G4); \draw[-,darkred] (G1) + (0,1) -- (G1);
  \draw[-,darkred] (G2) + (0,1) -- (G2);
  \end{tikzpicture}\,,
  \quad
  \begin{tikzpicture}[scale = 0.5,thick, baseline=30pt] 
  \tikzstyle{vertex} = [shape=circle, minimum size=2pt, inner sep=1pt]
  \foreach \i in {1,...,4} {
      \node[vertex] (G\i) at (\i, 2) [shape=circle, draw] {};
  }
  \draw[-,blue] (G2) .. controls +(0, .5) and +(0, .5) .. (G3); \draw[-,darkred] (G1) + (0,1) -- (G1);
  \draw[-,darkred] (G4) + (0,1) -- (G4);
  \end{tikzpicture}\,,
  \quad
  \begin{tikzpicture}[scale = 0.5,thick, baseline=30pt] 
  \tikzstyle{vertex} = [shape=circle, minimum size=2pt, inner sep=1pt]
  \foreach \i in {1,...,4} {
      \node[vertex] (G\i) at (\i, 2) [shape=circle, draw] {};
  }
  \draw[-,blue] (G1) .. controls +(0, .5) and +(0, .5) .. (G2); \draw[-,darkred] (G3) + (0,1) -- (G3);
  \draw[-,darkred] (G4) + (0,1) -- (G4);
  \end{tikzpicture}\,,
\\
W(211):
  \begin{tikzpicture}[scale = 0.5,thick, baseline=30pt] 
  \tikzstyle{vertex} = [shape=circle, minimum size=2pt, inner sep=1pt]
  \foreach \i in {1,...,4} {
      \node[vertex] (G\i) at (\i, 2) [shape=circle, draw] {};
  }
  \draw[-,darkred] (G1) + (0,1) -- (G1) .. controls +(0, .5) and +(0, .5) .. (G2);
  \draw[-,darkred] (G3) + (0,1) -- (G3);
  \draw[-,darkred] (G4) + (0,1) -- (G4);
  \end{tikzpicture}\,,
\quad
W(121):
  \begin{tikzpicture}[scale = 0.5,thick, baseline=30pt] 
  \tikzstyle{vertex} = [shape=circle, minimum size=2pt, inner sep=1pt]
  \foreach \i in {1,...,4} {
      \node[vertex] (G\i) at (\i, 2) [shape=circle, draw] {};
  }
  \draw[-,darkred] (G2) + (0,1) -- (G2) .. controls +(0, .5) and +(0, .5) .. (G3);
  \draw[-,darkred] (G1) + (0,1) -- (G1);
  \draw[-,darkred] (G4) + (0,1) -- (G4);
  \end{tikzpicture}\,,
\quad
W(112):
  \begin{tikzpicture}[scale = 0.5,thick, baseline=30pt] 
  \tikzstyle{vertex} = [shape=circle, minimum size=2pt, inner sep=1pt]
  \foreach \i in {1,...,4} {
      \node[vertex] (G\i) at (\i, 2) [shape=circle, draw] {};
  }
  \draw[-,darkred] (G3) + (0,1) -- (G3) .. controls +(0, .5) and +(0, .5) .. (G4);
  \draw[-,darkred] (G1) + (0,1) -- (G1);
  \draw[-,darkred] (G2) + (0,1) -- (G2);
  \end{tikzpicture}\,,
\\
W(1^4):
  \begin{tikzpicture}[scale = 0.5,thick, baseline=30pt] 
  \tikzstyle{vertex} = [shape=circle, minimum size=2pt, inner sep=1pt]
  \foreach \i in {1,...,4} {
      \node[vertex] (G\i) at (\i, 2) [shape=circle, draw] {};
      \draw[-,darkred] (G\i) + (0,1) -- (G\i);
  }
  \end{tikzpicture}\,,
\end{gather*}
We can see that $\dim W(211) = \dim W(121) = \dim W(112)$ and $\dim W(\emptyset) = \dim \mcP\mcE_2(\beta)$.
\end{ex}

\subsection{Planar \texorpdfstring{$r$}{r}-color algebra}
\label{sec:planar_color}

Like $\mcP\mcE_n(\beta)$ as a planar version of $\mcP\mcM_n(\beta)$, we can consider a planar version of $\mcG_n^{(r,1,m)}(\beta)$ for any $m \geq n$.
This is again independent of $m$ whenever $m \geq n$.
Let $\mcP\mcC_{r,n}(\beta)$ denote the corresponding algebra, which we call the \defn{planar $r$-color algebra}.
The set $\Lambda$ is similar to the case $r = 2$: We take all compositions $\lambda$ with parts in $[r]$ (\textit{i.e.}, $\lambda_i \in [r]$) such that $\abs{\lambda} \leq n$ and $\abs{\lambda} \equiv n \pmod{r}$.
Note that compositions are equivalent to words.

Recall that $\mcT\mcL_{2n}(\beta) \iso \mcP\mcP_n(\beta^2)$ has dimension equal to the Catalan numbers given by Equation~\eqref{eq:TL_dim}, which is the planar version of $\mcG_n^{(1,1,m)}(\beta^2)$.
Additionally recall Equation~\eqref{eq:even_planar_set_partitions}, which is for the planar version of $\mcG_n^{(1,1,m)}(\beta^2)$.
Thus, a natural guess is $\dim \mcP\mcC_{r,n}(\beta)$ is a \defn{Fuss--Catalan number} (of type $A_n$; see, \textit{e.g.},~\cite{STW15} and references therein or~\cite[A137211]{OEIS}):
\[
C_n^{(r)} := \frac{1}{rn + 1} \binom{(r+1)n}{n},
\]
Indeed, Edelman showed in~\cite{Edelman80} that the Fuss--Catalan numbers count the number of noncrossing set partitions of $[nr]$ with block sizes that are divisible by $r$, with the enumeration dating back to the work of Fuss~\cite{Fuss91}.
Unfortunately, this is not the case, as Table~\ref{table:dim_planar_color} indicates when compared with
\[
\bigl( C_n^{(3)} \bigr)_{n=0}^{\infty} = (1, 1, 4, 22, 140, 969, 7084, 53820, 420732, 3362260, \ldots).
\]

However, we do have that the dimension of the unique cell module with zero defects is equal to the Fuss--Catalan numbers from the description given by Edelman~\cite{Edelman80}.
Indeed, the following is just a rephrasing of this result.

\begin{prop}[{\cite{Edelman80}}]
\label{prop:planar_color_zero_defect_dim}
There is a diagram of $\mcP\mcC_{r,n}(\beta)$ with zero propagating blocks if and only if $r \divides n$.
Moreover, if $r \divides n$, then let $\lambda = \emptyset$ be the unique index for the cell module corresponding to zero defects in $\mcP\mcC_{r,n}(\beta)$.
Then
\[
\dim W(\emptyset) = C_{n/r}^{(r)}.
\]
\end{prop}

We can compute $\dim W(\lambda)$ by using the analog of Proposition~\ref{prop:even_planar_cell_dim}, where $E_P(j)$ is replaced by the corresponding Fuss--Catalan number.
Thus, we obtain a closed, if somewhat complicated, formula for $\dim \mcP\mcC_{r,n}$ by Equation~\eqref{eq:dim_formula}.
Moreover, given Equation~\eqref{eq:catalan_square} and Equation~\eqref{eq:motzkin_square} (as well as Theorem~\ref{thm:quasi_partition_cellular} below for the Riordan numbers), we have the following.

\begin{problem}
Determine if the dimensions of the cell modules $W(\lambda)$ can be used to define Fuss--Catalan triangle, or an $r+1$ dimensional simplex, numbers.
\end{problem}

For the remainder of this section, we will focus on the case $n \geq r > n/2$ and can give some explicit compact formulas for the dimensions of the planar $r$-color algebra and its cell modules $W(\lambda)$ (as well as $\abs{\Lambda}$).
From the definitions and Equation~\eqref{eq:planar_uniform_dim}, we have $\dim \mcP\mcC_{r,n}(\beta) = \dim \mcP\mcU_n = 2^{n-1}$ for all $r > n$.
It is easy to see that $\dim \mcP\mcC_{n,n}(\beta) = 2^{n-1} + 1$.
Indeed, we have all of the basis elements of $\mcP\mcC_{n+1,n}(\beta) = \mcP\mcU_n$ (where any such block must be propagating) plus the diagram $\{[n], [n]'\}$ (the unique element in $\mcP\mcC_{n,n}(\beta)$ with no propagating blocks).
Similarly, we can show $\dim \mcP\mcC_{n-1,n}(\beta) = 2^{n-1} + 8$ by noting that there are additional diagrams of the form
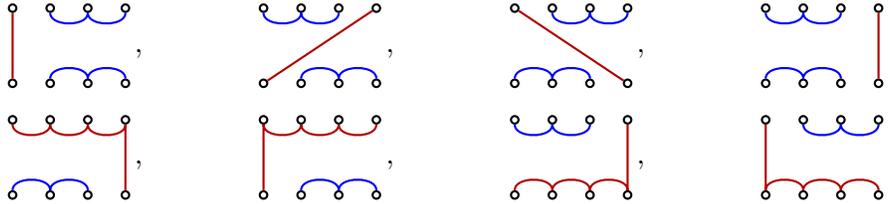
\[
\begin{array}{c@{\qquad\qquad}c@{\qquad\qquad}c@{\qquad\qquad}c}
\begin{tikzpicture}[scale = 0.5,thick, baseline={(0,-1ex/2)}] 
\tikzstyle{vertex} = [shape=circle, minimum size=2pt, inner sep=1pt]
\foreach \i in {1,2,3,4} {
    \node[vertex] (G\i) at (\i, 1) [shape=circle, draw] {};
    \node[vertex] (G-\i) at (\i, -1) [shape=circle, draw] {};
}
\draw[-,color=darkred] (G1) -- (G-1);
\draw[-,color=blue] (G2) .. controls +(0, -.5) and +(0, -.5) .. (G3) .. controls +(0, -.5) and +(0, -.5) .. (G4);
\draw[-,color=blue] (G-2) .. controls +(0., .5) and +(0, .5) .. (G-3) .. controls +(0., .5) and +(0, .5) .. (G-4);
\end{tikzpicture}\,,
&
\begin{tikzpicture}[scale = 0.5,thick, baseline={(0,-1ex/2)}] 
\tikzstyle{vertex} = [shape=circle, minimum size=2pt, inner sep=1pt]
\foreach \i in {1,2,3,4} {
    \node[vertex] (G\i) at (\i, 1) [shape=circle, draw] {};
    \node[vertex] (G-\i) at (\i, -1) [shape=circle, draw] {};
}
\draw[-,color=darkred] (G4) -- (G-1);
\draw[-,color=blue] (G1) .. controls +(0, -.5) and +(0, -.5) .. (G2) .. controls +(0, -.5) and +(0, -.5) .. (G3);
\draw[-,color=blue] (G-2) .. controls +(0., .5) and +(0, .5) .. (G-3) .. controls +(0., .5) and +(0, .5) .. (G-4);
\end{tikzpicture}\,,
&
\begin{tikzpicture}[scale = 0.5,thick, baseline={(0,-1ex/2)}] 
\tikzstyle{vertex} = [shape=circle, minimum size=2pt, inner sep=1pt]
\foreach \i in {1,2,3,4} {
    \node[vertex] (G\i) at (\i, 1) [shape=circle, draw] {};
    \node[vertex] (G-\i) at (\i, -1) [shape=circle, draw] {};
}
\draw[-,color=darkred] (G1) -- (G-4);
\draw[-,color=blue] (G2) .. controls +(0, -.5) and +(0, -.5) .. (G3) .. controls +(0, -.5) and +(0, -.5) .. (G4);
\draw[-,color=blue] (G-1) .. controls +(0., .5) and +(0, .5) .. (G-2) .. controls +(0., .5) and +(0, .5) .. (G-3);
\end{tikzpicture}\,,
&
\begin{tikzpicture}[scale = 0.5,thick, baseline={(0,-1ex/2)}] 
\tikzstyle{vertex} = [shape=circle, minimum size=2pt, inner sep=1pt]
\foreach \i in {1,2,3,4} {
    \node[vertex] (G\i) at (\i, 1) [shape=circle, draw] {};
    \node[vertex] (G-\i) at (\i, -1) [shape=circle, draw] {};
}
\draw[-,color=darkred] (G4) -- (G-4);
\draw[-,color=blue] (G1) .. controls +(0, -.5) and +(0, -.5) .. (G2) .. controls +(0, -.5) and +(0, -.5) .. (G3);
\draw[-,color=blue] (G-1) .. controls +(0., .5) and +(0, .5) .. (G-2) .. controls +(0., .5) and +(0, .5) .. (G-3);
\end{tikzpicture}\,,
\\[20pt]
\begin{tikzpicture}[scale = 0.5,thick, baseline={(0,-1ex/2)}] 
\tikzstyle{vertex} = [shape=circle, minimum size=2pt, inner sep=1pt]
\foreach \i in {1,2,3,4} {
    \node[vertex] (G\i) at (\i, 1) [shape=circle, draw] {};
    \node[vertex] (G-\i) at (\i, -1) [shape=circle, draw] {};
}
\draw[-,color=darkred] (G-4) -- (G4) .. controls +(0, -.5) and +(0, -.5) .. (G3) .. controls +(0, -.5) and +(0, -.5) .. (G2) .. controls +(0, -.5) and +(0, -.5) .. (G1);
\draw[-,color=blue] (G-1) .. controls +(0., .5) and +(0, .5) .. (G-2) .. controls +(0., .5) and +(0, .5) .. (G-3);
\end{tikzpicture}\,,
&
\begin{tikzpicture}[scale = 0.5,thick, baseline={(0,-1ex/2)}] 
\tikzstyle{vertex} = [shape=circle, minimum size=2pt, inner sep=1pt]
\foreach \i in {1,2,3,4} {
    \node[vertex] (G\i) at (\i, 1) [shape=circle, draw] {};
    \node[vertex] (G-\i) at (\i, -1) [shape=circle, draw] {};
}
\draw[-,color=darkred] (G-1) -- (G1) .. controls +(0, -.5) and +(0, -.5) .. (G2) .. controls +(0, -.5) and +(0, -.5) .. (G3) .. controls +(0, -.5) and +(0, -.5) .. (G4);
\draw[-,color=blue] (G-2) .. controls +(0., .5) and +(0, .5) .. (G-3) .. controls +(0., .5) and +(0, .5) .. (G-4);
\end{tikzpicture}\,,
&
\begin{tikzpicture}[scale = 0.5,thick, baseline={(0,-1ex/2)}] 
\tikzstyle{vertex} = [shape=circle, minimum size=2pt, inner sep=1pt]
\foreach \i in {1,2,3,4} {
    \node[vertex] (G\i) at (\i, 1) [shape=circle, draw] {};
    \node[vertex] (G-\i) at (\i, -1) [shape=circle, draw] {};
}
\draw[-,color=darkred] (G4) -- (G-4) .. controls +(0, .5) and +(0, .5) .. (G-3) .. controls +(0, .5) and +(0, .5) .. (G-2) .. controls +(0, .5) and +(0, .5) .. (G-1);
\draw[-,color=blue] (G1) .. controls +(0., -.5) and +(0, -.5) .. (G2) .. controls +(0., -.5) and +(0, -.5) .. (G3);
\end{tikzpicture}\,,
&
\begin{tikzpicture}[scale = 0.5,thick, baseline={(0,-1ex/2)}] 
\tikzstyle{vertex} = [shape=circle, minimum size=2pt, inner sep=1pt]
\foreach \i in {1,2,3,4} {
    \node[vertex] (G\i) at (\i, 1) [shape=circle, draw] {};
    \node[vertex] (G-\i) at (\i, -1) [shape=circle, draw] {};
}
\draw[-,color=darkred] (G1) -- (G-1) .. controls +(0, .5) and +(0, .5) .. (G-2) .. controls +(0, .5) and +(0, .5) .. (G-3) .. controls +(0, .5) and +(0, .5) .. (G-4);
\draw[-,color=blue] (G2) .. controls +(0., -.5) and +(0, -.5) .. (G3) .. controls +(0., -.5) and +(0, -.5) .. (G4);
\end{tikzpicture}\,,
\end{array}
\]
generalized to arbitrary $n$.
Extending this argument, we obtain the following lemma.

\begin{lemma}
\label{lemma:pc_dim}
There exists a sequence $(a_j)_{j=0}^{\infty}$ such that for $n \geq r > n / 2$, we have
\[
\dim \mcP\mcC_{r,n}(\beta) = 2^{n-1} + a_{n-r}.
\]
\end{lemma}

\begin{proof}
Since $n \geq r > n/2$, there can be at most one additional block of size $r$ on either side that is not propagating or adjoined to a propagating block.
Considering half diagrams, if every block is a defect, it is a half diagram of $\mcP\mcU_n$, which is our base case since $\mcP\mcC_{r,n} = \mcP\mcU_n$ for $r > n$.
Thus, we can assume that every half diagram has exactly one non-propagating block which occupies consecutive positions $\{j', \dotsc, (j+r)'\}$ for $j \in [n-r+1]$.
On the remaining $n-r$ elements, we must have defect blocks.
Therefore, the number of additional diagrams not in $\mcP\mcU_n \subseteq \mcP\mcC_{r,n}$ depends only on $n-r$.
\end{proof}

Expanding slightly on the proof of Lemma~\ref{lemma:pc_dim}, we can obtain explicit dimension formulas for the cell modules.

\begin{prop}
\label{prop:dim_large_r_PC_cell}
Let $r > n/2$.
Then the dimensions of the cell modules of $\mcP\mcC_{r,n}$ is given by
\[
\dim W(\lambda) = \begin{cases}
1 & \text{if } \abs{\lambda} = n, \\
n - r + 1 + \ell(\lambda) & \text{if } \abs{\lambda} = n - r.
\end{cases}
\]
\end{prop}

\begin{proof}
If $\abs{\lambda} = n$, then this is the same as $\mcP\mcU_n$; thus we assume $\abs{\lambda} = n - r$.
As mentioned above, there are $n - r + 1$ places to place the consecutive positions of the nondefect block of size $r$.
Then there are also $\ell(\lambda)$ to attach the additional (consecutive) $r$ elements.
\end{proof}

Furthermore, we can obtain a simple formula for $\dim \mcP\mcC_{r,n}(\beta)$ when $r > n/2$.
We separate the cases when $r = n$ and $r > n$ as the formula we give breaks down and these cases have already been given above.

\begin{cor}
Let $n > r > n / 2$.
Then, we have
\[
\dim \mcP\mcC_{r,n} = 
2^{n-1} + (9(n-r)^2 + 17(n-r) + 6) 2^{n-r-3}.
\]
\end{cor}

\begin{proof}
Note that number of compositions of $N$ of length $\ell$ are $\binom{N-1}{\ell-1}$.
Every composition of $n - r$ corresponds to a composition of $r$ containing a part of size strictly greater than $r$ and choosing a part of the composition.
So we remove these from our count of $2^{n-1}$.
Then from Proposition~\ref{prop:dim_large_r_PC_cell} and Equation~\eqref{eq:dim_formula}, we have
\begin{equation}
\label{eq:first_dim_PC}
\dim \mcP\mcC_{r,n} = 2^{n-1} - \sum_{\ell=1}^{n-r} \binom{n-r-1}{\ell-1} \ell + \sum_{\ell=1}^{n-r} \binom{n-r-1}{\ell-1} (n - r + 1 + \ell)^2.
\end{equation}
Then by well-known binomial coefficient sums (from the binomial theorem), we have
\begin{align*}
\sum_{\ell=0}^N \binom{N}{\ell} (M + \ell) = (N + 2M) 2^{N-1},
\quad
\sum_{\ell=0}^N \binom{N}{\ell} (M + \ell)^2
& = (4M^2 + 4MN + N + N^2) 2^{N-2},
\end{align*}
which applied to Equation~\eqref{eq:first_dim_PC} yields the claim after some simple manipulations.
\end{proof}

\begin{table}
\[
\begin{array}{c|ccccccc}
& 1 & 2 & 3 & 4 & 5 & 6 & 7 \\\hline
1& 2 & 14 & 132 & 1430 & 16796 & 208012 & 2674440 \\
2& 1 & 3 & 12 & 55 & 273 & 1428 & 7752\\
3& \cdot & 2 & 5 & 16 & 54 & 186 & 689 \\
4& \cdot & \cdot & 4 & 9 & 24 & 70 & 202 \\
5& \cdot & \cdot & \cdot & 8 & 17 & 40 & 102 \\
6& \cdot & \cdot & \cdot & \cdot & 16 & 33 & 72 \\
7& \cdot & \cdot & \cdot & \cdot & \cdot & 32 & 65 \\
8& \cdot & \cdot & \cdot & \cdot & \cdot & \cdot & 64 \\
\end{array}
\]
\caption{The $(r,n)$-th entry is equal to $\dim \mcP\mcC_{r,n}(\beta)$. Note that every entry marked with a $\cdot$ is equal to the entry above it, which equals $\dim \mcP\mcU_n = 2^{n-1}$ (and also holds for the subdiagonal entries; see Equation~\eqref{eq:planar_uniform_dim}).}
\label{table:dim_planar_color}
\end{table}

\begin{ex}
Consider $\mcP\mcC_{3,5}$.
We have
\[
\dim W(\lambda) = 1
\quad (\abs{\lambda} = 5),
\qquad
\dim W(2) = 4,
\qquad
\dim W(11) = 5,
\]
with the half diagrams spanning $W(2)$ and $W(11)$ being
\begin{align*}
W(2) : & \qquad
\begin{tikzpicture}[scale = 0.5,thick, baseline={(0,2ex/2)}] 
\tikzstyle{vertex} = [shape=circle, minimum size=2pt, inner sep=1pt]
\foreach \i in {1,2,3,4,5} {
    \node[vertex] (G-\i) at (\i, 0) [shape=circle, draw] {};
}
\draw[-,color=darkred] (G-2) .. controls +(0., .5) and +(0, .5) .. (G-1) -- ++ (0,1);
\draw[-,color=blue] (G-3) .. controls +(0., .5) and +(0, .5) .. (G-4) .. controls +(0., .5) and +(0, .5) .. (G-5);
\end{tikzpicture}\,,
\qquad\qquad
\begin{tikzpicture}[scale = 0.5,thick, baseline={(0,2ex/2)}] 
\tikzstyle{vertex} = [shape=circle, minimum size=2pt, inner sep=1pt]
\foreach \i in {1,2,3,4,5} {
    \node[vertex] (G-\i) at (\i, 0) [shape=circle, draw] {};
}
\draw[-,color=darkred] (G-5) .. controls +(0., 1.0) and +(0, 1.0) .. (G-1) -- ++ (0,1);
\draw[-,color=blue] (G-2) .. controls +(0., .5) and +(0, .5) .. (G-3) .. controls +(0., .5) and +(0, .5) .. (G-4);
\end{tikzpicture}\,,
\qquad\qquad
\begin{tikzpicture}[scale = 0.5,thick, baseline={(0,2ex/2)}] 
\tikzstyle{vertex} = [shape=circle, minimum size=2pt, inner sep=1pt]
\foreach \i in {1,2,3,4,5} {
    \node[vertex] (G-\i) at (\i, 0) [shape=circle, draw] {};
}
\draw[-,color=darkred] (G-5) .. controls +(0., .5) and +(0, .5) .. (G-4) -- ++ (0,1);
\draw[-,color=blue] (G-1) .. controls +(0., .5) and +(0, .5) .. (G-2) .. controls +(0., .5) and +(0, .5) .. (G-3);
\end{tikzpicture}\,,
\qquad\qquad
\begin{tikzpicture}[scale = 0.5,thick, baseline={(0,2ex/2)}] 
\tikzstyle{vertex} = [shape=circle, minimum size=2pt, inner sep=1pt]
\foreach \i in {1,2,3,4,5} {
    \node[vertex] (G-\i) at (\i, 0) [shape=circle, draw] {};
}
\draw[-,color=darkred] (G-5) .. controls +(0., .5) and +(0, .5) .. (G-4) .. controls +(0., .5) and +(0, .5) .. (G-3) .. controls +(0., .5) and +(0, .5) .. (G-2) .. controls +(0., .5) and +(0, .5) .. (G-1) -- ++ (0,1);
\end{tikzpicture}\,,
\\
W(11) : & \quad
\begin{tikzpicture}[scale = 0.5,thick, baseline={(0,2ex/2)}] 
\tikzstyle{vertex} = [shape=circle, minimum size=2pt, inner sep=1pt]
\foreach \i in {1,2,3,4,5} {
    \node[vertex] (G-\i) at (\i, 0) [shape=circle, draw] {};
}
\draw[-,color=darkred] (G-1) -- ++ (0,1);
\draw[-,color=darkred] (G-2) -- ++ (0,1);
\draw[-,color=blue] (G-3) .. controls +(0., .5) and +(0, .5) .. (G-4) .. controls +(0., .5) and +(0, .5) .. (G-5);
\end{tikzpicture}\,,
\qquad
\begin{tikzpicture}[scale = 0.5,thick, baseline={(0,2ex/2)}] 
\tikzstyle{vertex} = [shape=circle, minimum size=2pt, inner sep=1pt]
\foreach \i in {1,2,3,4,5} {
    \node[vertex] (G-\i) at (\i, 0) [shape=circle, draw] {};
}
\draw[-,color=darkred] (G-1) -- ++ (0,1);
\draw[-,color=darkred] (G-5) -- ++ (0,1);
\draw[-,color=blue] (G-2) .. controls +(0., .5) and +(0, .5) .. (G-3) .. controls +(0., .5) and +(0, .5) .. (G-4);
\end{tikzpicture}\,,
\qquad
\begin{tikzpicture}[scale = 0.5,thick, baseline={(0,2ex/2)}] 
\tikzstyle{vertex} = [shape=circle, minimum size=2pt, inner sep=1pt]
\foreach \i in {1,2,3,4,5} {
    \node[vertex] (G-\i) at (\i, 0) [shape=circle, draw] {};
}
\draw[-,color=darkred] (G-4) -- ++ (0,1);
\draw[-,color=darkred] (G-5) -- ++ (0,1);
\draw[-,color=blue] (G-1) .. controls +(0., .5) and +(0, .5) .. (G-2) .. controls +(0., .5) and +(0, .5) .. (G-3);
\end{tikzpicture}\,,
\qquad
\begin{tikzpicture}[scale = 0.5,thick, baseline={(0,2ex/2)}] 
\tikzstyle{vertex} = [shape=circle, minimum size=2pt, inner sep=1pt]
\foreach \i in {1,2,3,4,5} {
    \node[vertex] (G-\i) at (\i, 0) [shape=circle, draw] {};
}
\draw[-,color=darkred] (G-5) -- ++ (0,1);
\draw[-,color=darkred] (G-4) .. controls +(0., .5) and +(0, .5) .. (G-3) .. controls +(0., .5) and +(0, .5) .. (G-2) .. controls +(0., .5) and +(0, .5) .. (G-1) -- ++ (0,1);
\end{tikzpicture}\,,
\qquad
\begin{tikzpicture}[scale = 0.5,thick, baseline={(0,2ex/2)}] 
\tikzstyle{vertex} = [shape=circle, minimum size=2pt, inner sep=1pt]
\foreach \i in {1,2,3,4,5} {
    \node[vertex] (G-\i) at (\i, 0) [shape=circle, draw] {};
}
\draw[-,color=darkred] (G-1) -- ++ (0,1);
\draw[-,color=darkred] (G-5) .. controls +(0., .5) and +(0, .5) .. (G-4) .. controls +(0., .5) and +(0, .5) .. (G-3) .. controls +(0., .5) and +(0, .5) .. (G-2) -- ++ (0,1);
\end{tikzpicture}\,.
\end{align*}
We see that there are $13 = 2^4 - 1 - 2$ compositions of $5$ with parts in $\{1,2,3\}$, and we verify that
\[
\dim \mcP\mcC_{3,5} = 54 = 13 + 1 \cdot 4^2 + 1 \cdot 5^2 = \sum_{\lambda \in \Lambda} \bigl( \dim W(\lambda) \bigr)^2.
\]
\end{ex}

Additional examples of the dimensions of $\mcP\mcC_{r,n}(\beta)$ is given in Table~\ref{table:dim_planar_color}, which was computed by directly counting the number of diagrams in the basis using \textsc{SageMath}.
Furthermore, we see that the sequence in Lemma~\ref{lemma:pc_dim} is given by
\[
a_j = (9j^2 + 17j + 6) 2^{j-3}
\]
for $j > 0$ with $a_0 = 1$.
Some initial terms are
\[
(1, 8, 38, 138, 436, 1264, 3456, 9056, 22976, \ldots).
\]

\subsection{Planar quasi-partition algebra}

By~\cite[Lemma~2.2]{DO14}, the basis of the quasi-partition algebra $\mcQ\mcP_n(\beta)$ is given inside of $\mcP_n(\beta-1)$ by looking at certain refinements of the indexing diagram.
From this and~\cite[Cor.~2.7]{DO14}, we can restrict to the set of planar diagrams without singletons and this can be constructed as a subalgebra of $\mcP\mcP_n(\beta-1)$.
We call this the \defn{planar quasi-partition alagebra} and denote it by $\mcP\mcQ\mcP_n(\beta)$.
By~\cite[Thm.~5.12]{BE18} (see also~\cite[A099251]{OEIS}), we have
\[
\dim \mcP\mcQ\mcP_n(\beta) = R_N = \frac{1}{N+1} \sum_{k=1}^{\lfloor N/2 \rfloor} \binom{N+1}{k} \binom{N-k-1}{k-1} = \frac{N-1}{N+1} \left(2 R_{N-1} + 3 R_{N-2} \right),
\]
where $N = 2n$ and $R_N$ are the Riordan numbers~\cite{Riordan75} (see also~\cite[A005043]{OEIS}), with $R_0 = 1$ and $R_1 = 1$.
This can be described as the number of Motzkin paths from $(0,0)$ to $(n,0)$ that do not have any horizontal steps on the $y=0$ line.

Let us consider the algebra given in~\cite[Sec.~5]{BE18}, which we call the \defn{tangle algebra} and denote it by $\mcT_n(\gamma)$; here we add the parameter $\gamma$ that counts the number of interior components in the product.
Here we consider $\mcT_n(\gamma)$ as an abstract algebra defined by a basis (in~\cite[Sec.~5]{BE18}, these are the normalized $3$-tangles) and relations.
As noted in~\cite{BE18}, the tangle algebra has the same dimension as $\mcP\mcQ\mcP_n(\beta)$.
Furthermore, the tangle algebra for $\gamma = 1$ has a Schur--Weyl duality property~\cite[Thm.~5.7]{BE18} (see also~\cite[Rem.~5.13]{BE18}) similar to the cases considered here.
This leads to the following conjecture.

\begin{conj}
\label{conj:tangle_pqp}
Let $\field$ be a field with $\abs{\field} > n + 1$.
Let $\gamma \neq 0$.
There exists a $\beta \in \field$ such that
\[
\mcT_n(\gamma) \iso \mcP\mcQ\mcP_n(\beta).
\]
\end{conj}

By a brute-force computation, we can see that
\[
\begin{tikzpicture}[scale = 0.5,thick, baseline={(0,-1ex/2)}] 
\tikzstyle{vertex} = [shape=circle, minimum size=2pt, inner sep=1pt]
\foreach \i in {1,2} {
    \node[vertex] (G\i) at (\i, 1) [shape=circle, draw] {};
    \node[vertex] (G-\i) at (\i, -1) [shape=circle, draw] {};
}
\draw[-] (G-1) -- (G1);
\draw[-] (G-2) -- (G2);
\end{tikzpicture}
\, \mapsto \,
\begin{tikzpicture}[scale = 0.5,thick, baseline={(0,-1ex/2)}] 
\tikzstyle{vertex} = [shape=circle, minimum size=2pt, inner sep=1pt, fill=OCUenji]
\foreach \i in {1,2} {
    \node[vertex] (G\i) at (\i, 1) [shape=circle, draw] {};
    \node[vertex] (G-\i) at (\i, -1) [shape=circle, draw] {};
}
\draw[-] (G-1) -- (G1);
\draw[-] (G-2) -- (G2);
\end{tikzpicture}\,,
\qquad
\begin{tikzpicture}[scale = 0.5,thick, baseline={(0,-1ex/2)}] 
\tikzstyle{vertex} = [shape=circle, minimum size=2pt, inner sep=1pt]
\foreach \i in {1,2} {
    \node[vertex] (G\i) at (\i, 1) [shape=circle, draw] {};
    \node[vertex] (G-\i) at (\i, -1) [shape=circle, draw] {};
}
\draw[-] (G1) .. controls +(0, -.5) and +(0, -.5) .. (G2);
\draw[-] (G-1) .. controls +(0, .5) and +(0, .5) .. (G-2);
\end{tikzpicture}
\, \mapsto \,
\frac{\gamma}{\beta - 1}\,
\begin{tikzpicture}[scale = 0.5,thick, baseline={(0,-1ex/2)}] 
\tikzstyle{vertex} = [shape=circle, minimum size=2pt, inner sep=1pt, fill=OCUenji]
\foreach \i in {1,2} {
    \node[vertex] (G\i) at (\i, 1) [shape=circle, draw] {};
    \node[vertex] (G-\i) at (\i, -1) [shape=circle, draw] {};
}
\draw[-] (G1) -- ++(.5, -.5) -- (G2);
\draw[-] (G-1) -- ++(.5, .5) -- (G-2);
\end{tikzpicture}\,,
\qquad
\begin{tikzpicture}[scale = 0.5,thick, baseline={(0,-1ex/2)}] 
\tikzstyle{vertex} = [shape=circle, minimum size=2pt, inner sep=1pt]
\foreach \i in {1,2} {
    \node[vertex] (G\i) at (\i, 1) [shape=circle, draw] {};
    \node[vertex] (G-\i) at (\i, -1) [shape=circle, draw] {};
}
\draw[-] (G-1) -- (G1) .. controls +(0, -.5) and +(0, -.5) .. (G2) -- (G-2) .. controls +(0, .5) and +(0, .5) .. (G-1);
\end{tikzpicture}
\, \mapsto \,
- \frac{2 \beta}{\beta - 2}\,
\begin{tikzpicture}[scale = 0.5,thick, baseline={(0,-1ex/2)}] 
\tikzstyle{vertex} = [shape=circle, minimum size=2pt, inner sep=1pt, fill=OCUenji]
\foreach \i in {1,2} {
    \node[vertex] (G\i) at (\i, 1) [shape=circle, draw] {};
    \node[vertex] (G-\i) at (\i, -1) [shape=circle, draw] {};
}
\draw[-] (G1) -- ++(.5, -.5) -- (G2);
\draw[-] (G-1) -- ++(.5, .5) -- (G-2);
\draw[-] (1.5,-.5) -- (1.5,.5);
\end{tikzpicture}
+ \frac{2}{\beta - 2}\,
\begin{tikzpicture}[scale = 0.5,thick, baseline={(0,-1ex/2)}] 
\tikzstyle{vertex} = [shape=circle, minimum size=2pt, inner sep=1pt, fill=OCUenji]
\foreach \i in {1,2} {
    \node[vertex] (G\i) at (\i, 1) [shape=circle, draw] {};
    \node[vertex] (G-\i) at (\i, -1) [shape=circle, draw] {};
}
\draw[-] (G1) -- ++(.5, -.5) -- (G2);
\draw[-] (G-1) -- ++(.5, .5) -- (G-2);
\end{tikzpicture}\,,
\]
defines an isomorphism $\mcP\mcQ\mcP_2(\beta) \to \mcT_2(\gamma)$.
It would be interesting to determine for which values $\gamma$ and $\beta$ we have an isomorphism.
(Clearly when $\beta = 1, 2$, the above map is not valid.)

We define the \defn{Riordan triangle numbers} recursively by
\begin{equation}
\label{eq:riordan_triangle_recursion}
R_{n,\lambda} = \begin{cases}
R_{n-1,\lambda+1} + R_{n-1,\lambda} + R_{n-,\lambda-1} & \text{if } \lambda > 0, \\
R_{n-1,1} & \text{if } \lambda = 0,
\end{cases}
\end{equation}
with $R_{n,n} = 1$ and $R_{1,0} = 0$~\cite{Bernhart97,MRSV97}.

\begin{thm}
\label{thm:quasi_partition_cellular}
The planar quasi-partition algebra $\mcP\mcQ\mcP_n(\beta)$ is a cellular algebra of dimension $R_{2n}$ with $\Lambda = \{0, 1, \dotsc, n\}$.
Furthermore, for any $\lambda \in \Lambda$, we have
\[
\dim W(\lambda) = R_{n,\lambda}.
\]
Moreover, we have
\[
R_{2n} = \sum_{\lambda=0}^n R_{n,\lambda}^2.
\]
\end{thm}

\begin{proof}
From Theorem~\ref{thm:quasi_cellular} and Theorem~\ref{thm:planar_cellular}, the planar quasi-partition algebra is cellular.
We note that any singleton in a half diagram must be a propagating block.
We can also move all of the propagating blocks to the front and have size $1$.
This proves the indexing set $\Lambda = \{0, 1, \dotsc, n\}$.
The last claim is simply Equation~\eqref{eq:dim_formula}.
Thus, we only need to prove the dimension of the cell modules.

We show that the half diagrams satisfy the same recurrence relation as the Riordan triangle numbers.
Let $\rho$ be a half diagram, and suppose $n \in \rho_1$.
If $\abs{\rho_1} > 1$, then construct a new half diagram by having $\rho_1 \setminus \{n\}$ be a defect block and keeping the other parts the same.
Thus the number of defects has increased by one if and only if $\rho_1$ is a defect (otherwise the number of defects does not change).
This is bijective as we simply take the rightmost defect block in the half diagram of $[n-1]$ to reconstruct $\rho$.
Lastly, if $\abs{\rho_1} = 1$, then it necessarily must be a defect.
We simply remove $\rho_1$ from the half diagram to form the new half diagram.
This decreases the number of defects and is clearly bijective.
\end{proof}

For some other interesting appearances of the Riordan triangle numbers, see~\cite{KLO17,OS19} (there is an unfortunate misprint in the definition of the Riordan triangle numbers in~\cite{OS19}).

It is a straightforward exercise to show that Theorem~\ref{thm:quasi_partition_cellular} also holds for the tangle algebra.
Hence, Conjecture~\ref{conj:tangle_pqp} is true when both $\mcT_n(\gamma)$ and $\mcP\mcQ\mcP_n(\beta)$ are semisimple as they have the same dimension and all simple modules (which are the cell modules) have the same dimensions.
Therefore, we have an abstract isomorphism by the Artin--Wedderburn theorem.
The author thanks Hyohe Miyachi for noting this.
It would be good to have an explicit (combinatorial) isomorphism.

It would also be interesting to show they are both satisfy the same Schur--Weyl duality directly (and could lead to such an isomorphism).
To this, we believe there is a minor misprint in~\cite{BE18}, as the module should be $V(2)$ instead of the adjoint representation.
In particular, the dimension of the irreducible module $V(\lambda)$, which equals the cell module $W(\lambda)$ in this case by semisimplicity, is the multiplicity of $V(2\lambda)$ in the decomposition of $V(2)^{\otimes n}$ (as $U(\fsl_2)$-modules).
This can be seen by counting the multiplicities inductively on $n$, where the Pieri rule yields the Riordian triangle number recursion relation~\eqref{eq:riordan_triangle_recursion}.
Alternatively, if it was built using the adjoint representation, then~\cite{BH14} implies that the tangle algebra is isomorphic to the Motzkin algebra, but they (and their cell modules) have different dimensions.

\begin{conj}
The tangle algebra $\mcT_n(\gamma)$ and the quasi-partition algebra $\mcP\mcQ\mcP_n(\beta)$ satisfy Schur--Weyl duality with $U_q(\fsl_2)$ for the module $V(2)^{\otimes n}$ for some $\gamma$ and $\beta$ when $\field$ is a field of characteristic $0$.
\end{conj}

Given that the half (integer) quasi-partition algebra exists, we can also do the same for the planar version.
This would lead to an algebra $\mcP\mcQ\mcP_{n-1/2}(\beta)$ whose dimension is $R_{2n-1}$.

\section{Alternative perspectives and generalizations}

In this section, we discuss just a few generalizations of the partition algebra, although there are indubitably many more than we discuss here.
We also give an alternative perspective using tensor categories, which are often strongly linked to combinatorics (which can be seen in, \textit{e.g.},~\cite{BMT21}).

\subsection{Blob algebra}
\label{sec:blob_algebra}

The blob algebra $\mcB_n(\beta, \gamma, \delta)$, for parameters $\beta,\gamma,\delta \in \field$, defined by Martin and Saleur~\cite{MS94} can be considered as the type $B$ analog the Temperley--Lieb algebra, where we can put idempotent blobs on strands that can escape out the left boundary of the diagram.
Alternative, if we unfold the diagram (along the right side), then these are the strands that are not nested in the noncrossing perfect matching.
Multiplication is given as for the Temperley--Lieb algebra except loops with a blob contribute $\delta$ instead of $\beta$ and the blobs are idempotent (which is resolved before loops are removed): when we combine two blobs together, we have a blob remaining and multiply by a factor of $\gamma$.
The blob algebra has also been well-studied (see, \textit{e.g.},~\cite{ILZ18} and references therein) and is known to be cellular~\cite{GL03} (along with some generalizations, such as in~\cite{LRH20} using the version of Martin and Woodcock~\cite{MW00}).

The classical blob algebra also fits into our framework, but now
\[
\Lambda = \{n-2k, \overline{n-2k} \mid k \in \langle n \rangle\}
\]
under the ordering $\overline{1} < 1 < \overline{2} <  2 < \cdots$.
For $\lambda \in \Lambda$ with $\lambda = k$ or $\overline{k}$, the number of defects is equal to $k$.
Furthermore, the barred values indicate that the leftmost defect has a blob on it and unbarred entries have no blobs on the defects.
Otherwise, $M(\lambda)$ is the expected set of blobbed half diagrams.
We can see this is a cellular algebra since blobs are idempotent, and so we cannot remove a blob from a strand/defect once it has been added.

\subsection{Other Schur--Weyl duality algebras}

We briefly remark on some other variations of the partition algebra that have appeared coming from a Schur--Weyl duality.

The first is the rook partition algebra $\mcR\mcP_n(\beta)$ introduced by Grood~\cite{Grood06} that comes from Schur--Weyl duality involving the $\GL_m(\CC)$ module $V = V(1) \oplus V(0)$ restricted to the corresponding $\Sym_m$ action.
This version involves the usual diagrams except we color singletons by two different colors.
Alternatively, we can color the nodes by two different colors, call them red and green, and if a node is colored red, then it must be a singleton.
Since the only difference is coloring singletons, the proof that it is a cellular algebra is the same as for the classical partition algebra.

\begin{prop}
The rook partition algebra is a cellular algebra with the same cell datum as $\mcP_n(\beta)$ except $M(\lambda)$ consists of all half diagrams with singletons colored one of two colors.
\end{prop}

\begin{cor}
Let $\lambda \in \Lambda$ be a partition of $k$.
Then we have
\[
\dim W(\lambda) = f_{\lambda} \sum_{i=0}^{n-k} \binom{n}{i} \sum_{j=k}^{n-i} \binom{j}{k} \stirling{n-i}{j}.
\]
\end{cor}

\begin{proof}
This comes from choosing $i$ nodes to first color red, and the rest is just the usual partition algebra formula.
\end{proof}

Grood also showed~\cite{Grood06} that $\dim \mcR\mcP_n(\beta) = B_{2n+1}$ through an indirect combinatorial argument, but we can give a more straightforward argument.
The red colored nodes are simply one special block, which how we mark it as special is say it contains an extra node $\{0\}$.
This perspective gives us an alternative formula for the cell module dimensions as
\[
\dim W(\lambda) = f_{\lambda} \sum_{j=k}^n \binom{n}{j} \stirling{j}{k} B_{n-j+1}.
\]

Ly in~\cite{Ly19} studied a Schur--Weyl type duality using the supercharacter theory of $U_m(\mathbb{F}_q)$ of upper triangular matrices from~\cite{Thiem10}, which is used to approximate its ``wild'' type representation theory.
However, it is not clear how to apply the techniques of this paper to that construction.

\begin{problem}
Determine if the centralizer algebra in~\cite{Ly19} is cellular.
\end{problem}

Lastly, there is the \defn{walled Brauer algebra} $\mcB_{n,k}(\beta)$ coming from Schur--Weyl duality with $\GL_m(\CC)$ on $\natrepr^{\otimes n} \otimes (\natrepr^*)^{\otimes k}$ that was initially studied in~\cite{Turaev89,Koike89,BCHLLS94}.
This is a subalgebra of $\mcB_{n+k}(\beta)$, where we do not have any propagating strands $\{a, \overline{b}\}$ or $\{b, \overline{a}\}$ for $1 \leq a < n + \frac{1}{2} < b \leq n+k$.
That is, there is a ``wall'' between positions $k$ and $k+1$ that does not allow propagating blocks to pass it (although caps and cups are fine).
The walled Brauer algebra was shown to be cellular in~\cite{CDVDM08}, where they also classified when it is semisimple.
This could also be seen from Proposition~\ref{prop:subcellular}, where we are restricting $\Sym_{n+k}$ that dictates the decomposition~\eqref{eq:partition_decomposition} to the Young subgroup $\Sym_n \times \Sym_k$.
Hence, $\Lambda$ now given by a pair of partitions $(\lambda, \mu)$ of sizes at most $k$ and $n$, respectively (\textit{cf.}~\cite{Halverson96}).
Generalizations have also been considered, such as~\cite{RS15,Sartori14}.

\subsection{Diagram algebras as categories}

Another perspective on the diagram algebras is to view them as (graded) monoidal categories, where the algebra is isomorphic to the Grothendieck ring of the category.
This is known as categorification, and it can bring out new properties both for the category and Grothendieck ring.

One example the Temperley--Lieb category $\mathsf{TL}$, which was first introduced by Graham and Lehrer~\cite{GL98}, with objects $\ZZ_{>0}$ and morphisms $n \to m$ corresponding to Temperley--Lieb diagrams from $[n] \to [m']$ (with $m$ not necessarily equal to $n$).
The composition of diagrams is the composition of morphisms (as per our multiplication convention).
If we restrict to the subcategory corresponding to the object $n$, then the Grothendieck ring is isomorphic to $\mcT\mcL_n(\beta)$.
The category $\mathsf{TL}$ and its representations have been studied with expected applications to conformal field theory; see, \textit{e.g.},~\cite{BSA18II} and references therein.
In~\cite{KMY19}, an interpolation category was constructed between $\mathsf{TL}$ and the corresponding categorification of the Brauer algebra.

A different categorification of $\mcT\mcL_n(1)$ was given by Bernstein, Frenkel, and Khovanov~\cite{BFK99}, although it is just the Karubi envelope of the additive closure the previous construction.
Their construction was later extended to $\mcT\mcL_n(\beta)$ and other types by Stroppel~\cite{Stroppel05}.
Recently, a two parameter analog of $\mathsf{TL}$ was given in~\cite{KS21} as a method to categorify Chebyshev polynomials of the second kind (yet another appearance of these polynomials).

A related but different construction for the invariant spaces of the $U_q(\fsl_2)$-action on $\natrepr^{\otimes n}$ and generalizations was given diagrammatically by Kuperberg webs and spiders~\cite{Kuperberg96}.
These have seen some attention, such as in~\cite{Tymoczko12,Scherer21} with many open questions remaining, such as a basis for $\fsl_m$ for $m \geq 4$.

\subsection{Knots, braids, and ties}

One of the important results from the Temperley--Lieb algebra has been the construction of invariants of knots and links.
By realizing $\mcT\mcL_n\bigl((q+q^{-1})\bigr)$ as a quotient of the braid group algebra and using Kauffman's bracket polynomial~\cite{Kauffman90}, we are naturally led to a Markov trace of connecting the top of a diagram to its bottom.
This is the construction of the famous Jones polynomial~\cite{Jones85}.
Some surveys on the Jones polynomial are~\cite{Abramsky08,FKP14,KL22}.

The Ariki--Koike algebra~\cite{AK94} is a deformation of the group algebra $\field [G(r,1,m)]$ analogous to the usual (Iwahori--)Hecke algebra deformation, but there is another deformation known as the Yokonuma--Hecke algebra $Y_{m,r}(q)$~\cite{Yokonuma67} that originally arose in the context of Chevalley groups.
In particular, $Y_{m,r}(q)$ is the centralizer algebra of the permutation representation of $\GL_m(\mathbb{F}_{p^m})$, where $q^2 = p^m$ for a prime $p$, with respect to a maximal unipotent subgroup.
The Yokonuma--Hecke algebra $Y_{m,r}(q)$ is more naturally attuned to knot theory as it is the quotient of the framed braid group $\ZZ \wr \mathbf{B}_m$ and the modular framed braid group $\ZZ_r \wr \mathbf{B}_m$, with generalizations to other Coxeter systems given by Marin~\cite{Marin18}.
From the braid group construction, an invariant on framed knots was constructed~\cite{JL07,JL07II,JL13} by using the Markov trace given by Juyumaya~\cite{Juyumaya04}.
It also can be extended to define an invariant on classical and singular knots~\cite{JL09,JL11}.
Subsequent results have made further connections between $Y_{m,r}(q)$ and knot theory, such as to the HOMFLYPT polynomial~\cite{PW18}.

Based upon a new presentation for $Y_{m,r}(q)$ introduced by Juyumaya~\cite{Juyumaya98}, a new algebra coined the braids and ties algebra, or bt-algebra, $\mcE_m(q)$ was introduced by Aicardi and Juyumaya~\cite{AJ00,Juyumaya99}.
A basis for $\mcE_m(q)$ was given by Ryom-Hansen~\cite{RH11}, showing its dimension is $m! B_m$.
However, Banjo~\cite{Banjo13} has proven that $\mathcal{E}_m(1)$ is isomorphic to the small ramified partition algebra from~\cite{Martin11}, which allows us to construct a diagrammatic basis for $\mcE_m(q)$.
A Markov trace on $\mcE_m(q)$ was constructed in~\cite{AJ16}, allowing it to be used to construct knot invariants.
There is also a two parameter version of the bt-algebra~\cite{AJ20}, which arose from extending the Markov trace to two parameters~\cite{CJKL20} to construct knot invariants such as the HOMFLYPT polynomial.
The representation theory of $\mcE_m(q)$ has been studied in papers such as~\cite{AJ00,Banjo13,RH11}.
Both $Y_{m,d}(q)$ and $\mcE_m(q)$ were shown to be cellular in~\cite{ERH18} with a combinatorial description than can be translated into the diagrammatic language.

Another way to construct $\mcE_m(q)$ is using the monoid of tied braids, coming from tied links introduced in~\cite{AJ16II} with analogous polynomial invariants~\cite{AJ18,AJ21}.
Therefore, we have a natural description of $\mcE_m(q)$ in terms of diagrams.
This diagrammatic description was later extended to the Brauer algebra and other submonoids of the partition monoid~\cite{AAJ23}, which are referred to as ramified monoids and includes other Coxeter/Artin-type monoids.
We call the corresponding monoid algebras ramified algebras.
The inverse symmetric monoid and planar monoid cases were examined in detail in the recent work~\cite{AAJ22}.

\begin{problem}
Determine if the ramified algebras are cellular.
\end{problem}

\section{Wreath products}
\label{sec:wreath}

Recall that $G(r,1,m) \iso \ZZ_r \wr \Sym_m$, whose group algebra has a cellular basis~\cite{GL96} roughly speaking by taking ``product'' of the cellular basis of $\field [\ZZ_r]$ and $\field [\Sym_m]$.
The cellularity of $\field [\ZZ_r]$ when $\zeta_r \in \field$ (alternatively $\field$ splits $x^r - 1$) follows from a special case of the the Ariki--Koike algebras from~\cite{GL96}, which is done over $\ZZ[q,u_1,\dotsc,u_r]$ but the group algebra is under the specialization $q = 1$ and $u_k = \zeta_r^k$.
A more direct construction was done in, \textit{e.g.},~\cite[Sec.~4]{RX04} (see also~\cite[Ex.~4.14]{ZC06}).
There is also the cyclotomic blob (resp.\ Temperley--Lieb) algebra introduced in~\cite{ZC06} (resp.~\cite{RX04}), which can be constructed as the blob  (resp.\ Temperley--Lieb) algebra with each strand carrying a copy of $\ZZ_r$, where it was shown to be cellular (again, assuming $\field$ splits $x^r - 1$).
Furthermore, the cellular basis of the cyclotomic blob/Termperley--Lieb algebra has the same ``product'' of cellular bases structure.
In this section, we generalize this construction as a method to produce new cellular algebras by taking a wreath product of an arbitrary cellular algebra with certain subalgebras of the partition algebra.

Let $\mcA$ be any finite dimensional $\field$-algebra with basis $B$.
Let $\mcS$ be any subalgebra of the partition algebra such that blocks have size at most $2$.
Define the wreath product $\mcA \wr \mcS$ as the $\field$-span of diagrams of $\mcS$ with an element of $B$ attached to each block.
The multiplication is the natural concatenation of diagrams: where we take the product in $\mcS$, then multiply the elements of $\mcA$ along any strand and expanding this as a linear combination of basis elements in the natural way before removing cycles.
This has a natural involution $\iota_{\wr}$ induced from the involution $\iota$ of $\mcS$.
We call this the \defn{cellular wreath product} of the base $\mcS$ by the algebra $\mcA$ and denote this by $\mcA \wr \mcS$.

\begin{thm}
\label{thm:wreath_product}
Let $\mcA$ be a cellular algebra with cell datum $(\widetilde{\Lambda}, \widetilde{\iota}, \widetilde{M}, \widetilde{C})$.
Let $\mcS$ be any subalgebra of the partition algebra such that blocks have size at most $2$ with cell datum $(\Lambda, \iota, M, C)$.
Then the wreath product $\mcA \wr \mcS$ is a cellular algebra.
If $\mcS = \mcR\mcB_n(\beta, \gamma)$ (resp.~$\mcM_n(\beta, \gamma)$), then the cell datum $(\Lambda_{\wr}, \iota_{\wr}, M_{\wr}, C_{\wr})$ is given by
\begin{itemize}
\item $\Lambda_{\wr} = \{(\lambda, L) \mid \lambda \in \Lambda, L \in \langle \widetilde{\Lambda}, D(\lambda) \rangle \}$, where $D(\lambda)$ is the number of defects corresponding to $\lambda$ and $\langle \widetilde{\Lambda}, \ell \rangle$ is the number of multisets (resp.\ sequences) of size $\ell$ with elements in $\widetilde{\Lambda}$, under the natural lexicographic order;
\item $\iota_{\wr}$ is the natural involution induced from $\iota$ on $\mcS$;
\item $M_{\wr}(\lambda, L) = M(\lambda) \times \widetilde{M}(L) \times (\bigsqcup \widetilde{M})^{N(\lambda)}$, where $N(\lambda)$ equals the number of non-defect blocks corresponding to~$\lambda$, $\widetilde{M}(L) := \prod_{i=1}^{D(\lambda)} \widetilde{M}(L_i)$, and $\bigsqcup \widetilde{M} := \bigsqcup_{\lambda \in \widetilde{\Lambda}} \widetilde{M}(\lambda)$;
\item $C_{\wr}$ is formed by writing the cellular basis element $C_{ST}^{\lambda}$ in terms of the natural diagram basis and attaching the corresponding element in $C_{\widetilde{S}_i\widetilde{T}_i}^{\widetilde{\lambda}_i}$ to the $i$-th strand for all such $i$ determined by (inverse) RSK.
\end{itemize}
Otherwise it is given by the appropriate restriction.
\end{thm}

\begin{proof}
Consider some $\lambda_{\wr} \in \Lambda_{\wr}$ and $S_{\wr}, T_{\wr} \in M_{\wr}(\lambda, L)$.
We have $\iota_{\wr}(C^{\lambda}_{S_{\wr}T_{\wr}}) = C^{\lambda}_{T_{\wr}S_{\wr}}$ by construction.
For every $\lambda \in \Lambda$, $S,T \in M(\lambda)$, and $a \in \mcA$, we have
\[
a C^{\lambda_{\wr}}_{S_{\wr}T_{\wr}} = \sum_{U,V \in M(\lambda, L)} r_a(U,S_{\wr},V) C^{\lambda_{\wr}}_{UV} + \mcA^{<\lambda_{\wr}},
\]
since we attach elements of the cellular basis of $\mcA$ to each strand and that multiplication in $\mcS$ cannot increase the number of propagating blocks.
It remains to show that we must have $V = T_{\wr}$ and $r_a(U,S_{\wr},V)$ does not depend on $T$.
Indeed, this follows from the fact that we can think of the (partial) permutation as a (partial) mapping from $\bigsqcup \widetilde{M} \to \bigsqcup \widetilde{M}$, the multiplication properties of cellular basis elements of $\mcA$ to each strand, and properties of the Kazhdan--Lusztig basis (see, \textit{e.g.},~\cite[Ex.~1.2]{GL96}).
The restriction is Proposition~\ref{prop:subcellular}.
\end{proof}

Our proof is essentially a mild generalization of the proof of~\cite[Thm.~5.5]{GL96}.
In fact, this suggests that there should be a good set of commutative elements that would play the role of the Jucys--Murphy elements in a Hecke algebra analog (\textit{cf.}~\cite[Sec.~2.2]{Mathas04}).

\begin{problem}
Show this can be extended to the Hecke algebra action on the diagrams of $\mcS$ (instead of $\field [\Sym_n]$).
Furthermore, construct a set of elements $L_2, \dotsc, L_N$ such that:
\begin{itemize}
    \item they generate an abelian subalgebra,
    \item there is a basis given by monomials in $L_i$ times a Hecke algera basis element $T_w$, and
    \item the center contains all symmetric polynomials in $L_2, \dotsc, L_N$.
\end{itemize}
\end{problem}

Towards this, we note that the generalized rook monoid~\cite{Steinberg08} is constructed as the wreath product $\ZZ_r \wr R_n$, where $R_n$ is the rook monoid.
Thus, our cellular wreath product says the corresponding monoid algebra is a cellular algebra by using the rook algebra $\mcR_n(\beta)$, and Jucys--Murphy elements (for $\beta = 1$) were recently constructed in~\cite{MS22}.

For the product to be well-defined, it is necessary that the blocks have the same size.
We have used the particular case that each block has size $2$, but this could be extended for the block partition algebra by associating to each block of size $k$ a bundle of $k$ (ordered) strands.
In this case, we get a natural Hecke algebra analog by taking the corresponding product of (type $A$) Hecke algebras, analogous to how we have used a product of symmetric group algebras (equivalently the group algebra of a product of symmetric groups).

Let us remark on two other constructions based on wreath products.
The first is the coloring of the partition algebras by a finite group $G$ in the work by Bloss~\cite{Bloss03} to understand Schur--Weyl duality with the wreath product $G \wr \Sym_m$.
When we restrict to subalgebras of the rook Brauer algebra, we obtain our construction.
The second is the ramified partition algebra of Martin and Elgamal~\cite{ME04}, which was further studied in~\cite{Martin11} using the symmetric group and the partition algebra.
However, those constructions are distinct from $\field [\Sym_n] \wr \mcP_n(\beta)$ since this has the same dimension as $\field [\Sym_n] \otimes \mcP_n(\beta)$ but they are not isomorphic for $n  > 1$ since there are more irreducible representations of $\field [\Sym_n] \wr \mcP_n(\beta)$ (\textit{e.g.}~$12$ for $n = 2$) than $\field [\Sym_n] \otimes \mcP_n(\beta)$ (resp. $6$).
Moreover, the algebra $\mcP_n^{\ltimes}(\beta)$ in~\cite{Martin11} clearly has smaller dimension than $\field [\Sym_n] \otimes \mcP_n(\beta)$.

It is clear that the wreath product construction extends to graded cellular algebras defined in~\cite{HM10}.
It is expected that the wreath product construction will extend to generalizations of cellular algebras, such as affine cellular algebras~\cite{KX12}, the recently defined skew graded cellular algebras~\cite{HMR21}.
It should also work by replacing the underlying diagram algebra; for example, the base algebra could use the blob algebra $\mcB_n(\beta, \gamma, \delta)$, generalizations of the Temperley--Lieb algebra (see, \textit{e.g.},~\cite{BCF22} and references therein), the higher genus diagram algebras~\cite{TV21}, the BMW algebra~\cite{BW89,Murakami87} or the associated tangle algebra~\cite{FG95}, or quiver Hecke algebras (also known as Khovanov--Lauda--Rouquier (KLR) algebras~\cite{KL09,Rouquier08})~\cite{HM10}.

One important algebra could be the cellular wreath product of the symmetric group with itself.
Indeed, the cell modules for this are given by composing a $\Sym_n$-representation with a $\Sym_k$-representation.
This would be a restriction of the corresponding $\GL_n$-representation with a $\GL_k$-representation that defines the plethysm $s_{\lambda}[s_{\mu}]$.
This leads to the following problem.

\begin{problem}
\label{prob:plethysm}
Determine the relationship between the representation theory of $\field [\Sym_n] \wr \field [\Sym_k]$ and the plethysm coefficients $a_{\lambda\mu}^{\nu}$ given by $s_{\lambda}[s_{\mu}] = \sum_{\nu} a_{\lambda\mu}^{\nu} s_{\nu}$.
\end{problem}

We refer the reader to~\cite{COSSZ22} for some recent information on plethysm coefficients.

\bibliographystyle{alpha}
\bibliography{cellular}

\end{document}